\documentclass[12pt,CJK]{article}
\usepackage{amssymb}
\usepackage{stmaryrd}
\usepackage{geometry}
\usepackage{amsfonts,amsmath,amssymb,amsthm,mathrsfs}
\usepackage{latexsym}
\usepackage{graphicx}
\usepackage{indentfirst}
\usepackage{color}
\usepackage{float} 
\usepackage{subfigure} 
\usepackage{fancyhdr}
\usepackage[colorlinks,linktocpage,linkcolor=blue,citecolor=red]{hyperref}
\usepackage{titlesec}

\usepackage{authblk}

\usepackage{pgf}
\usepackage{tikz}
\usepackage[utf8]{inputenc}
\usetikzlibrary{arrows,automata}
\usetikzlibrary{positioning}

\tikzset{
    state/.style={
           rectangle,
           rounded corners,
           draw=black, very thick,
           minimum height=2em,
           inner sep=2pt,
           text centered,
           },
}

\tikzset{global scale/.style={
    scale=#1,
    every node/.append style={scale=#1}
  }
}

\definecolor{darkgreen}{rgb}{0.00,0.50,0.25}
\definecolor{purple}{rgb}{0.54,0.17,0.89}
\definecolor{purple1}{rgb}{1.00,0.00,1.00}

\titleformat*{\subsection}{\bfseries}

\theoremstyle{plain}                       
\newtheorem{lemma}{Lemma}[section]
\newtheorem{theorem}[lemma]{Theorem}
\newtheorem{corollary}[lemma]{Corollary}
\newtheorem{remark}[lemma]{Remark}
\newtheorem{definition}[lemma]{Definition}
\newtheorem{proposition}[lemma]{Proposition}

\theoremstyle{remark}

\numberwithin{equation}{section}

\def\Xint#1{\mathchoice
  {\XXint\displaystyle\textstyle{#1}}%
  {\XXint\textstyle\scriptstyle{#1}}%
  {\XXint\scriptstyle\scriptscriptstyle{#1}}%
  {\XXint\scriptscriptstyle\scriptscriptstyle{#1}}%
  \!\int}
\def\XXint#1#2#3{{\setbox0=\hbox{$#1{#2#3}{\int}$}
  \vcenter{\hbox{$#2#3$}}\kern-.5\wd0}}

\def\dashint{\Xint-}

\begin{document}

\allowdisplaybreaks

\title{\textbf{Fundamental solution expansions in stochastic homogenization of parabolic systems with time-dependent coefficients}}

\author{Junxin Pan}

\author{Li Wang\thanks{Email:li$\_$wang@lzu.edu.cn.}}

\author{Qiang Xu\thanks{Email: xuq@lzu.edu.cn.}}

\affil{School of Mathematics and Statistics, Lanzhou University, Lanzhou, 730000, China.}

\maketitle

\begin{abstract}
This paper studies fundamental solutions of uniformly parabolic systems with stationary random coefficients depending on space and time. Under ergodicity and local regularity, we establish quenched Gaussian bounds for the kernel and its spatial derivatives, with stationary random prefactors. A spectral gap
assumption gives moments of every finite order for these prefactors and hence
annealed Gaussian bounds. We also prove quantitative two-scale expansions
for the kernel, its spatial gradient, and its mixed spatial derivative.
The error estimates hold in every finite probability moment and retain
Gaussian decay. Their dimension-dependent growth is described by a corrector
weight, with a single additional logarithmic factor in the derivative
estimates. The proofs combine large-scale regularity, extended correctors,
weighted estimates, and a parabolic weak-norm estimate. The latter supplies
the passage from deterministic pairings to local annealed estimates without
interchanging expectation and a supremum over test functions. The
mixed-derivative estimate builds on the gradient expansion and does not
introduce a further logarithmic loss.

\end{abstract}

\tableofcontents

\section{Introduction}

\noindent
This paper studies the asymptotic behavior of fundamental solutions for the family of parabolic operators
\begin{equation}\label{operator}
  \partial_t+\mathcal L_\varepsilon
  :=\partial_t-\nabla\cdot a^\varepsilon\nabla,
  \qquad
  a^\varepsilon(x,t):=a(x/\varepsilon,t/\varepsilon^2),
  \qquad \varepsilon>0,
\end{equation}
in $\mathbb R^d\times\mathbb R$, with $d\geq2$. The coefficient field is
uniformly elliptic and bounded, and its law is stationary under space-time
translations. We assume a spectral gap inequality and moment bounds on
the local H\"older seminorms of the coefficients, as specified below.
Our objective is to establish annealed Gaussian upper bounds and
quantitative two-scale expansions for the fundamental solution and its
spatial derivatives. The results and methods apply to parabolic systems, while we employ the scalar notation throughout to simplify the statement.

Homogenization of parabolic equations with coefficients depending on space and
time is treated in, for example, \cite{Bensoussan-Lions-Papanicolaou-11,VSO,Zhikov-Kozlov-Oleinik-82}. In the random setting, Armstrong,
Bordas, and Mourrat \cite{Armstrong-Bordas-Mourrat-18} developed quantitative homogenization and large-scale
regularity under a finite range of dependence assumption\footnote{We note
that, in \cite{Armstrong-Kuusi-22}, Armstrong and Kuusi have already developed their quantitative theory under very general mixing
conditions, going beyond the finite-range setting of \cite{Armstrong-Bordas-Mourrat-18}. These conditions include the spectral-gap type inequalities used in the present article. This line of argument is now known as the coarse-grained theory \cite{Armstrong-Kuusi-22,Armstrong-Kuusi-25}.}. Bella, Chiarini, and Fehrman \cite{Bella-Chiarini-Fehrman-19} established corrector sublinearity and large-scale excess decay for stationary and ergodic ensembles of parabolic systems. A recent work of Geng and Xu \cite{Geng-Xu-26} provides the extended corrector estimates and homogenization error bounds that underlie the present analysis.

In periodic parabolic homogenization, Geng and Shen \cite{Geng-Shen-19} obtained two-scale expansions for the fundamental solution of parabolic systems, its gradient, and its mixed spatial derivative, with Gaussian decay in the error estimates.
These results extend the corresponding elliptic theory developed in \cite{Avellaneda-Lin-91,Kenig-Lin-Shen-14}. In stochastic elliptic homogenization, the relation between corrector estimates and local errors was studied by Bella, Giunti, and Otto \cite{Bella-Giunti-Otto-17}. Related annealed expansion estimates were further developed by Clozeau, Josien, Otto, and Xu \cite{Clozeau-Josien-Otto-Xu}. Wang and Xu \cite{Wang-Xu-26} recently studied a mixed-derivative expansion of Green functions in smooth domains, using boundary correctors. Our aim is to
obtain corresponding quantitative kernel expansions for continuum
parabolic systems with time-dependent random coefficients.

In the random setting, Armstrong, Kuusi, and Mourrat \cite[Section 8.6 and Chapter 9]{Armstrong-Kuusi-Mourrat-19} established quantitative homogenization estimates for parabolic Green functions and their two-scale expansions with time-independent random coefficients, including derivative expansions.
The time dependence considered here requires parabolic correctors and additional temporal flux-corrector terms in the expansion error. Controlling these terms while retaining Gaussian decay is a central part of the present analysis. Expansions in the almost-periodic
setting and for non-self-similar parabolic scalings are considered in
\cite{Geng-Shi-25,Meng-Niu-24}.

A related line of work concerns random walks in dynamic
random environments. Deuschel, Kumagai, and Slowik \cite{Deuschel-Kumagai-Slowik-26} obtain annealed derivative estimates and local limit theorems for random walks among time-dependent conductances, allowing the conductances to be unbounded above. Their entropy method operates under weak statistical assumptions. Here we consider continuum parabolic systems under some quantitative assumptions below and study 
quantitative expansions containing random correctors. The two settings therefore differ in both their assumptions and their conclusions.

A first issue is to obtain Gaussian bounds for systems. In contrast
to scalar equations, boundedness and uniform ellipticity alone do not
give the local boundedness needed for such estimates. Hofmann and Kim
\cite[Theorem 1.1]{Hofmann-Kim-04} showed that suitable local bounds for a parabolic system and
its backward adjoint imply Gaussian bounds for the fundamental solution.
We implement this approach with random local boundedness and gradient
constants. Weighted energy estimates then give quenched Gaussian bounds
with random factors evaluated at the two endpoints. The spectral-gap
assumption controls the moments of these factors and yields annealed
bounds for the kernel and its spatial derivatives.

A further analytical issue is the passage from solution estimates to annealed bounds for kernel expansion errors. The deterministic duality argument developed by Geng and Shen  
\cite{Geng-Shen-19} does not directly supply this passage: moment bounds for each
deterministic test function do not control the probability norm of the
path wise supremum over all test functions. We therefore use a parabolic
weak-norm estimate that controls local moments of a homogeneous solution
and its gradient by a supremum of probability norms of deterministic
pairings. The order of these operations is essential. Together with
extended-corrector identities, weighted estimates, and local regularity,
this estimate provides the kernel and gradient expansions. The
mixed-derivative estimate then uses the gradient expansion to control a
locally homogeneous remainder. Decomposing the source into parabolic
annuli produces a single logarithmic factor in the derivative errors;
the mixed-derivative estimate does not increase its power.

The paper is organized as follows. Section $\ref{sec:2}$ establishes the
scale-invariant regularity estimates used throughout the paper.
Section $\ref{sec:3}$ proves the quenched and annealed Gaussian upper bound in Theorem $\ref{thm:0-new}$.
Section $\ref{sec:4}$ develops the weighted
and smoothing operator estimates and proves the weak-norm estimate in Lemma 
$\ref{lemma:approxi}$. Section $\ref{sec:5}$ proves the kernel, gradient, and mixed-derivative expansions
in Theorems $\ref{thm:1-new}$ and $\ref{thm:3-new}$. 

Finally, unlike the previous works \cite{Avellaneda-Lin-91,Geng-Shen-19,Geng-Shi-25,Meng-Niu-24}, we mention that the methods shown in this paper can be extended to the higher-order expansions.

\subsection{Assumptions and main results}

\noindent
Precisely, assume that the coefficients are uniformly elliptic and bounded: for some
$0<\lambda\le\Lambda<\infty$,
\begin{equation}\label{c:1}
 \lambda|\xi|^2\le \xi\cdot a(z)\xi,
 \quad |a(z)\xi|\le\Lambda|\xi|
 \quad
\text{for~all~}\xi\in\mathbb{R}^{d}\text{~and~}z:=(x,t)\in\mathbb{R}^{d+1}.
\end{equation}
Einstein's summation convention for repeated indices is
used throughout. One can introduce the configuration space that is the set
of coefficient fields satisfying the above condition $\eqref{c:1}$, equip it with a
probability measure (the ensemble),
and denote the expectation by $\langle\cdot\rangle$.
This ensemble is assumed to be \emph{stationary}\footnote{Stationarity means that the law of the coefficient field is invariant under space-time translations.}, i.e., for all shift vectors $z=(x,t)\in\mathbb{R}^{d+1}$, $a(\cdot+x,\cdot+t)$ and $a(\cdot,\cdot)$ have the same law under $\langle\cdot\rangle$, and additionally satisfies the \emph{spectral gap condition}:
\begin{equation}\label{c:3}
 \big\langle (F-\langle F\rangle)^2 \big\rangle \leq \lambda_1
 \Big\langle\int_{\mathbb{R}^{d+1}} dz\Big(\dashint_{Q_1(z)}\big|\frac{\partial F}{\partial a}\big|\Big)^2 \Big\rangle,
\end{equation}
where the notation $Q_r(z):=B_r(x)\times(t-r^2,t+r^2)$
with $z=(x,t)$ is known as the parabolic cube, and
the random tensor field $\frac{\partial F}{\partial a}$ (depending on $(a,z)$) is the functional derivative of $F$ with respect to $\delta a$
(see $\eqref{functional}$ for the specific definition form).
To obtain the pointwise estimates, the local smoothness assumption of the ensemble is introduced, i.e., there exists $\bar{\sigma}\in(0,1)$ such that for any $0<\sigma<\bar{\sigma}$ and $1\leq \gamma<\infty$ there holds
\begin{equation}\label{c:4}
\big\langle [a]_{C^{\sigma,\sigma/2}(Q_1)}^{\gamma}\big\rangle^{1/\gamma}
\lesssim \lambda_{2}(\lambda,\Lambda,\sigma,\gamma).
\end{equation}

Let $\Gamma_\varepsilon$ and $\Gamma_0$ be the fundamental solutions of
$\partial_t+\mathcal L_\varepsilon$ and $\partial_t+\mathcal L_0$,
respectively, where $\mathcal L_0=-\nabla\cdot\bar a\nabla$ and
$\bar a$ is the homogenized matrix defined in $\eqref{eq:1.1}$.
For each fixed $(y,s)$, they satisfy
\begin{equation}\label{pde:100-new}
\left\{\begin{aligned}
&\big(\partial_t+\mathcal{L}_\varepsilon\big)\Gamma_\varepsilon(\cdot,\cdot;y,s) = 0
\quad \text{in}\quad \mathbb{R}^d\times (s,\infty);\\
& \lim_{t\downarrow s}\Gamma_\varepsilon(\cdot,t;y,s) =\delta_y
\quad \text{on} \quad \mathbb{R}^d;\\
& \Gamma_{\varepsilon}(x,t;y,s) = 0
\quad \forall x\in \mathbb{R}^d, \quad \text{if} \quad t<s;
\end{aligned}\right.
\end{equation}
and
\begin{equation}\label{pde:101-new}
\left\{\begin{aligned}
&\big(\partial_t+\mathcal{L}_0\big)\Gamma_0(\cdot,\cdot;y,s) = 0
\quad \text{in}\quad \mathbb{R}^d\times (s,\infty);\\
& \lim_{t\downarrow s}\Gamma_0(\cdot,t;y,s) =\delta_y
\quad \text{on} \quad \mathbb{R}^d;\\
& \Gamma_{0}(x,t;y,s) = 0
\quad \forall x\in \mathbb{R}^d, \quad \text{if} \quad t<s,
\end{aligned}\right.
\end{equation}
respectively, where $\delta_y$ represents the Dirac delta measure at $y$, and the limits are understood in the sense of distributions in the spatial variable.

\medskip
The main results in the paper are stated as follows.

\begin{theorem}[Gaussian upper bound]\label{thm:0-new}
Let $\varepsilon\in(0,1]$ and $d\geq 2$.  Suppose that the ensemble
$\langle\cdot\rangle$ is stationary and ergodic\footnote{In terms of ``ergodic'', we refer the reader to \cite[pp.222-225]{VSO} or \cite[\textcolor{red}{p. 105}]{Gloria-Neukamm-Otto-20} for the definition.}, and the admissible coefficients satisfy the local regularity assumption $\eqref{c:4}$.
Let $\Gamma_\varepsilon$ be the fundamental solution given by $\eqref{pde:100-new}$. Then, there exist nonnegative stationary random fields
$(\mathcal{H}_*,\mathcal{C}_*)$ and $(\tilde{\mathcal{H}}_{*},
\tilde{\mathcal{C}}_{*})$ such that, almost surely, the following Gaussian upper bound holds
\begin{subequations}
\begin{align}
&|\Gamma_\varepsilon(x,t;y,s)|\lesssim_{\lambda,\Lambda,d}
\frac{\mathcal{H}_*(z/\varepsilon)\tilde{\mathcal{H}}_*(z'/\varepsilon)}{(t-s)^{\frac{d}{2}}}\exp\Big\{
-\frac{\kappa_0|x-y|^2}{t-s}
\Big\}; \label{pri:1.1-new*} \\
&|\nabla_x\nabla_y^{a}\Gamma_\varepsilon(x,t;y,s)|+ |\nabla_y\nabla_x^{a}\Gamma_\varepsilon(x,t;y,s)|
\lesssim \frac{\mathcal{C}_*(z/\varepsilon)\tilde{\mathcal{C}}_*(z'/\varepsilon)}{(t-s)^{\frac{d+1+a}{2}}}\exp\Big\{
-\frac{\kappa_0|x-y|^2}{t-s}
\Big\}, \label{pri:1.1-new*2}
\end{align}
\end{subequations}
for all $x,y\in\mathbb R^d$ and $s<t$, where $a\in\{0,1\}$,
$z=(x,t)$, $z'=(y,s)$, and
$z/\varepsilon=(x/\varepsilon,t/\varepsilon^2)$.
Here $\tilde{\mathcal H}_*$ and
$\tilde{\mathcal C}_*$ are the corresponding random fields for
the backward adjoint operator $-\partial_t-\nabla\cdot a^{*}\nabla$, and $\kappa_0>0$ depends only on
$\Lambda^2/\lambda$. If, in addition, the ensemble satisfies $\eqref{c:3}$, then the random fields have moments of every finite order, with the bounds
\begin{equation}\label{pri:1.2-new*}
\big\langle\big(\mathcal{H}_{*}^{\beta}(z),\mathcal{C}_{*}^{\beta}(z)\big)\big\rangle^{1/\beta}
+ \big\langle\big(\tilde{\mathcal{H}}_{*}^{\beta}(z'),\tilde{\mathcal{C}}_{*}^{\beta}(z')\big)\big\rangle^{1/\beta}
 \lesssim_{\lambda,\Lambda,\lambda_1,\lambda_2,d,\beta} 1
 \quad \forall \beta\in[1,\infty), z,z'\in\mathbb{R}^{d+1}.
\end{equation}
Consequently,
for every  $q\in[1,\infty)$ and $a,b\in\{0,1\}$, there exists a constant $C_q$ depending only on
$\lambda,\Lambda,\lambda_1,\lambda_2,d,q$, such that the following annealed decay estimates
\begin{equation}\label{pri:1.3-new*}
\big\langle|\nabla_x^{a}\nabla_y^{b}\Gamma_\varepsilon(x,t;y,s)|^q\big\rangle^{\frac{1}{q}}
\lesssim \frac{C_q}{(t-s)^{\frac{d+a+b}{2}}}\exp\Big\{
-\frac{\kappa_0|x-y|^2}{t-s}
\Big\}
\end{equation}
hold for any $x,y\in\mathbb{R}^d$ and $t>s$, where we use the convention $\nabla^{0}=1$.
\end{theorem}

Theorem $\ref{thm:0-new}$ separates the deterministic Gaussian decay rate from the
random prefactors (i.e., $\mathcal{H}_*$ and $\tilde{\mathcal{H}}_*$) at the two space-time endpoints. The spectral gap assumption $\eqref{c:3}$ is merely used to control the moments of these prefactors (see $\eqref{pri:1.2-new*}$).

Let $\phi_j$ and $\phi_j^*$ denote the forward and backward adjoint
correctors, respectively:
\begin{equation}
 \partial_t\phi_j-\nabla\cdot a(\nabla\phi_j+e_j)=0,
 \qquad
 -\partial_t\phi_j^*-\nabla\cdot a^*(\nabla\phi_j^*+e_j)=0
 \quad\text{in}\quad \mathbb{R}^{d+1}.
\end{equation}
Their normalizations and the associated flux correctors are specified
in Section $\ref{sec:2}$.
 Since $\Gamma_0$ is the fundamental solution associated with the homogenized operator, it is well known that $\Gamma_0(x,t;y,s) = \Gamma_0(x-y,t-s;0,0)$ for any $x,y\in\mathbb{R}^d$ and $-\infty<s<t<\infty$.
For simplicity, write $F^\varepsilon(x,t)=F(x/\varepsilon,t/\varepsilon^2)$.

\begin{theorem}[Kernel and gradient expansions]\label{thm:1-new}
Let $\varepsilon\in(0,1]$ and $d\geq 2$.
Suppose that the ensemble
$\langle\cdot\rangle$ is stationary and satisfies the spectral gap condition $\eqref{c:3}$, and the admissible coefficients satisfy the local regularity assumption $\eqref{c:4}$.
Let $\Gamma_\varepsilon$ and $\Gamma_0$ be the fundamental solutions defined by $\eqref{pde:100-new}$ and $\eqref{pde:101-new}$, respectively. Then, for every
$p\in[1,\infty)$, the following estimates hold:
\begin{equation}\label{pri:1.1-x}
\big\langle\big|\Gamma_\varepsilon(x,t;y,s) - \Gamma_0(x,t;y,s)\big|^p\big\rangle^{\frac{1}{p}}\lesssim
\frac{\varepsilon\max\big\{2,\mu_d(\sqrt{t-s}/\varepsilon)\big\}}{(t-s)^{\frac{d+1}{2}}}
\exp\Big\{-\frac{\kappa|x-y|^2}{t-s}\Big\}
\end{equation}
and
\begin{equation}\label{pri:1.1-xx}
\begin{aligned}
\big\langle\big|\nabla_x\Gamma_\varepsilon(x,t;y,s)
&- (e_i+\nabla\phi_{i})^{\varepsilon}(x,t)\partial_{x_i}\Gamma_0(x,t;y,s)\big|^p\big\rangle^{\frac{1}{p}}\\
&+\big\langle\big|\nabla_y\Gamma_\varepsilon(x,t;y,s)-
(e_j+\nabla\phi_{j}^{*})^{\varepsilon}(y,s)\partial_{y_j}\Gamma_0(x,t;y,s)\big|^p\big\rangle^{\frac{1}{p}}\\
&\qquad\quad\lesssim
\frac{\varepsilon\max\big\{4\ln2,\mu_d(\sqrt{t-s}/\varepsilon)\ln(2+\sqrt{t-s}/\varepsilon)\big\}}{(t-s)^{\frac{d+2}{2}}}
\exp\Big\{-\frac{\kappa|x-y|^2}{t-s}\Big\}
\end{aligned}
\end{equation}
for any $x,y\in\mathbb{R}^d$ and $s<t$,
where $\mu_d$ is given by
\begin{equation}\label{pri:5.5**}
\mu_d(r):=\left\{\begin{aligned}
& r^{\alpha} &\quad&\text{if}~~r\leq 2;\\
&\ln^{\frac{3}{4}}(2+r) &\quad&\text{if}~~d=2,\quad r>2;\\
& 1  &\quad&\text{if}~~d>2,\quad r>2,
\end{aligned}\right.
\qquad \alpha\in(0,1],
\end{equation}
and the multiplicative constants depend on
$\lambda,\Lambda,\lambda_1,\lambda_2,d,p$, and $\kappa>0$ depends only on $\Lambda^2/\lambda$.
\end{theorem}

\begin{theorem}[Mixed-derivative expansion]\label{thm:3-new}
Under the assumptions of Theorem 1.2,  the corresponding error for the truncated two-scale expansion of the mixed spatial derivative is defined by
\begin{equation}\label{eq:6.1-x}
 \mathcal{E}_\varepsilon(x,t;y,s) :=\nabla_x\nabla_y\Gamma_\varepsilon(x,t;y,s)
 +(e_i+\nabla\phi_i)^{\varepsilon}(x,t)\otimes(e_j+\nabla\phi_j^{*})^{\varepsilon}(y,s)\partial_{ij} \Gamma_0(x,t;y,s),
\end{equation}
where $\partial_{ij}:=\partial_{x_i}\partial_{x_j}$.
Then, for any $p\in[1,\infty)$, there holds the annealed decay estimate
\begin{equation}\label{pri:1.1-new}
\begin{aligned}
\big\langle\big|\mathcal{E}_\varepsilon(x,t;y,s)\big|^p\big\rangle^{\frac{1}{p}}\lesssim
\frac{\varepsilon\max\big\{4\ln2,\mu_d(\sqrt{t-s}/\varepsilon)\ln(2+\sqrt{t-s}/\varepsilon)\big\}}{(t-s)^{\frac{d+3}{2}}}
\exp\Big\{-\frac{\kappa|x-y|^2}{t-s}\Big\}
\end{aligned}
\end{equation}
for any $x,y\in\mathbb{R}^d$ and $s<t$, where
the multiplicative constant depends only on
$\lambda,\Lambda,\lambda_1,\lambda_2,d,p$, and $\kappa = \lambda/(C\Lambda^2)$ with  $C>1$ being a sufficiently large constant.
\end{theorem}

Theorems $\ref{thm:1-new}$ and $\ref{thm:3-new}$ quantify the two-scale approximation of the
fundamental solution and its spatial derivatives in every finite
probability moment. The gradient expansion uses the corrector at the
observation point $(x,t)$, while the mixed-derivative expansion also involves
the adjoint corrector at the source point $(y,s)$. All the expansion errors retain
Gaussian decays. Each additional spatial derivative contributes a factor
of $(t-s)^{-1/2}$; the dimension-dependent growth is encoded by
$\mu_d$, with an additional logarithmic factor in the derivative estimates. We emphasize that the mixed derivative does not increase the power of this logarithmic factor.

\subsection{Proof strategy and main ingredients}

\noindent
The proofs start from scale-invariant regularity estimates. Under
stationarity and ergodicity, Lemma $\ref{lemma:approxi}$ gives a local approximation by a homogenized solution beyond a random minimal radius.
A Campanato iteration yields large-scale H\"older estimates and excess decay.
Classical local estimates then extend these bounds to all scales under
the coefficient regularity assumption $\eqref{c:4}$. The resulting local
boundedness and gradient constants (i.e., $\mathcal{H}_*$ and $C_*$ given in Propositions $\ref{prop:2.1}$ and $\ref{prop:2.2-new}$, respectively)
are stationary random fields whose dependence on the minimal radius and local coefficient norms is explicit. The estimates respect parabolic rescaling. Quantitative ergodicity such as the spectral gap condition $\eqref{c:3}$ is
needed to control the moments of these random constants, rather than to
establish the qualitative large-scale estimates themselves.

\paragraph{Gaussian upper bound.}
For Theorem $\ref{thm:0-new}$, we work first with a fixed realization of the coefficients and
set $\varepsilon=1$. Given a bounded Lipschitz function $\psi$ with
$\|\nabla\psi\|_{L^\infty}\le\gamma$, define
\[
 P^\psi_{s\to t}f:=e^\psi P_{s\to t}(e^{-\psi}f).
\]
The classical weighted $L^2$ energy estimate and the local boundedness estimate
give Lemma $\ref{lemma:3.1-newnew}$:
\begin{equation*}
\begin{aligned}
|P^\psi_{s\to t}f(x)|
& \lesssim \mathcal H_*(x,t)(t-s)^{-d/4}
 e^{\kappa\gamma^2(t-s)}\|f\|_{L^2(\mathbb R^d)}; \\
|\nabla P_{s,t}^{\psi}(f)(x)|
&\lesssim \mathcal{C}_{*}(x,t)(t-s)^{-d/4-1/2}
 e^{\kappa\gamma^2(t-s)}\|f\|_{L^2(\mathbb{R}^d)},
\end{aligned}
\end{equation*}
where $\mathcal{C}_{*}:=\mathcal{H}_{*}\vee C_{*}$ is a stationary random field determined by the coefficient fields and 
the minimal radius, due to the regularity theory. Applying the corresponding estimates to the backward adjoint equation and composing at an intermediate time gives all four endpoint bounds in Theorem $\ref{thm:0-new}$. Optimizing the
weight $\psi$ yields the Gaussian decay.

\paragraph{Weak-norm estimates.}
For a random linear functional $T_\omega$ on a deterministic test space
$X$, one has
\[
\sup_{\|f\|_X\le1}\|T_\omega f\|_{L^p(\Omega)}
\le
\left\|\sup_{\|f\|_X\le1}|T_\omega f|\right\|_{L^p(\Omega)},
\]
with the usual measurability convention. The reverse bound is unavailable
in general. Thus, the periodic duality method in \cite{Geng-Shen-19} cannot simply be combined with the available moment estimates. For locally homogeneous solutions, Lemma $\ref{lemma:4-xx}$ provides the required substitute: local solution and gradient moments are controlled by deterministic pairings, with the probability norm taken before the supremum.

\begin{lemma}[weak-norm estimates]\label{lemma:4-xx}
Let $\varepsilon\in(0,1]$ and $r>0$. Suppose that
$\langle\cdot\rangle$ is an ensemble of $(\lambda,\Lambda)$-uniformly elliptic and bounded coefficients and $u_\varepsilon$ is
a weak solution of $\big(\partial_t+\mathcal{L}_\varepsilon\big)u_\varepsilon=0$ in $Q_{2r}$. Then,
for all $p\in[1,\infty)$, we obtain that
\begin{subequations}
\begin{align}
&\Big\langle\Big(\dashint_{Q_{r/2}}|u_\varepsilon|^2\Big)^{\frac{p}{2}}
 \Big\rangle^{\frac{1}{p}}
 \lesssim_{\lambda,\Lambda,d,p}
 \sup_{w\in C^{\infty}_0(Q_r)}\frac{\big\langle|\dashint_{Q_r}u_\varepsilon
  w|^p
\big\rangle^{\frac{1}{p}}}{r^2\sup|\nabla^2 w|}; \label{pri:27-xxnew*}\\
& \Big\langle\Big(\dashint_{Q_{r/2}}|\nabla u_\varepsilon|^2\Big)^{\frac{p}{2}}
 \Big\rangle^{\frac{1}{p}}
 \lesssim_{\lambda,\Lambda,d,p}
 \sup_{w\in C^{\infty}_0(Q_r)}\frac{\big\langle|\dashint_{Q_r}u_\varepsilon
  w|^p
\big\rangle^{\frac{1}{p}}}{r^3\sup|\nabla^2 w|}.\label{pri:27-xx}
\end{align}
\end{subequations}
\end{lemma}

The proof of Lemma $\ref{lemma:4-xx}$ adapts the weak-norm approach of
\cite{Bella-Giunti-Otto-17,Clozeau-Josien-Otto-Xu,Wang-Xu-26} to the
parabolic setting. Parabolic smoothing and local interpolation reduce
the estimate to a compact spectral construction for a positive self-adjoint realization of a sufficiently high even power of
$-\partial_t^2+\Delta_x^2$. Both terms have the same parabolic homogeneity.
The summable spectral weights allow the probability norm to be taken
before the spectral sum, so the argument applies to every finite moment
without interchanging expectation and a random supremum.

\paragraph{Kernel, gradient, and mixed-derivative expansions.}
Theorems $\ref{thm:1-new}$ and $\ref{thm:3-new}$ use the same local error mechanism. We normalize the
original and adjoint correctors at the relevant endpoint and include the
temporal flux-corrector term so that the expansion error satisfies a
divergence-form equation. Weighted estimates for the homogenized equation
and parabolic smoothing control its source while retaining the Gaussian
factor. For the kernel error, a local decomposition separates the forced
part from a homogeneous remainder; Lemma $\ref{lemma:4-xx}$ and weighted duality control
the latter through deterministic pairings. Local gradient estimates and
a decomposition of the source into parabolic annuli then give the
gradient expansion. Corrector increments make the small-scale sum
convergent, whereas the scales between $\varepsilon$ and
$\sqrt{t-s}$ contribute a single logarithmic factor. For the mixed
derivative, we repeat this decomposition after differentiating at the
opposite endpoint. The established gradient expansion controls the local
size of the homogeneous remainder, so another weak-norm duality argument
is unnecessary. This adds one factor of $(t-s)^{-1/2}$ without a further
logarithmic loss. The weight $\mu_d$ records the growth of the temporal
flux corrector, including its additional logarithmic growth in dimension
two.

\subsection{Notation}
\noindent
We use the following conventions throughout the paper.
\begin{enumerate}
  \item Notation for estimates.
\begin{enumerate}
  \item $\lesssim$ and $\gtrsim$ stand for $\leq$ and $\geq$
  up to a multiplicative constant,
  which may depend on some given parameters in the paper
  but never on $\varepsilon$. The subscript form $\lesssim_{\lambda,\cdots,\lambda_n}$ indicates the constant depends only on parameters $\lambda,\cdots,\lambda_n$.
  In addition, superscripts such as $\lesssim^{(2.1)}$ are used to indicate the referenced formula or estimate.

  \item  The notation $\sim$ is used when both $\lesssim$ and $\gtrsim$ hold.
\end{enumerate}
  \item Notation for derivatives and contractions.
  \begin{enumerate}
  \item Temporal derivative: $\partial_t u:= \partial u/\partial t$ is the derivative with respect to the time variable.
    \item Spatial derivatives:
  $\nabla v = (\partial_1 v, \cdots, \partial_d v)$ is the gradient of $v$, where
  $\partial_i v = \partial v /\partial x_i = \partial_{x_i}v$ denotes the
  $i^{\text{th}}$ derivative of $v$.
  $\nabla^2 v$  denotes the Hessian matrix of $v$;
  $\nabla\cdot v=\sum_{i=1}^d \partial_i v_i$
  denotes the divergence of $v$, where
  $v = (v_1,\cdots,v_d)$ is a vector-valued function.
$\Delta_{d+1}:=\sum_{i=1}^{d+1}\partial_i^2$ represents the $(d+1)$-dimensional Laplace operator, where we note that  $\partial_{d+1} := \partial_t$.
   \item Functional (or vertical) derivative:  the random tensor field $\frac{\partial F}{\partial a}$ (depending on $(a,z)$) is the functional derivative of $F$ with respect to $a$, defined by
       \begin{equation}\label{functional}
       \lim_{\varepsilon\to 0} \frac{F(a+\varepsilon\delta a)
       -F(a)}{\varepsilon}
       =\int_{\mathbb{R}^{d+1}}dz \frac{\partial F(a)}{\partial a_{ij}(z)}
       (\delta a)_{ij}(z).
       \end{equation}
   \item The vectors $e_1,\cdots,e_d$ form the standard basis of $\mathbb{R}^d$. We write $(v\otimes w)_{ij} = v_iw_j$, and
   $A:B=\sum_{ij}A_{ij}B_{ij}$, with analogous contractions for higher-order tensors.
  \end{enumerate}

\item Notation for functions, space-time geometry and scaling.
\begin{enumerate}
    \item The function $\textbf{1}_{E}$ is the indicator function of $E$.

\item $a\vee b := \max\{a,b\}$; $a\wedge b:=\min\{a,b\}$.

\item We denote the support of $f$ by $\text{supp}(f)$, and $\dashint_{A}f:=\frac{1}{|A|}\int_{A}f$ with $|A|$ being the volume of $A$.

\item Let $\mu_d$ and $\bar{\mu}_d$ be the (large scale) growth weights given in
 $\eqref{pri:5.5**}$ and $\eqref{pri:5.5}$, respectively.

\item Write $z=(x,t)$ and $z'=(y,s)$ in $\mathbb R^{d+1}$. Set
\[
 B_r(x):=\{y\in\mathbb R^d:|y-x|<r\},\qquad
 Q_r(z):=B_r(x)\times(t-r^2,t+r^2),
\]
and write $B_r=B_r(0)$ and $Q_r=Q_r(0)$.
The parabolic distance is
\[
 \mathrm d(z,z'):=|x-y|+|t-s|^{1/2}.
\]

For a parabolic cube $Q_r$, the forward parabolic
boundary of $Q_r$ is
\[
 \partial_p Q_r:=(\partial B_r\times[t-r^2,t+r^2])
                 \cup(\overline B_r\times\{t-r^2\}).
\]

\item We use the parabolic rescaling:
\[
 z/\varepsilon:=(x/\varepsilon,t/\varepsilon^2),\qquad
 F^\varepsilon(z):=F(z/\varepsilon).
\]
In particular, $(\nabla F)^\varepsilon=\varepsilon\nabla F^\varepsilon$ and
$(\partial_tF)^\varepsilon=\varepsilon^2\partial_tF^\varepsilon$.

  \end{enumerate}


\item Notation for function spaces.
\begin{enumerate}
  \item $C_0^1(Q_{R})$ denotes the space of functions with continuous first-order derivatives with respect to spatial or time variables, requiring each element to vanish near $\partial Q_{R}$; $C_0^\infty(Q_R)$ consists of compactly supported functions with arbitrary-order space-time derivatives.
  For $0<\alpha<1$, $C^{\alpha,\alpha/2}(Q_R)$ denotes the parabolic Hölder space,
whereas the parabolic H\"older seminorm is
\begin{equation*}
[u]_{C^{\alpha,\alpha/2}(Q_R)}
 :=\sup_{\substack{z,z'\in Q_R\\z\ne z'}}
 \frac{|u(z)-u(z')|}{\mathrm d(z,z')^\alpha}.
\end{equation*}

\item
Let $(\Sigma,\mathcal{F},\mathbb{P})$ be the probability space
of coefficient fields and let $\langle\cdot\rangle$ denote expectation.
For a random variable $F$, set
\begin{equation*}
 \|F\|_p:=\langle|F|^p\rangle^{1/p},\qquad 1\le p<\infty.
\end{equation*}

  \item For a Banach space $X$, let $L^q(I;X)$ be the Bochner space of strongly measurable $X$-valued functions with finite norm
$(\int_I dt\|f(t)\|_X^q)^{1/q}$, where $1\leq q<\infty$.

\item
We use the usual Sobolev spaces $H^m$ and $H_0^m$; in particular,
$H^{-1}(U)=(H_0^1(U))^*$.
For $k\ge1$, the anisotropic Sobolev norm on a cylinder $Q_R$ is
\[
 \|u\|_{H^{2k,k}(Q_R)}^2
 :=\sum_{|\alpha|+2j\le2k}
       \|\partial_x^\alpha\partial_t^j u\|_{L^2(Q_R)}^2,
\]
where $\alpha$ is a spatial multi-index. The space $H_0^{2k,k}(Q_R)$ is the
closure of $C_0^\infty(Q_R)$ in this norm.

\end{enumerate}

\end{enumerate}

\begin{remark}
\emph{A solution with zero initial trace at time $s_0$ is extended by zero to
$t<s_0$ when needed. If an extension beyond a terminal time $T$ is required,
we continue the forward equation with the given coefficient field, zero
forcing, and initial trace $u(T)$. The extended functions retain their original
notation.}
\end{remark}

\section{Scale-invariant estimates}\label{sec:2}


\begin{lemma}[extended correctors \cite{Geng-Xu-26}]\label{lemma:2.1}
Let $d\geq 2$ and $1\leq p<\infty$. Suppose that the ensemble $\langle\cdot\rangle$ is stationary and ergodic. Let $1\leq i,j,k\leq d$ and
$1\leq k'\leq d+1$ be integers. Then, there exist random fields $\phi=\{\phi_j\}$,  $\sigma=\{\sigma_{k'ij}\}$  with the following properties:
\begin{enumerate}
  \item The random fields $\nabla\phi_j$  and $\nabla\sigma_{kij}$ are
stationary\footnote{We say that a random field $X(a,x)$ is stationary if it is shift-covariant in the sense of $X(a,x+z) = X(a(z),x)$ for any $x,z\in\mathbb{R}^{d+1}$ and
$\langle\cdot\rangle$-a.e. $a$. This implies that $X(a,\cdot)$ and $X(a,\cdot+z)$ have the same law under
$\langle\cdot\rangle$.}, and they satisfy
the following equations:
\begin{equation}\label{pde:correctors}
\begin{aligned}
\partial_t\phi_j -\nabla\cdot a(\nabla\phi_j+e_j) &= 0;\\
\Delta_{d+1}\sigma_{k'ij}
 &=\partial_{k'}q_{ij}-\partial_i q_{k'j};\\
 \partial_{k'}\sigma_{k'ij}
 &= q_{ij};\\
 \partial_i\sigma_{(d+1)ij}
 &= M_j-\phi_j
\end{aligned}
\end{equation}
in a weak sense on $\mathbb{R}^{d+1}$, where $M$ is a constant\footnote{If we implicitly assume that $\phi_j$ is stationary, one can take $M_j=\langle\phi_j\rangle$.} depending on the choice of normalization of $\sigma_{(d+1)ij}$ and $\phi_j$, and we use the following notations:
\begin{equation}\label{eq:1.1}
q_{ij}:= \bar{a}_{ij}
- a_{ij}
- a_{ik}\partial_k\phi_j,
\quad
q_{(d+1)j} :=\phi_j-\langle\phi_j\rangle,
\quad \bar{a}_{ij}:=\big\langle e_i\cdot a(\nabla\phi_j+e_j)\big\rangle,
\end{equation}
and we also mention that $\nabla\sigma_{(d+1)ij}$ is not stationary, merely admitting the sublinear growth
\begin{equation*}\limsup_{R\to\infty}\frac{1}{R^2}\dashint_{Q_R}
 \big|\nabla\sigma_{(d+1)} - \dashint_{Q_{R}}\nabla\sigma_{(d+1)}\big|^2 = 0.
\end{equation*}
  \item The field $\sigma$ (known as the flux-corrector) is skew-symmetric in its first two indices, i.e.,
      \begin{equation}\label{skew-symmetric}
      \sigma_{k'ij} = -\sigma_{ik'j}.
      \end{equation}
  \item  Suppose that the ensemble
$\langle\cdot\rangle$ additionally satisfies the spectral gap condition $\eqref{c:3}$.
There hold the following quantitative estimates:
  \begin{subequations}
  \begin{align}
  & \bigg\langle\Big(\dashint_{Q_1(z)}\big|(\nabla \phi_j,\nabla \sigma_{k'ij})\big|^{2}\Big)^p\bigg\rangle^{1/p}\lesssim_{\lambda,\Lambda,\lambda_1,d,p}1;
  \label{pri:5.2}\\
  &\bigg\langle \Big|\dashint_{Q_1(z)}\big(\phi_j,\sigma_{kij}\big) - \dashint_{Q_1(0)}\big(\phi_j,\sigma_{kij}\big)\Big|^{2p}\bigg\rangle^{1/p}
\lesssim_{\lambda,\Lambda,\lambda_1,d,p}1;
  \label{pri:5.4}\\
  &\bigg\langle \Big|\dashint_{Q_1(z)}\sigma_{(d+1)ij} - \dashint_{Q_1(0)}\sigma_{(d+1)ij}\Big|^{2p}\bigg\rangle^{1/p}
\lesssim_{\lambda,\Lambda,\lambda_1,d,p} \bar{\mu}_d^2(z)
\qquad \forall z\in\mathbb{R}^{d+1},
\label{pri:5.4b}
  \end{align}
  \end{subequations}
 in which the weight function $\bar{\mu}_d$ is defined as
 \begin{equation}\label{pri:5.5}
\bar{\mu}_d(z):=\left\{\begin{aligned}
&\ln^{\frac{3}{4}}(4+\text{d}(z,0)) &\quad&\text{if}~~d=2;\\
& 1  &\quad&\text{if}~~d>2.
\end{aligned}\right.
\end{equation}
\item  Apart from the condition $\eqref{c:3}$, assume that the ensemble
$\langle\cdot\rangle$ additionally satisfies the local smoothness condition $\eqref{c:4}$. Let $0<\alpha< 1$.
There hold the following pointwise estimates:
  \begin{subequations}
  \begin{align}
  & \Big\langle \big|(\nabla \phi_j,\nabla \sigma_{k'ij})(z)\big|^{2p}\Big\rangle^{1/p}
  +\big\langle[\nabla\sigma_{(d+1)}]_{C^{\alpha,\alpha/2}(Q_1(z))}^{2p}\big\rangle^{1/p}\lesssim_{\lambda,\Lambda,\lambda_1,\lambda_2,d,p}1;
  \label{pri:5.2-n}\\
  &\Big\langle \big|\big(\phi_j,\sigma_{kij}\big)(z)-\big(\phi_j,\sigma_{kij}\big)(0)\big|^{2p}\Big\rangle^{1/p}
\lesssim_{\lambda,\Lambda,\lambda_1,\lambda_2,d,p}\min\big\{[\text{d}(z,0)]^{2\alpha},1\big\};
 \label{pri:5.4-n} \\
  &\Big\langle \big|\sigma_{(d+1)ij}(z)-\sigma_{(d+1)ij}(0)\big|^{2p}\Big\rangle^{1/p}
\lesssim_{\lambda,\Lambda,\lambda_1,\lambda_2,d,p} \hat{\mu}_d^2(z)
\qquad \forall z\in\mathbb{R}^{d+1},
\label{pri:5.4b-n}
  \end{align}
  \end{subequations}
\end{enumerate}
where the notation ``$\lesssim_{\mu,\lambda_1,p,d}$'' indicates that the multiplicative constant depends on $\mu,\lambda_1,p,d$, and the weight function $\hat{\mu}_d$ is defined as
\begin{equation*}
\hat{\mu}_d(z):=\left\{\begin{aligned}
& [\text{d}(z,0)]^{\alpha} &\quad&\text{if}~~\text{d}(z,0)\leq 2;\\
&\bar{\mu}_d(z) &\quad&\text{if}\quad \text{d}(z,0)>2.
\end{aligned}\right.
\end{equation*}
\end{lemma}

\begin{remark}
\emph{
Let $-\partial_t+\mathcal{L}_\varepsilon^{*}$ be the adjoint operator, where $\mathcal{L}_\varepsilon^{*}:=-\nabla\cdot a^{*}(\cdot/\varepsilon)\nabla$. Parallel to Lemma $\ref{lemma:2.1}$, in our subsequent arguments, we need to use the adjoint correctors and the corresponding flux correctors, denoted as $\phi_j^{*}$
and $\sigma_{k'ij}^{*}$ with $k'\in\{1,\cdots,d+1\}$ and $i,j\in\{1,\cdots,d\}$. They satisfy the following system of equations:
\begin{equation}\label{pde:correctors*}
\begin{aligned}
-\partial_t\phi_j^{*} -\nabla\cdot a^{*}(\nabla\phi_j^{*}+e_j) &= 0;\\
\Delta_{d+1}\sigma_{k'ij}^{*}
 &=\partial_{k'}q_{ij}^{*}-\partial_i q_{k'j}^{*};\\
 \partial_{k'}\sigma_{k'ij}^{*}
 &= q_{ij}^{*};\\
 \partial_i\sigma_{(d+1)ij}^{*}
 &=\phi_j^{*}- M_j^{*},
\end{aligned}
\end{equation}
 where $M_j^{*}$ is a constant depending on the choice of normalization of $\sigma_{(d+1)ij}^{*}$ and $\phi_j^{*}$, and we employ the following notations:
\begin{equation}\label{eq:1.1**}
q_{ij}^{*}:= \bar{a}_{ji}
- a_{ji}
- a_{ik}^{*}\partial_k\phi_j^{*},
\quad
q_{(d+1)j}^{*} :=\langle\phi_j^{*}\rangle-\phi_j^{*},
\quad\text{and}\quad \bar{a}^{*}=(\bar{a})^{*}.
\end{equation}
Also, we also have the skew-symmetric property: $\sigma_{k'ij}^{*}=-\sigma_{ik'j}^{*}$. Since the transpose of $a$ does not alter the distribution properties and the smoothness of the sample orbits, when we assume the spectral gap condition $\eqref{c:3}$ and the local smoothness condition $\eqref{c:4}$, the extended correctors $(\phi_j^{*},\sigma_{k'ij}^{*})$ satisfy the same quantitative estimates as those stated for $(\phi_j,\sigma_{k'ij})$ in Lemma $\ref{lemma:2.1}$. We do not repeat these estimates here, so when making relevant references, directly refer to the corresponding estimates of $(\phi_j,\sigma_{k'ij})$ in Lemma $\ref{lemma:2.1}$.}
\end{remark}

%

\begin{lemma}[qualitative theory of local version]\label{lemma:approxi}
Let $\varepsilon\in(0,1]$ and $d\geq 2$. Assume that the ensemble $\langle\cdot\rangle$
is stationary and ergodic. Then, for any $\nu>0$,  there exists
a stationary random field
$\chi_*$,  such that,
for any solution $u_\varepsilon$ to $\big(\partial_t+\mathcal{L}_\varepsilon\big) (u_\varepsilon) = F$
in $Q_{8r}$, there exists
a weak solution $\bar{u}_r$ satisfying
$\big(\partial_t+\mathcal{L}_0\big)(\bar{u}_r) = F$ in $Q_{2r}$ with
$\bar{u}_r = u_\varepsilon$ on $\partial_p  Q_{2r}$, such that
\begin{subequations}
\begin{align}
&\dashint_{Q_r}|u_\varepsilon -\bar{u}_r|^2 \leq
\nu^2
 \bigg\{\dashint_{Q_{4r}}|u_\varepsilon|^2+r^4\dashint_{Q_{4r}}|F|^2\bigg\};
\label{pri:2.1px} \\
&\dashint_{Q_r}|\nabla u_\varepsilon -
(e_j+(\nabla\phi_j)^\varepsilon)\partial_j\bar{u}_r|^2 \leq
\nu^2\bigg\{
 \dashint_{Q_{4r}}|\nabla u_\varepsilon|^2
 + r^2\dashint_{Q_{4r}}|F|^2\bigg\}
 \label{pri:2.2px}
\end{align}
\end{subequations}
hold for any $r\geq \varepsilon \chi_*$. Moreover,
if the ensemble $\langle\cdot\rangle$ satisfies the spectral gap condition $\eqref{c:3}$, then the stationary random field $\chi_*$ satisfies
\begin{equation}\label{pri:2.3px}
\langle
\chi_*^\beta\rangle\lesssim_{\mu,\lambda_1,d,\beta} 1
\qquad \forall\beta\in[1,\infty).
\end{equation}
\end{lemma}


\begin{proof}
The result $\eqref{pri:2.2px}$ has been shown in \cite[Lemma 1.4]{Geng-Xu-26}, while we reproduce the same idea\footnote{For a  compactness argument, we also refer the readers to \cite[Proposition 2.1]{Josien-Otto-22} for an alternative approach.} to show the estimate $\eqref{pri:2.1px}$ for the reader's convenience.
By a rescaling argument it is fine to assume $r=1$.
Let $\Psi_{\delta}\in C_0^{2}(Q_2)$ with $0< \delta\leq (1/2)$ (fixed later on) be a cut-off function satisfying $\Psi_{\delta} = 1$ on $Q_{2-4\delta}$,
$\Psi_{\delta} = 0$ outside $Q_{2-\delta}$ and
\begin{equation*}
(1/\delta)|\nabla\Psi_{\delta}|+|\nabla^2\Psi_{\delta}|+|\partial_t\Psi_{\delta}|
\lesssim_d 1/\delta^2
\qquad\text{on}
\quad Q_{2-\delta}\setminus Q_{2-2\delta}.
\end{equation*}
Let $\bar{u}_1$
be given by the equations: $\big(\partial_t+\mathcal{L}_0\big)(\bar{u}_1) = \tilde{F}$ in $Q_{2}$ with $\bar{u}_1 = \tilde{u}_{\varepsilon'}$ on $\partial_p  Q_{2}$,
where $\tilde{u}_{\varepsilon'}(y) := u_\varepsilon(ry)$ and
$\tilde{F}(y):=r^2F(ry)$ with $y\in Q_{2}$ and $\varepsilon'=\varepsilon/r$.
The non-zero $F$ does not have an essential impact on the proof (it only involves the traditional PDE regularity theory, i.e., Meyers-type estimates, and 
the minimal radius defined later on is independent of $F$). For simplicity in the subsequent calculations, we may assume that $F = 0$.

Appealing to the two-scale asymptotic analysis and Lemma $\ref{lemma:2.1}$, the error expansion
related to $\tilde{u}_{\varepsilon'}$ (still denoted by $u_\varepsilon$)
and $\bar{u}_1$
can be expressed by
\begin{equation*}
w_{\varepsilon} := u_{\varepsilon}  - \bar{u}_1
-\varepsilon\bar{\phi}_j^\varepsilon S_\delta (\Psi_{\delta}\partial_j\bar{u}_1)
-\varepsilon^2 \tilde{\sigma}_{l(d+1)j}^\varepsilon\partial_l
S_\delta (\Psi_{\delta}\partial_j\bar{u}_1),
\end{equation*}
with $S_\delta$ being the (standard) smoothing operator (see Definition $\ref{def:1}$), and
\begin{equation*}
\begin{aligned}
&\bar{\phi}_j:=\phi_j-\dashint_{Q_R}\phi_j;
\qquad \bar{\sigma}_{kij}(x,t)
:=\sigma_{kij}(x,t)-\dashint_{B_R}\sigma_{kij}(\cdot,t);\\
&\tilde{\sigma}_{l(d+1)j}(x,t)
:= \sigma_{l(d+1)j}(x,t) - \dashint_{Q_R}\sigma_{l(d+1)j}
-x_k \dashint_{Q_R}\partial_{x_k}\sigma_{l(d+1)j}, \qquad i,j,k\in\{1,\cdots,d\}.
\end{aligned}
\end{equation*}
Moreover, there holds
\begin{equation*}
\left\{\begin{aligned}
\partial_t w_\varepsilon + \mathcal{L}_\varepsilon(w_\varepsilon)
& = \nabla\cdot\tilde{f}_\varepsilon &\quad&\text{in}\quad ~Q_2,\\
w_\varepsilon & = 0 &\quad&\text{on}\quad \partial_p Q_{2},
\end{aligned}\right.
\end{equation*}
in which $\tilde{f}_\varepsilon$ can be concretely formulated as
\begin{equation*}
\begin{aligned}
f_\varepsilon
&:= (a^\varepsilon-\bar{a})
\big(\nabla u_0-\varphi\big)
 + \varepsilon \big[a^{\varepsilon}\bar{\phi}_j^{\varepsilon}
-\bar{\sigma}_{j}^{\varepsilon}
+ a^{\varepsilon}(\partial \tilde{\sigma}_{(d+1)j})^{\varepsilon}
\big]\nabla\varphi_j \\
& \qquad\quad+ \varepsilon^2\Big[
a^{\varepsilon}\otimes \tilde{\sigma}_{(d+1)j}^{\varepsilon}: \partial^2 \varphi_j
- \tilde{\sigma}_{(d+1)j}^{\varepsilon}\partial_t\varphi_j\Big].
\end{aligned}
\end{equation*}

By the computations shown as in \cite[Lemma 1.4]{Geng-Xu-26}, one can derive that
\begin{equation*}
\begin{aligned}
\|\nabla w_\varepsilon\|_{L^2(Q_2)}
&\lesssim \bigg\{\delta^{\beta}
+\frac{\varepsilon}{\delta^{d+3}}\Big(\dashint_{Q_{2/\varepsilon}} |\varpi|^{2}
\Big)^{\frac{1}{2}}
+\frac{\varepsilon^2}{\delta^{d+4}}
\Big(\dashint_{Q_{2/\varepsilon}}|\tilde{\sigma}_{(d+1)}|^2\Big)^{\frac{1}{2}}
\bigg\}\|\nabla u_\varepsilon\|_{L^2(Q_3)},
\end{aligned}
\end{equation*}
where $\varpi_j := (\bar{\phi}_j,\bar{\sigma}_j,\nabla\tilde{\sigma}_{(d+1)j})$ is known as the extended correctors, and $0<\beta\ll 1$ is given by the corresponding Meyers-type estimate. 
Using Poincare's inequality and Caccioppoli's inequality (see \cite[Lemma 5.3]{Geng-Xu-26}), we can rewrite
the above estimate as
\begin{equation*}
\begin{aligned}
\|w_\varepsilon\|_{L^2(Q_2)}
&\lesssim \bigg\{\delta^{\beta}
+\frac{\varepsilon}{\delta^{d+3}}\Big(\dashint_{Q_{2/\varepsilon}} |\varpi|^{2}
\Big)^{\frac{1}{2}}
+\frac{\varepsilon^2}{\delta^{d+4}}
\Big(\dashint_{Q_{2/\varepsilon}}|\tilde{\sigma}_{(d+1)}|^2\Big)^{\frac{1}{2}}
\bigg\}\|u_\varepsilon\|_{L^2(Q_4)},
\end{aligned}
\end{equation*}
and (recalling that the notations $u_\varepsilon$ and $\varepsilon$ represent $\tilde{u}_{\varepsilon'}$ and $\varepsilon'$, respectively) this implies
\begin{equation}\label{f:2.1}
\begin{aligned}
\|\tilde{u}_{\varepsilon'}-\bar{u}_1\|_{L^2(Q_2)}
&\lesssim \bigg\{\delta^{\beta}
+\frac{\varepsilon'}{\delta^{d+3}}\Big(\dashint_{Q_{2/{\varepsilon'}}}
|\varpi|^{2}\Big)^{\frac{1}{2}}
+\frac{{\varepsilon'}^2}{\delta^{d+4}}
\Big(\dashint_{Q_{2/\varepsilon'}}|\tilde{\sigma}_{(d+1)}|^2\Big)^{\frac{1}{2}}
\bigg\}\|u_{\varepsilon'}\|_{L^2(Q_4)}.
\end{aligned}
\end{equation}

Rescaling back (i.e., one can set $\bar{u}_r:=\bar{u}_1(\cdot/r)$ and it satisfies $\big(\partial_t+\mathcal{L}_0\big)(\bar{u}_r) = F$ in $Q_{2r}$ with $\bar{u}_r = u_\varepsilon$ on $\partial_p  Q_{2r}$), the stated estimates $\eqref{f:2.1}$ can therefore be rewritten as
\begin{equation}\label{f:2.2}
\begin{aligned}
\Big(\dashint_{Q_r}|u_\varepsilon -\bar{u}_r|^2\Big)^{\frac{1}{2}}
\lesssim \bigg\{{\delta}^{\beta}+\frac{\varepsilon'}{{\delta}^{d+3}}
\Big(\dashint_{Q_{2/\varepsilon'}} \big|
\varpi\big|^{2}\Big)^{\frac{1}{2}}
+\frac{{\varepsilon'}^2}{{\delta}^{d+4}}
\Big(\dashint_{Q_{2/\varepsilon'}}
|\tilde{\sigma}_{(d+1)}|^2\Big)^{\frac{1}{2}}\bigg\}
\Big(\dashint_{Q_{4r}}|u_{\varepsilon}|^2\Big)^{\frac{1}{2}},
\end{aligned}
\end{equation}
where we mention that the multiplicative constant in \eqref{f:2.1} is scale-invariant.
Moreover, for any $\theta>0$ (to be determined later), one can define the stationary field $\chi_*$ (i.e., minimal radius, which plays the same role as the minimal radius $r_*$ defined in \cite{Gloria-Neukamm-Otto-20}) as follows:
\begin{equation}\label{min_rad-ap}
\chi_{*}(0):=
 \inf\Bigg\{l>0:
 \frac{1}{R}
  \Big(\dashint_{Q_{2R}}
 \big|(\bar{\phi},\bar{\sigma},\nabla\tilde{\sigma}_{(d+1)})\big|^{2}\Big)^{\frac{1}{2}}
+ \frac{1}{R^2}\Big(\dashint_{Q_{2R}}
 |\tilde{\sigma}_{(d+1)}|^{2}\Big)^{\frac{1}{2}}
 \leq \theta;~\forall R\geq l\Bigg\}\vee 1
\end{equation}
(omitting the center point, we briefly denote it by $\chi_*$),
where we prefer constants $R=2/\varepsilon'$.

By the definition of $\chi_*$, we obtain that
\begin{equation*}
\begin{aligned}
\Big(\dashint_{Q_r}|u_\varepsilon - \bar{u}_r|^2\Big)^{\frac{1}{2}}
&\lesssim \theta^{\frac{\beta}{d+4+\beta}}
\Big(\dashint_{Q_{4r}}|u_\varepsilon|^2\Big)^{\frac{1}{2}}
\leq \nu \Big(\dashint_{Q_{4r}}|u_\varepsilon|^2\Big)^{\frac{1}{2}}
\qquad \forall r\geq \varepsilon\chi_{*},
\end{aligned}
\end{equation*}
which proves the stated result $\eqref{pri:2.1px}$ by preferring
$\theta = (\nu/ C(\lambda,\Lambda,d))^{\frac{d+4+\beta}{\beta}}$, where one can fix $\delta = \theta^{\frac{1}{d+4+\beta}}$
by optimizing the size of
the right-hand side of $\eqref{f:2.2}$ according to the definition of $\chi_*$ given in $\eqref{min_rad-ap}$.
Finally, the arguments for $\eqref{pri:2.3px}$ have already been shown
in the proof of \cite[Lemma 1.4]{Geng-Xu-26}, and we omit the details
here. We complete the whole proof.
\end{proof}

\begin{proposition}[quenched H\"older estimates]\label{prop:2.1}
Let $\varepsilon\in(0,1]$, $d\geq 2$, and $\alpha\in(0,1)$. Suppose that the ensemble $\langle\cdot\rangle$
is stationary and ergodic. Given $F\in L^p(Q_{4R})$ with $p>\frac{d+2}{2-\alpha}$ and $R>0$, let $u_\varepsilon$ and $F$ be associated by the following equation:
\begin{equation}\label{pde:2.1x}
\big(\partial_t+\mathcal{L}_{\varepsilon}\big)(u_\varepsilon)
= F \quad \text{in}~~ Q_{4R}(z),
\end{equation}
and we introduce the following quantity to quantify the H\"older regularity, i.e.,
\begin{equation}\label{def:2.2}
H(r,\alpha,v) : = \inf_{c\in \mathbb{R}} \frac{1}{r^{\alpha}}\bigg\{\Big(\dashint_{Q_r(z)}
|v-c|^2\Big)^{\frac{1}{2}} +r^2\Big(\dashint_{Q_r(z)}|F|^p\Big)^{\frac{1}{p}}\bigg\}.
\end{equation}
Then, there exists a stationary random field $\chi_*$ (given as in Lemma $\ref{lemma:approxi}$) such that,  there holds the large-scale H\"older estimate
\begin{equation}\label{pri:2.2x}
H(r,\alpha,u_\varepsilon)
\lesssim_{\lambda,\Lambda,d,\alpha,p}
H(R,\alpha,u_\varepsilon)
\end{equation}
for any $\varepsilon\chi_*(z/\varepsilon)\leq r\leq R$.
Moreover, if the admissible coefficients satisfy
the local smoothness condition $\eqref{c:4}$, then
there exists a stationary random field
\begin{equation}\label{def:2.3-new*}
\mathcal{H}_*(z/\varepsilon):=\Big([a]_{C^{\sigma,\sigma/2}(Q_1(z/\varepsilon))}^{
\frac{1}{\sigma}(\frac{d}{2}+1)}+1\Big)
\chi_{*}^{\frac{d+2}{2}+\alpha}(z/\varepsilon)
\qquad \forall z\in\mathbb{R}^{d+1},
\end{equation}
such that the uniform H\"older estimate
\begin{equation}\label{pri:2.5}
\bigg(\dashint_{Q_r(z)}|u_\varepsilon-\dashint_{Q_r(z)}u_\varepsilon|^2
\bigg)^{\frac{1}{2}}
\lesssim_{\lambda,\Lambda,d,\alpha,p}
\mathcal{H}_*(z/\varepsilon)\Big(\frac{r}{R}\Big)^{\alpha}
\bigg\{
\Big(\dashint_{Q_R(z)}|u_\varepsilon|^2
\Big)^{\frac{1}{2}}+R^2\Big(\dashint_{Q_R(z)}|F|^p\Big)^{\frac{1}{p}}\bigg\}
\end{equation}
holds for any $0<r\leq R$. Also, the uniform $L^\infty$-estimate holds
\begin{equation}\label{pri:2.11-1}
\begin{aligned}
&|u_\varepsilon(z)|
\lesssim
\mathcal{H}_*(z/\varepsilon)\bigg\{\Big(\dashint_{Q_R(z)}|u_\varepsilon|^2\Big)^{\frac{1}{2}}
+ R^2\Big(\dashint_{Q_R(z)}|F|^p
\Big)^{\frac{1}{p}}\bigg\}.
\end{aligned}
\end{equation}
In addition, if the ensemble $\langle\cdot\rangle$ additionally satisfies the spectral gap condition $\eqref{c:3}$, for any $\beta\in[1,\infty)$, one can derive that
\begin{equation}\label{def:2.3-new}
 \big\langle\mathcal{H}_{*}^{\beta}(z/\varepsilon)\big\rangle^{1/\beta}
 \lesssim_{\lambda,\lambda_1,\lambda_2,d,\beta} 1.
\end{equation}
\end{proposition}

\begin{remark}\label{remark:2.1-new}
\emph{The estimates $\eqref{pri:2.2x}$, $\eqref{pri:2.5}$, and $\eqref{pri:2.11-1}$ are scaling-invariant. It means that all of the above estimates can be reduced to the case $R = 1$\footnote{The case $0<R<\varepsilon$ leads to $\varepsilon'=\varepsilon/R>1$ which is not in $(0,1]$. However, this only involves the small scale regularity theory. To see this, one can take
$\tilde{u}_{1}(y,s):=u_\varepsilon(\varepsilon y,\varepsilon^2 s) = u_\varepsilon(x,t)$. There holds the scaled equation:
    $\big(\partial_{s}-\nabla_y\cdot a(y,s)\nabla_y\big)\tilde{u}_{1} = \tilde{F}$ in $Q_{R/\varepsilon}\subset Q_{1}$, where $\tilde{F}=\varepsilon^2F(\varepsilon y,\varepsilon^2s)$. Although the estimate in such a case has nothing to do with the large scales, in order to achieve the same form, we still choose $\mathcal{H}_\varepsilon$ as the estimated constant.}.
Taking the estimate $\eqref{pri:2.5}$ as an example, set $x=Ry$ with $t=R^2s$, and $(y,s)\in Q_{4}(\hat{z})$ with $\hat{z}$
being obtained by performing the same parabolic scaling transformation on $z$. Let $\tilde{u}_{\varepsilon'}(y,s)=u_\varepsilon(x,t)$ with $\varepsilon'=\varepsilon/R$. Therefore, one can have the scaled equation:
    $\big(\partial_{s}-\nabla_y\cdot a^{\varepsilon'}\nabla_y\big)(\tilde{u}_{\varepsilon'}) = \tilde{F}$ in $Q_{4}$, where $\tilde{F}=R^2F(Ry,R^2s)$. According to the stated estimate $\eqref{pri:2.5}$, one can derive that
    \begin{equation}\label{f:2.100-new*}
  \bigg(\dashint_{Q_{\hat{r}}(\hat{z})}|\tilde{u}_{\varepsilon'}-\dashint_{Q_{\hat{r}}(\hat{z})}\tilde{u}_{\varepsilon'}|^2
\bigg)^{\frac{1}{2}}
\lesssim_{\lambda,d,\alpha,p}
\mathcal{H}_*(\hat{z}/\varepsilon')\hat{r}^{\alpha}
\bigg\{
\Big(\dashint_{Q_1(\hat{z})}|\tilde{u}_{\varepsilon'}|^2
\Big)^{\frac{1}{2}}+\Big(\dashint_{Q_1(\hat{z})}|\tilde{F}|^p\Big)^{\frac{1}{p}}\bigg\}
    \end{equation}
    holds for any $0<\hat{r}\leq 1$. Rescaling back, one can derive that
\begin{equation}\label{f:2.101-new*}
  \bigg(\dashint_{Q_{R\hat{r}}(z)}|u_{\varepsilon}-\dashint_{Q_{R\hat{r}}(z)}u_{\varepsilon}|^2
\bigg)^{\frac{1}{2}}
\lesssim_{\lambda,d,\alpha,p}
\mathcal{H}_*(z/\varepsilon)\hat{r}^{\alpha}
\bigg\{
\Big(\dashint_{Q_R(z)}|u_{\varepsilon}|^2
\Big)^{\frac{1}{2}}+R^2\Big(\dashint_{Q_R(z)}|F|^p\Big)^{\frac{1}{p}}\bigg\}
    \end{equation}
holds for $0<R\hat{r}\leq R$. Let $r=R\hat{r}$. Plugging this back into the above estimate leads to the stated estimate $\eqref{pri:2.5}$. Obviously, comparing the estimates
$\eqref{f:2.100-new*}$ and $\eqref{f:2.101-new*}$, the estimated constant essentially remains unchanged.}
\end{remark}

\medskip
\noindent
\textbf{Proof of Proposition $\ref{prop:2.1}$.}
In the proof, the large-scale estimate $\eqref{pri:2.2x}$ serves as the foundation for the uniform estimate $\eqref{pri:2.5}$, while estimate $\eqref{pri:2.11-1}$ can be regarded as a deduction of estimate $\eqref{pri:2.5}$. Although the proof strategy of $\eqref{pri:2.2x}$ can be found in Shen's work \cite{Shen-17}, we need to emphasize a subtle difference here. Our proof does not require a quantitative convergence rate, and therefore only the stationarity and ergodicity of the ensemble
$\langle\cdot\rangle$ are introduced for the large-scale estimates.  The entire proof is divided into three steps.

\textbf{Step 1.}
We start from the proof of the estimate $\eqref{pri:2.2x}$.
Let $\alpha\in (0,1)$ and $\varepsilon\chi_*\leq r\leq R$. For the case $R<\varepsilon\chi_{*}$, there is nothing to show. For the ease of the statement, it is fine to assume $z=0$. In view of
Lemma $\ref{lemma:approxi}$, there exists a weak solution $\bar{u}_{r}$ in the parabolic energy space satisfying
$\big(\partial_t+\mathcal{L}_0\big)(\bar{u}_r) = F$ in $Q_{2r}$ with $\bar{u}_r = u_\varepsilon$ on $\partial_p  Q_{2r}$.
By applying H\"older estimates for parabolic operators with constant coefficients (see e.g. \cite{Schlag-96}), for any $\vartheta\in (0,1/4)$, we have
\begin{equation}\label{f:2.7}
\begin{aligned}
H(\vartheta r,\alpha,\bar{u}_r)
&\lesssim_{\lambda,\Lambda,d,\alpha}
(\vartheta r)^{\bar{\alpha}-\alpha}
[\bar{u}_r]_{C^{\bar{\alpha},\bar{\alpha}/2}(Q_r)} + \vartheta^{2-\frac{d+2}{p}-\alpha}r^{2-\alpha}\Big(\dashint_{Q_r}|F|^p\Big)^{\frac{1}{p}}\\
&\lesssim_{\lambda,\Lambda,d,\alpha}
\vartheta^{\alpha_0}  H(2r,\alpha,u_\varepsilon)
\end{aligned}
\end{equation}
where $\alpha_0:=\min\{\bar{\alpha}-\alpha,2-\alpha-(d+2)/p\}$, and $\alpha<\bar\alpha<\min\{1,2-(d+2)/p\}$. 
This result combined with the approximation established in Lemma $\ref{lemma:approxi}$ leads to
\begin{equation}\label{f:2.8}
\begin{aligned}
H(\vartheta r,\alpha,u_\varepsilon)
&  \leq  (\vartheta r)^{-\alpha}\Big(\dashint_{Q_{\vartheta r}}
|u_\varepsilon-\bar{u}_r|^2\Big)^{\frac{1}{2}}
+ H(\vartheta r,\alpha,\bar{u}_r) \\
& \leq^{\eqref{pri:2.1px},\eqref{f:2.7}}  C_d\vartheta^{-\alpha-\frac{d+2}{2}}\nu
H(4r,\alpha,u_\varepsilon)
+C_{\lambda,\Lambda,d,\alpha}\vartheta^{\alpha_0}  H(4r,\alpha,u_\varepsilon) \leq \frac{1}{4} H(4r,\alpha,u_\varepsilon),
\end{aligned}
\end{equation}
in which we first fix $\vartheta\in(0,1/4)$ and then determine $\nu>0$ in the last inequality above.

Therefore, on account of the iteration formula $\eqref{f:2.8}$, for any $\rho\in [\varepsilon\chi_*,R/4]$ we can establish
\begin{equation*}
 \int_{\rho}^{R/4}
 H(\vartheta r,\alpha,u_\varepsilon) \frac{dr}{r}
 \leq \frac{1}{4}
 \int_{\rho}^{R/4}
 H(4r,\alpha,u_\varepsilon) \frac{dr}{r},
\end{equation*}
which, by changing variables, implies
\begin{equation*}
 \int_{\vartheta\rho}^{\vartheta R/4}
 H( r,\alpha,u_\varepsilon) \frac{dr}{r}
 \leq \frac{1}{3}
 \int_{\vartheta R/4}^{R}
 H(r,\alpha,u_\varepsilon) \frac{dr}{r}
 \leq C_{d,\vartheta}H(R,\alpha,u_\varepsilon).
\end{equation*}
By performing a routine computation we can derive that
\begin{equation*}
 H(r,\alpha,u_\varepsilon)\lesssim H(R,\alpha,u_\varepsilon)
\end{equation*}
holds for $\varepsilon\chi_*\leq r\leq R$, which gives the stated estimate $\eqref{pri:2.2x}$.


\textbf{Step 2.} We now turn to the estimate $\eqref{pri:2.5}$, and first introduce the following preparatory result $\eqref{f:2.6-new*}$. To see this, let $R':=\varepsilon\chi_{*}(z/\varepsilon)\wedge R$ and $l:=\varepsilon\wedge R$. Then, we have the following relationship:
\begin{equation}\label{f:2.5-new*}
0<l\leq R'\leq R; \qquad (R'/l)\leq \chi_{*}(z/\varepsilon).
\end{equation}
Now, for the case $R'=\varepsilon\chi_*$, one can derive from $\eqref{pri:2.2x}$ and $\eqref{f:2.5-new*}$ that
\begin{equation}\label{f:2.4-new*}
 H(R',\alpha,u_\varepsilon)\lesssim H(R,\alpha,u_\varepsilon).
\end{equation}
Obviously, it is not hard to see the case $R'=R$ leading to the stated estimate $\eqref{f:2.4-new*}$. Moreover, one can derive that
\begin{equation}\label{f:2.6-new*}
H(l,\alpha,u_\varepsilon)\lesssim \Big(\frac{R'}{l}\Big)^{\frac{d+2}{2}+\alpha}H(R',\alpha,u_\varepsilon)
\lesssim^{\eqref{f:2.5-new*},\eqref{f:2.4-new*}}\big[\chi_{*}(z/\varepsilon)\big]^{\frac{d+2}{2}+\alpha}H(R,\alpha,u_\varepsilon).
\end{equation}

Now, to address $0<r\leq R$, the demonstration is divided into: (I). $l/2\leq r\leq R$; (II). $0<r<l/2$.

Case (I). For $l/2\leq r\leq R$, there are two cases : (1) $r\geq R'$; (2) $r<R'$. Obviously,
the case $r\geq R'$, together with $\eqref{pri:2.2x}$, leads to $H(r,\alpha,u_\varepsilon)\lesssim H(R,\alpha,u_\varepsilon)$. For the case $r<R'$,
we have $(R'/r)\leq (2R'/l)$ and further derive that
\begin{equation*}
H(r,\alpha,u_\varepsilon)\lesssim \Big(\frac{R'}{r}\Big)^{\frac{d+2}{2}+\alpha}H(R',\alpha,u_\varepsilon)
\lesssim \Big(\frac{R'}{l}\Big)^{\frac{d+2}{2}+\alpha}H(R',\alpha,u_\varepsilon)
\lesssim\big[\chi_{*}(z/\varepsilon)\big]^{\frac{d+2}{2}+\alpha}H(R,\alpha,u_\varepsilon).
\end{equation*}
This together with the case (1) leads to
\begin{equation}\label{f:2.7-new*}
H(r,\alpha,u_\varepsilon)
\lesssim \big[\chi_{*}(z/\varepsilon)\big]^{\frac{d+2}{2}+\alpha}H(R,\alpha,u_\varepsilon),
\end{equation}
where we use the fact that $\chi_*(z/\varepsilon)\geq 1$.
For the ease of the statement, it is fine to assume
\begin{equation}\label{f:2.8-new*}
 \Big(\dashint_{Q_R(z)}|u_\varepsilon|^2\Big)^{\frac{1}{2}}
 + R^2\Big(\dashint_{Q_R(z)}|F|^p\Big)^{\frac{1}{p}} =1.
\end{equation}
One can derive that
\begin{equation}\label{f:2.13}
\begin{aligned}
 \dashint_{Q_r(z)}\big|u_\varepsilon-\dashint_{Q_r(z)}u_\varepsilon\big|^2
& \leq r^{2\alpha}H^2(r,\alpha,u_\varepsilon) \\
& \lesssim^{\eqref{f:2.7-new*}}
 r^{2\alpha}
 \big[\chi_{*}(z/\varepsilon)\big]^{d+2+2\alpha}H^2(R,\alpha,u_\varepsilon)
 \lesssim
 \Big(\frac{r}{R}\Big)^{2\alpha}\big[\chi_{*}(z/\varepsilon)\big]^{d+2+2\alpha}.
\end{aligned}
\end{equation}

Case (II). For $0<r<l/2$, this is the small scales. By applying classical H\"older estimates (the accurate dependence on the coefficient $a$ can be found by a covering technique (see e.g. \cite[Lemma A.3]{Josien-Otto-22})), we obtain that
\begin{equation*}
\begin{aligned}
&  \dashint_{Q_r(z)}|u_\varepsilon-\dashint_{Q_r(z)}u_\varepsilon|^2
\lesssim_{d} r^{2\alpha}[u_\varepsilon]_{C^{\alpha,\alpha/2}(Q_{l/2}(z))}^2\\
&\lesssim
\Big([a]_{C^{\sigma,\sigma/2}(Q_1(z/\varepsilon))}^{\frac{2}{\sigma}(\frac{d}{2}+1)}+1\Big)\Big(\frac{r}{l}\Big)^{2\alpha}\bigg\{
\dashint_{Q_{l}(z)}|u_\varepsilon-c_{l}|^2
+ l^4\Big(\dashint_{Q_{l}(z)}|F|^{p}\Big)^{\frac{2}{p}}\bigg\}\\
&\lesssim^{\eqref{f:2.13}}
\Big([a]_{C^{\sigma,\sigma/2}(Q_1(z/\varepsilon))}^{\frac{2}{\sigma}(\frac{d}{2}+1)}+1\Big)
\big[\chi_*(z/\varepsilon)\big]^{d+2+2\alpha}\Big(\frac{r}{R}\Big)^{2\alpha}\\
&\qquad\qquad\qquad+
\Big([a]_{C^{\sigma,\sigma/2}(Q_1(z/\varepsilon))}^{\frac{2}{\sigma}(\frac{d}{2}+1)}+1\Big)
\Big(\frac{r}{R}\Big)^{2\alpha}
\Big(\frac{l}{R}\Big)^{4-\frac{2(d+2)}{p}-2\alpha}
R^4\bigg(\dashint_{Q_R(z)}|F|^p\bigg)^{\frac{2}{p}}
\lesssim \mathcal{H}_{*}^2(z)\Big(\frac{r}{R}\Big)^{2\alpha}
\end{aligned}
\end{equation*}
by setting $\mathcal{H}_*(z/\varepsilon):=
\big([a]_{C^{\sigma,\sigma/2}(Q_1(z/\varepsilon))}^{
\frac{1}{\sigma}(\frac{d}{2}+1)}+1\big)
\chi_{*}^{\frac{d+2}{2}+\alpha}(z/\varepsilon)$,
where $c_{l}:=\dashint_{Q_{l}(z)}u_\varepsilon$ and we also employ the fact that
$(l/R)\leq 1$ and $p>\frac{d+2}{2-\alpha}$ for the last inequality above.
Moreover, when $\langle\cdot\rangle$  satisfies $\eqref{c:3}$ and $\eqref{c:4}$, we can infer that
$\mathcal{H}_{*}$ is a stationary random field, and
$\big\langle\mathcal{H}_{*}^{\beta}(z/\varepsilon)\big\rangle\lesssim_{\lambda,\Lambda,\lambda_1,\lambda_2,\beta} 1$.
Consequently, combining the cases (I) and (II) leads to
the desired estimate $\eqref{pri:2.5}$.

\textbf{Step 3.} Arguments for $\eqref{pri:2.11-1}$.
Let $R_j:=2^{-j}R$ and $U_j:=\dashint_{Q_{R_j}(z)}u_\varepsilon$. Also, it is fine to assume $\eqref{f:2.8-new*}$.
One can start from the following observation:
\begin{equation}\label{f:2.110-new}
 U_{\infty} - U_{0} = \sum_{j=0}^{\infty}\big(U_{j+1}-U_{j}\big)
\end{equation}
On the other hand, we can derive that
\begin{equation}\label{f:2.111-new}
\begin{aligned}
\big|U_{j+1}-U_{j}\big|
&= \Big|\dashint_{Q_{R_{j+1}}(z)}u_\varepsilon
-\dashint_{Q_{R_{j}}(z)}u_\varepsilon\Big|
\leq \dashint_{Q_{R_{j+1}}(z)}\big|u_\varepsilon
-\dashint_{Q_{R_j}(z)}u_\varepsilon\big|\\
&\leq 2^{d+2}\Big(\dashint_{Q_{R_j}(z)}\big|u_\varepsilon-\dashint_{Q_{R_j}(z)}u_\varepsilon\big|^2\Big)^{\frac{1}{2}}
\lesssim^{\eqref{pri:2.5}} \mathcal{H}_{*}(z/\varepsilon)
\Big(\frac{R_j}{R}\Big)^{\alpha} = \mathcal{H}_{*}(z/\varepsilon)
2^{-j\alpha}.
\end{aligned}
\end{equation}
Combining the above estimates $\eqref{f:2.110-new}$ and
$\eqref{f:2.111-new}$, we obtain that
\begin{equation*}
|u_\varepsilon(z)| \leq \dashint_{Q_R(z)}|u_\varepsilon|+\sum_{j=0}^{\infty}\big|U_{j+1}-U_{j}\big|
\lesssim 1 + \mathcal{H}_{*}(z/\varepsilon)\sum_{j=0}^\infty2^{-j\alpha}
\lesssim \mathcal{H}_{*}(z/\varepsilon),
\end{equation*}
which gives the desired estimate $\eqref{pri:2.11-1}$.
We have completed the whole proof.
\qed

\begin{proposition}[Quenched Schauder estimates]\label{prop:2.2-new}
Let $\varepsilon\in(0,1]$ and $\alpha\in(0,1)$. Suppose that the ensemble $\langle\cdot\rangle$ is stationary and ergodic. Let $u_\varepsilon$ be the solution to the following equation:
$(\partial_t+\mathcal{L}_{\varepsilon})(u_\varepsilon) = 0$
in $Q_{2R}$ with $R>0$. For all $\xi\in\mathbb{R}^d$, we introduce
the following quantity to quantify Schauder regularity
\begin{equation*}
\text{Exc}(g;r):=\inf_{\xi\in\mathbb{R}^d}\dashint_{Q_r}
|g-(\nabla\phi_{\xi}+\xi)|^2, \quad\text{and}\quad \phi_{\xi}:=\phi_j\xi_j.
\end{equation*}
Then,
there hold the large-scale Schauder estimates
\begin{subequations}
\begin{align}
\dashint_{Q_r}\big|\nabla u_\varepsilon\big|^2
&\lesssim_{\lambda,\Lambda,d}
\dashint_{Q_R}\big|\nabla u_\varepsilon\big|^2; \label{pri:2.1-Lip}\\
\text{Exc}(\nabla u_\varepsilon, r)
&\lesssim_{\lambda,\Lambda,d,\alpha} \Big(\frac{r}{R}\Big)^{2\alpha}\text{Exc}(\nabla u_\varepsilon, R)
\label{pri:2.1x-exc}
\end{align}
\end{subequations}
for any $\varepsilon\chi_{*}\leq r\leq R$. Moreover, if the admissible coefficients satisfy
the local smoothness condition $\eqref{c:4}$, then
there exists a stationary random field (similar to its counterpart stated in Proposition $\ref{prop:2.1}$)
\begin{equation}\label{f:2.2-new**}
C_*(z/\varepsilon):=\Big([a]_{C^{\sigma,\sigma/2}(Q_1(z/\varepsilon))}^{
\frac{1}{\sigma}(\frac{d}{2}+2)}+1\Big)
\chi_{*}^{\frac{d+2}{2}}(z/\varepsilon)
\qquad \forall z\in\mathbb{R}^{d+1},
\end{equation}
such that the uniform Lipschitz estimate
\begin{equation}\label{pri:20}
\Big(\dashint_{Q_r(z)}|\nabla u_\varepsilon|^2
\Big)^{\frac{1}{2}}
\lesssim_{\lambda,\Lambda,d}
C_*(z/\varepsilon)
\Big(\dashint_{Q_R(z)}|\nabla u_\varepsilon|^2
\Big)^{\frac{1}{2}}
\end{equation}
holds for all $0<r\leq R$. In addition, if the ensemble $\langle\cdot\rangle$ satisfies the spectral gap condition $\eqref{c:3}$, then we have
\begin{equation}\label{f:2.2-new*}
 \big\langle C_{*}^{\beta}(z/\varepsilon)\big\rangle^{1/\beta}
 \lesssim_{\lambda,\Lambda,\lambda_1,\lambda_2,d,\beta} 1
 \qquad\forall\beta\in[1,\infty),
\end{equation}
and for any $1\leq p<\bar{p}<\infty$ and $0<r\leq R$, there holds
\begin{equation}\label{pri:2.10}
 \big\langle|\nabla u_\varepsilon(z)|^p\big\rangle^{1/p}\lesssim_{\lambda,\Lambda,\lambda_1,\lambda_2,d,p,\bar{p}}\bigg\langle\Big(\dashint_{Q_r(z)}|\nabla u_\varepsilon|^2
\Big)^{\frac{\bar{p}}{2}}\bigg\rangle^{\frac{1}{\bar{p}}}.
\end{equation}
\end{proposition}

\begin{lemma}\label{lemma:2.1-new}
Let $\varepsilon\in(0,1]$.
For any small parameter $\nu>0$, let $\chi_{*}$ be given as in Lemma $\ref{lemma:approxi}$. Assume that the ensemble is stationary and ergodic.
Suppose that $u_\varepsilon$ is the solution to
the following equation: $(\partial_t+\mathcal{L}_\varepsilon)(u_\varepsilon) = 0$ in $Q_{2R}$ with $R\geq 4\varepsilon\chi_{*}$. Then, there exists
$\xi\in\mathbb{R}^d$ such that
\begin{equation}\label{pri:4.1-x}
\dashint_{Q_r}\big|\nabla u_\varepsilon-(\nabla\phi_\xi+\xi)\big|^2
\lesssim \bigg[\nu^2\Big(\frac{R}{r}\Big)^{d+2}
+\Big(\frac{r}{R}\Big)^{2}\bigg]\dashint_{Q_R}|\nabla u_\varepsilon|^2
\end{equation}
holds for any $r\in[\varepsilon\chi_*,R/4]$.
\end{lemma}

\begin{proof}
Let $\bar{u}$ satisfy the auxiliary equations: $(\partial_t+\mathcal{L}_0)(\bar{u})=0$
in $Q_R$ and $\bar{u} = u_\varepsilon$ on $\partial_pQ_R$. Then, by the triangle inequality and taking $\xi=\nabla\bar{u}(0)$, we have
\begin{equation*}
\begin{aligned}
\dashint_{Q_r}\big|\nabla u_\varepsilon
&-\big((\nabla\phi_i)^{\varepsilon}+e_i\big)\xi_i\big|^2
\lesssim
\dashint_{Q_r}\big|\nabla u_\varepsilon-\big((\nabla\phi_i)^{\varepsilon}+e_i\big)\partial_i\bar{u}\big|^2
+ \dashint_{Q_r}\big|\big((\nabla\phi_i)^{\varepsilon}+e_i\big)(\xi_i-\partial_i\bar{u})\big|^2\\
&\lesssim \Big(\frac{R}{r}\Big)^{d+2}\dashint_{Q_{R/4}}
\big|\nabla u_\varepsilon-((\nabla\phi_i)^{\varepsilon}+e_i)\partial_i\bar{u}\big|^2
+r^{2}[\nabla\bar{u}]^{2}_{C^{1,1/2}(Q_r)}\Big(\dashint_{Q_{r/\varepsilon}}|\nabla\phi|^2 + 1\Big)\\
&\lesssim^{\eqref{pri:2.2px}} \nu^2\Big(\frac{R}{r}\Big)^{d+2}\dashint_{Q_R}|\nabla u_\varepsilon|^2
+ \Big(\frac{r}{R}\Big)^{2}\dashint_{Q_{R/2}}|\nabla \bar{u}|^2\\
&\lesssim \bigg[\nu^2\Big(\frac{R}{r}\Big)^{d+2}
+\Big(\frac{r}{R}\Big)^{2}\bigg]\dashint_{Q_R}|\nabla u_\varepsilon|^2,
\end{aligned}
\end{equation*}
where we employ classical Schauder's estimate (see e.g.,
\cite[Lemma 2]{Schlag-96}) in the third inequality and
the energy estimate in the last one for $\bar{u}$.
This ends the proof.
\end{proof}

\medskip
\noindent
\textbf{Proof of Proposition $\ref{prop:2.2-new}$.}
The results $\eqref{pri:2.1-Lip}$ and $\eqref{pri:2.1x-exc}$ can be found in the literature \cite{Armstrong-Bordas-Mourrat-18,Bella-Chiarini-Fehrman-19}.
Based upon the new definition of the minimal radius $\chi_{*}$ introduced in Lemma $\ref{lemma:approxi}$, we can provide a  relatively brief proof for the readers' convenience, and the underlying idea of this proof can be traced back to its elliptic counterpart developed by Gloria et al. \cite{Gloria-Neukamm-Otto-20,Avellaneda-Lin-87}.
The proof is divided into the three steps.

\textbf{Step 1.}
Arguments for $\eqref{pri:2.1x-exc}$. We first mention that,
for the case $R<\varepsilon\chi_{*}$, there is nothing to show.
Note that $u_\varepsilon-(\phi_i + x_i)$ is still a solution to
the equation: $(\partial_t+\mathcal{L}_\varepsilon)(u_\varepsilon) = 0$ in $Q_{2R}$. It follows from Lemma $\ref{lemma:2.1-new}$ that
\begin{equation*}
 \text{Exc}(\nabla u_\varepsilon,r)
 \lesssim^{\eqref{pri:4.1-x}}
 \bigg[\nu^2\Big(\frac{R}{r}\Big)^{d+2}
+\Big(\frac{r}{R}\Big)^{2}\bigg]\dashint_{Q_R}\big|\nabla u_\varepsilon
-\big((\nabla\phi_i)^{\varepsilon}+e_i\big)\xi_i'\big|^2
\qquad \forall \xi'\in\mathbb{R}^{d},
\end{equation*}
which implies that
\begin{equation*}
 \text{Exc}(\nabla u_\varepsilon,r)
\lesssim \bigg[\nu^2\Big(\frac{R'}{r}\Big)^{d+2}
+\Big(\frac{r}{R'}\Big)^{2}\bigg]\text{Exc}(\nabla u_\varepsilon,R')
\qquad \varepsilon\chi_{*}\leq r<R'\leq R/4.
\end{equation*}

Let $\alpha\in(0,1)$ be fixed, and
set $r=\theta R'$ with $\theta\in(0,1)$ satisfying
$C\theta^2\leq (1/2)\theta^{2\alpha}$. Then, one may choose $\nu$ small enough such that $C\nu^2\theta^{-d-2}\leq (1/2)\theta^{2\alpha}$. Therefore, the above estimate gives us
\begin{equation*}
\text{Exc}(\nabla u_\varepsilon,\theta R')
\leq \theta^{2\alpha}	\text{Exc}(\nabla u_\varepsilon,R')
\qquad R'\in[\varepsilon\chi_*,R/4].
\end{equation*}
For any $\theta^{n-1} R\geq \varepsilon\chi_*$, by iterations one can derive that
\begin{equation*}
\text{Exc}(\nabla u_\varepsilon,\theta^n R)
\leq \big(\theta^n\big)^{2\alpha}	\text{Exc}(\nabla u_\varepsilon,R).
\end{equation*}
Thus, for any $\varepsilon\chi_*\leq r\leq R$, one can choose
$n\in\mathbb{N}$ such that $\theta^{n+1} R< r \leq \theta^{n}R$, and there holds
\begin{equation*}
\text{Exc}(\nabla u_\varepsilon,r)
\lesssim_\theta \Big(\frac{r}{R}\Big)^{2\alpha}
\text{Exc}(\nabla u_\varepsilon,R),
\end{equation*}
which gives the stated estimate  $\eqref{pri:2.1x-exc}$.

\textbf{Step 2.} Arguments for $\eqref{pri:2.1-Lip}$.
The key ingredient is to control the slope $\xi_\rho\in\mathbb{R}^d$. By the same idea as shown by Gloria-Neukamm-Otto \cite[\textcolor{red}{p. 132}]{Gloria-Neukamm-Otto-20}, for any $\varepsilon\chi_{*}\leq \rho\leq R'\leq R$ we have
\begin{equation}\label{f:4.1-x}
\big|\xi_{\rho}-\xi_{R'}\big|^2\lesssim\text{Exc}(\nabla u_\varepsilon, R'),
\end{equation}
which can be established by the case $\rho\leq R'\leq 2\rho$, where $\xi_\rho\in\mathbb{R}^d$ is such that
\begin{equation*}
\dashint_{Q_\rho}\big|\nabla u_\varepsilon-\big(\nabla\phi_{\xi_\rho}+\xi_{\rho}\big)\big|^2 = \text{Exc}(\nabla u_\varepsilon,\rho).
\end{equation*}

On the one hand, one can derive that
\begin{equation}\label{f:4.2-x}
 \dashint_{Q_r}|\nabla u_\varepsilon|^2
 \lesssim \text{Exc}(\nabla u_\varepsilon,r)
 +|\xi_r|^2,
\end{equation}
where we employ the non-degeneracy fact
\begin{equation}\label{f:4.4-x}
 \dashint_{Q_r}|\nabla\phi_{\xi}+\xi|^2 \sim |\xi|^2
 \qquad \text{as}\quad r\geq \varepsilon\chi_*.
\end{equation}
On the other hand, noting the definition of $\xi_R$, one can derive that
\begin{equation}\label{f:4.3-x}
\text{Exc}(\nabla u_\varepsilon,R)+|\xi_R|^2
\lesssim^{\eqref{f:4.4-x}}
\text{Exc}(\nabla u_\varepsilon,R)
+\dashint_{Q_R}|\nabla u_\varepsilon|^2
\lesssim \dashint_{Q_R}|\nabla u_\varepsilon|^2.
\end{equation}

Combining the estimates $\eqref{f:4.1-x}$, $\eqref{f:4.2-x}$, and $\eqref{f:4.3-x}$, we obtain that
\begin{equation*}
\dashint_{Q_r}|\nabla u_\varepsilon|^2\lesssim \dashint_{Q_R}|\nabla u_\varepsilon|^2,
\end{equation*}
provided $R\geq r\geq \varepsilon\chi_{*}$, which completes the proof of
the estimate $\eqref{pri:2.1-Lip}$.

Finally, we use the same method as in $\eqref{pri:2.5}$ to obtain the random field $C_{*}$ appearing in $\eqref{pri:20}$, and therefore the desired estimate $\eqref{pri:2.10}$. We do not reproduce the details here, and this ends the whole proof.
\qed

\section{Proof of Theorem $\ref{thm:0-new}$}\label{sec:3}

\noindent
Since the estimated constants in $\eqref{pri:1.1-new*}$ and $\eqref{pri:1.1-new*2}$ are scale-invariant, we may assume $\varepsilon = 1$ in this section. The fundamental solution $\Gamma_1$ and $u_1$ (which can naturally be associated with $\Gamma_1$ by an initial problem) will be simply denoted as $\Gamma$ and $u$, respectively. Let $\mathcal{S}(\mathbb{R}^d)$ denote the Schwartz class on $\mathbb{R}^d$.

\begin{lemma}[Quenched-type Davies' estimates]
\label{lemma:3.1-newnew}
For any $\gamma>0$, let $\psi$ be a Lipschitz function such that
\begin{equation}\label{f:2.1-x}
|\nabla\psi(x)|\leq \gamma \qquad\forall x\in\mathbb{R}^d.
\end{equation}
Let $s,t\in\mathbb{R}$ be such that $t>s$ and $x\in\mathbb{R}^d$.
We introduce the definition of $P_{s\to t}^{\psi}$ on $\mathcal{S}(\mathbb{R}^d)$ as follows:
\begin{equation}\label{def:3.1}
	 P_{s\to t}^{\psi}f(x)
	:=\exp\{\psi(x)\}\int_{\mathbb{R}^d}dy\Gamma(x,t;y,s)\exp\{-\psi(y)\}f(y).
\end{equation}
Then, there exists \textcolor{red}{a} random variable $\mathcal{H}_*(z)$ given as in $\eqref{def:2.3-new}$ with $z:=(x,t)$ such that
\begin{equation}\label{pri:2.1-x}
\big|P_{s\to t}^{\psi}f(x)\big|\lesssim_{\lambda,\Lambda,d}
\mathcal{H}_{*}(z)(t-s)^{-\frac{d}{4}}\exp\big\{\kappa\gamma^2(t-s)\big\}\|f\|_{L^2(\mathbb{R}^d)}
\end{equation}
where $\kappa:=C\Lambda^2/\lambda$, and $\lambda,\Lambda$ are given in $\eqref{c:1}$, respectively.
\end{lemma}


\begin{proof}
This proof is the most crucial step to employ  the Hofmann–Kim/Davies method \cite{Davies-87,Fabes-Stroock-86,Hofmann-Kim-04} in its corresponding version in the quenched sense. Although we still employed the framework of energy estimates, we emphasize the application of the quenched-type $L^\infty$-estimate established in Proposition $\ref{prop:2.1}$.
    For the reader's convenience, we provide a proof here.

    Let $(s,t)\subset\mathbb{R}$ and $x\in\mathbb{R}^d$ be arbitrarily fixed, and $R:=\sqrt{t-s}/4$.
The definition of $P_{s\to t}^{\psi}f$ leads to the relationship: $e^{\psi(x)}u(x,t) = P_{s\to t}^{\psi}f(x)$, where $u$ satisfies the following initial problem:
\begin{equation}\label{pde:3.1}
\left\{\begin{aligned}
\big(\partial_t - \nabla\cdot a\nabla\big)u &= 0 &\quad&\textbf{in}\quad \mathbb{R}^d\times(s,\infty);\\
u|_{t=s} &= e^{-\psi}f
&\quad&\textbf{on}\quad \mathbb{R}^d.
\end{aligned}\right.
\end{equation}
 This, together with Proposition $\ref{prop:2.1}$, yields
\begin{equation}\label{f:3.1-new*}
\begin{aligned}
e^{-\psi(x)}|P_{s\to t}^{\psi}f(x)|
 &=|u(x,t)| \\
 &\lesssim^{\eqref{pri:2.11-1}}
 \mathcal{H}_{*}(z)R^{-\frac{d+2}{2}}
 \bigg(\int_{t-R^2}^{t+R^2}d\tau\int_{B_R(x)}dx'|u(x',\tau)|^2\bigg)^{\frac{1}{2}}\\
 &=\mathcal{H}_{*}(z)R^{-\frac{d+2}{2}}
 \bigg(\int_{t-R^2}^{t+R^2}d\tau\int_{B_R(x)}dx' e^{-2\psi(x')}|P_{s\to \tau}^{\psi}f(x')|^2\bigg)^{\frac{1}{2}}.
\end{aligned}
\end{equation}
On account of classical weighted Davies' estimate
(see e.g. \cite[(1.15)]{Hofmann-Kim-04}), we have
\begin{equation}\label{pri:2.2-x}
\int_{\mathbb{R}^d}\big|P_{s\to t}^{\psi}f\big|^2
\lesssim_{\lambda,\Lambda,d} e^{\frac{C\Lambda^2\gamma^2}{\lambda}(t-s)}\int_{\mathbb{R}^d}|f|^2.
\end{equation}
It follows from the estimates $\eqref{f:3.1-new*}$ and $\eqref{pri:2.2-x}$ that
\begin{equation}\label{pri:2.3-x}
\begin{aligned}
|P_{s\to t}^{\psi}f(x)|
&\lesssim \mathcal{H}_{*}(z)R^{-\frac{d+2}{2}}
 \bigg(\int_{t-R^2}^{t+R^2}d\tau\int_{B_R(x)}dx' e^{2\psi(x)-2\psi(x')}|P_{s\to \tau}^{\psi}f(x')|^2\bigg)^{\frac{1}{2}}\\
&\lesssim \mathcal{H}_{*}(z)R^{-\frac{d+2}{2}}e^{\gamma R}
\bigg(\int_{t-R^2}^{t+R^2}d\tau\int_{\mathbb{R}^d}\big|P_{s\to\tau}^{\psi}f\big|^2\bigg)^{\frac{1}{2}}
\lesssim \frac{\mathcal{H}_{*}(z)e^{\gamma R + \kappa\gamma^2R^2}}{(t-s)^{\frac{d}{4}}}\|f\|_{L^2(\mathbb{R}^d)}
\end{aligned}
\end{equation}
where we take $\kappa = C\Lambda^2/\lambda$. Using Young's inequality, the above estimate further gives
the desired estimate $\eqref{pri:2.1-x}$. This ends the proof.
\end{proof}

Before showing the proof of Theorem $\ref{thm:0-new}$,
we first introduce the definition of
the backward adjoint fundamental solutions $\Gamma^*_{\varepsilon}$ and $\Gamma^*_0$ as follows. Let $-\partial_t-\nabla\cdot (a^\varepsilon)^*\nabla$
and $-\partial_t-\nabla\cdot\bar{a}^{*}\nabla$ be the space-time adjoint operators, denoted by $-\partial_t+\mathcal{L}_\varepsilon^{*}$ and $-\partial_t+\mathcal{L}_0^{*}$, respectively, where $a^{*}$ is the transpose of the coefficient $a$. For any fixed point $(x,t)\in\mathbb{R}^{d+1}$, they satisfy
\begin{equation}\label{pde:100-new-*}
\left\{\begin{aligned}
&\big(-\partial_s+\mathcal{L}_\varepsilon^*\big)\Gamma_\varepsilon^{*}(\cdot,\cdot;x,t) = 0
\quad \text{in}\quad \mathbb{R}^d\times (-\infty,t);\\
& \lim_{s\uparrow t}\Gamma_\varepsilon^{*}(\cdot,s;x,t) =\delta_x
\quad \text{on} \quad \mathbb{R}^d;\\
& \Gamma_{\varepsilon}^{*}(y,s;x,t) = 0
\quad \forall y\in \mathbb{R}^d, \quad \text{if} \quad s>t;
\end{aligned}\right.
\end{equation}
and
\begin{equation}\label{pde:101-new-*}
\left\{\begin{aligned}
&\big(-\partial_s+\mathcal{L}_0^{*}\big)\Gamma_0^{*}(\cdot,\cdot;x,t) = 0
\quad \text{in}\quad \mathbb{R}^d\times (-\infty,t);\\
& \lim_{s\uparrow t}\Gamma_0^{*}(\cdot,s;x,t) =\delta_x
\quad \text{on} \quad \mathbb{R}^d;\\
& \Gamma_{0}^{*}(y,s;x,t) = 0
\quad \forall y\in \mathbb{R}^d, \quad \text{if} \quad s>t,
\end{aligned}\right.
\end{equation}
respectively. We mention that for any $(x,t)\not=(y,s)$ in $\mathbb{R}^{d+1}$,
there holds\footnote{If we consider the system of equations, we need an additional transpose therein, i.e., $\Gamma_\varepsilon^{*}(y,s;x,t)^T = \Gamma_\varepsilon(x,t;y,s)$.  }
\begin{equation}\label{id:3.1}
\Gamma_\varepsilon^{*}(y,s;x,t) = \Gamma_\varepsilon(x,t;y,s)
\quad \text{and}\quad \Gamma_0^{*}(y,s;x,t)=\Gamma_0(x,t;y,s)
\end{equation}
(see, e.g., \cite[Lemma 3.5]{Cho-Dong-Kim-08}), where $\Gamma_\varepsilon$ and $\Gamma_0$
are given in $\eqref{pde:100-new}$ and $\eqref{pde:101-new}$, respectively.

\bigskip
\noindent
\textbf{Proof of Theorem $\ref{thm:0-new}$.}
The basic idea is to extend the Hofmann–Kim/Davies method \cite{Davies-87,Fabes-Stroock-86,Hofmann-Kim-04} to its corresponding version in the quenched sense. Lemma
$\ref{lemma:3.1-newnew}$ provides the first step of this method.
The remaining proof mainly involves using the duality method to obtain the desired estimates $\eqref{pri:1.1-new*}$ and therefore $\eqref{pri:2.10-1}$.
To address the decay estimates at the gradient level (i.e., $\eqref{pri:1.1-new*2}$), we appeal to
uniform Lipschitz estimates established in Proposition $\ref{prop:2.2-new}$.  The whole proof is divided into two steps.
For the ease of the statement,  let $(s,t)\subset\mathbb{R}$ and $x\in\mathbb{R}^d$ be arbitrarily fixed, and $R:=\sqrt{t-s}/4$.

\textbf{Step 1.} Arguments for $\eqref{pri:1.1-new*}$.
Recalling the definition of $P_{s\to t}^{\psi}f$ in Lemma $\ref{lemma:3.1-newnew}$, applying the duality argument to $\eqref{pri:2.1-x}$ leads to
\begin{equation}\label{f:3.3-new*}
\bigg(\int_{\mathbb{R}^d}dy\big|e^{\psi(x)-\psi(y)}\Gamma(x,t;y,s)\big|^2\bigg)^{\frac{1}{2}}\lesssim
\mathcal{H}_{*}(z)R^{-\frac{d}{2}}
e^{\kappa\gamma^2R^2}.
\end{equation}
Then, by the identity (see e.g. \cite[Theorem 2.11]{Cho-Dong-Kim-08} or \cite[(1.27)]{Hofmann-Kim-04})
\begin{equation*}
\Gamma(x,t;y,s) = \int_{\mathbb{R}^d}dz\Gamma(x,t;z,r)\Gamma(z,r;y,s) \qquad \forall x,y\in\mathbb{R}^d, \quad\text{and}~ r\in(s,t),
\end{equation*}
we have the following identity:
\begin{equation}\label{f:3.4-new*}
\begin{aligned}
 e^{\psi(x)-\psi(y)}\Gamma(x,t;y,s)
 &=\int_{\mathbb{R}^d}dz e^{\psi(x)-\psi(z)}
 \Gamma(x,t;z,r)e^{\psi(z)-\psi(y)}\Gamma(z,r;y,s)\\
 &=^{\eqref{id:3.1}}\int_{\mathbb{R}^d}dz e^{\psi(x)-\psi(z)}
 \Gamma(x,t;z,r)e^{\psi(z)-\psi(y)}\Gamma^*(y,s;z,r)
\end{aligned}
\end{equation}
where $\Gamma^*(y,s;z,r)$ is the fundamental solution associated with the adjoint operator defined in $\eqref{pde:100-new-*}$. Therefore, one can introduce
$(P_{s\to r}^{\psi})^{*}f(y):=\int_{\mathbb{R}^d}dze^{\psi(z)-\psi(y)}f(z)\Gamma^{*}(y,s;z,r)$ with $r\in(s,t)$, and a similar computation as given for $\eqref{pri:2.1-x}$ leads to
\begin{equation*}
\big|(P_{s\to r}^{\psi})^{*}f(y)\big|\lesssim_{\lambda,\Lambda,d}\tilde{\mathcal{H}}_{*}(z')(r-s)^{-\frac{d}{4}}e^{\kappa\gamma^2(r-s)}\|f\|_{L^2(\mathbb{R}^d)}
\qquad \text{with}\quad z'=(y,s),
\end{equation*}
where we mention that the solution to the backward equations satisfies the same regularity estimates as the solution to the forward evolution equation, and therefore, let $\tilde{\mathcal{H}}_*$ (determined by $a^*$ instead of $a$) represent the counterpart of $\mathcal{H}_{\varepsilon}$ for the uniform $L^\infty$-estimate associated with the adjoint operator $-\partial_t+\mathcal{L}_\varepsilon^{*}$.
By using the duality argument again, one can derive that
\begin{equation}\label{f:3.5-new*}
\bigg(\int_{\mathbb{R}^d}dz\big|e^{\psi(z)-\psi(y)}\Gamma^*(y,s;z,r)\big|^2\bigg)^{\frac{1}{2}}\lesssim
\tilde{\mathcal{H}}_{*}(z')(r-s)^{-\frac{d}{4}}e^{\kappa\gamma^2(r-s)}.
\end{equation}

Let $r=(s+t)/2$, and it follows from $\eqref{f:3.3-new*}$, $\eqref{f:3.4-new*}$, and $\eqref{f:3.5-new*}$ that
\begin{equation*}
\begin{aligned}
e^{\psi(x)-\psi(y)}\big|\Gamma(x,t;y,s)\big|
&\leq \bigg(\int_{\mathbb{R}^d}dz\big|e^{\psi(x)-\psi(z)}\Gamma(x,t;z,r)\big|^2\bigg)^{\frac{1}{2}}
\bigg(\int_{\mathbb{R}^d}dz\big|e^{\psi(z)-\psi(y)}\Gamma^*(y,s;z,r)\big|^2\bigg)^{\frac{1}{2}}\\
&\lesssim \mathcal{H}_{*}(z)\tilde{\mathcal{H}}_{*}(z')R^{-d}e^{\kappa\gamma^2R^2}.
\end{aligned}
\end{equation*}
By setting $\psi(x)=\gamma|x-y|$, one can derive from
the above estimate that
\begin{equation*}
|\Gamma(x,t;y,s)|
\lesssim \mathcal{H}_{*}(z)\tilde{\mathcal{H}}_{*}(z')R^{-d}e^{\kappa\gamma^2R^2-\gamma|x-y|}.
\end{equation*}
By choosing $\gamma=|x-y|/(2\kappa(t-s))$,
the same computation as given in
\cite[pp.486-487]{Hofmann-Kim-04} consequently leads to
the desired estimate $\eqref{pri:1.1-new*}$. On account of
the assumption $\eqref{c:3}$, the high-moment estimates for
$\mathcal{H}_*$ and $\tilde{\mathcal{H}}_*$ have been included in
$\eqref{def:2.3-new}$. By this, the annealed type Gaussian upper bound estimate $\eqref{pri:2.10-1}$ can follow from $\eqref{pri:1.1-new*}$ straightforwardly.

\textbf{Step 2.} Arguments for $\eqref{pri:1.1-new*2}$. Let $z:=(x,t)$, and let $z':=(y,s)$ be fixed. In view of the equations $\eqref{pde:100-new}$ and $\eqref{pde:3.1}$,
we have $(\partial_t+\mathcal{L}_1)u = 0$ in $Q_{2R}(z)$ and $e^{\psi(x)}u(x,t) = P_{s\to t}^{\psi}f(x)$.  Therefore,  there holds
\begin{equation}\label{f:3.6-new*}
\begin{aligned}
\big|\nabla P_{s\to t}f(x)\big|
&\leq^{\eqref{f:2.1-x}} |e^{\psi(x)}\nabla u(x,t)| + \gamma|e^{\psi(x)} u(x,t)|\\
&\lesssim^{\eqref{pri:2.3-x}} |e^{\psi(x)}\nabla u(x,t)| +
\mathcal{H}_{*}(z)R^{-\frac{d}{2}-1} e^{2\gamma R+\kappa\gamma^2R^2}\|f\|_{L^2(\mathbb{R}^d)},
\end{aligned}
\end{equation}
where we also employ the fact that $\gamma R e^{\gamma R}\leq e^{2\gamma R}$ in the last inequality. Then, we handle the first term in the right-hand side of $\eqref{f:3.6-new*}$. Applying the uniform Lipschitz estimate stated in Proposition $\ref{prop:2.2-new}$, we have
\begin{equation}
\begin{aligned}
&|e^{\psi(x)}
\nabla u(x,t)|
\lesssim^{\eqref{pri:20}}  C_{*}(z/\varepsilon) e^{\psi(x)}\Big(\dashint_{Q_{R/2}(z)}|\nabla u|^2\Big)^{\frac{1}{2}}\\
 &\lesssim \frac{C_*(z/\varepsilon)}{R^{2+\frac{d}{2}}}
 \bigg(\int_{t-R^2}^{t+R^2}d\tau\int_{B_R(x)}dx'
 e^{2\psi(x)-2\psi(x')}|P_{s\to \tau}^{\psi}f(x')|^2\bigg)^{\frac{1}{2}}\\
 &\lesssim C_{*}(z)R^{-\frac{d}{2}-2}e^{\gamma R}
\bigg(\int_{t-R^2}^{t+R^2}d\tau\int_{\mathbb{R}^d}\big|P_{s\to\tau}^{\psi}f\big|^2\bigg)^{\frac{1}{2}}
\lesssim^{\eqref{pri:2.2-x}} C_{*}(z)R^{-\frac{d}{2}-1}e^{\gamma R + \kappa\gamma^2R^2}\|f\|_{L^2(\mathbb{R}^d)},
\end{aligned}
\end{equation}
where we employ Caccioppoli's inequality (see, e.g. \cite[Lemma B.1]{Geng-Xu-26}) for the second inequality above.
This, together with $\eqref{f:3.6-new*}$ and the duality argument, gives
\begin{equation}\label{f:3.50-new*}
\bigg(\int_{\mathbb{R}^d}dy\big|e^{\psi(x)-\psi(y)}\nabla_x\Gamma(x,t;y,s)\big|^2\bigg)^{\frac{1}{2}}\lesssim
\mathcal{C}_{*}(z)R^{-\frac{d}{2}-1}
e^{\kappa\gamma^2R^2},
\end{equation}
where we set $\mathcal{C}_*(z):=C_*(z)\vee \mathcal{H}_{*}(z)$,
and it is stationary since both are stationary with respect to the same translation action.
The same computation as given for $\eqref{f:3.50-new*}$ also leads to
\begin{equation}\label{f:3.51-new*}
\bigg(\int_{\mathbb{R}^d}dx\big|e^{\psi(x)-\psi(y)}\nabla_y\Gamma^{*}(y,s;x,t)\big|^2\bigg)^{\frac{1}{2}}\lesssim
\tilde{\mathcal{C}}_{*}(z')R^{-\frac{d}{2}-1}
e^{\kappa\gamma^2R^2},
\end{equation}
where we take $\tilde{\mathcal{C}}_{*}(z'):=\tilde{C}_*(z')\vee \tilde{\mathcal{H}}_{*}(z')$.
In view of the identity $\eqref{f:3.4-new*}$, we obtain that
\begin{equation*}
\begin{aligned}
e^{\psi(x)-\psi(y)}\big|\nabla_x\Gamma(x,t;y,s)\big|
&\leq \bigg(\int_{\mathbb{R}^d}dz\big|e^{\psi(x)-\psi(y)}\nabla_x\Gamma(x,t;z,r)\big|^2\bigg)^{\frac{1}{2}}
\bigg(\int_{\mathbb{R}^d}dz\big|e^{\psi(z)-\psi(y)}\Gamma^*(y,s;z,r)\big|^2\bigg)^{\frac{1}{2}}\\
&\lesssim^{\eqref{f:3.5-new*},\eqref{f:3.50-new*}} \mathcal{C}_{*}(z)\tilde{\mathcal{H}}_{*}(z')R^{-d-1}e^{\kappa\gamma^2R^2}.
\end{aligned}
\end{equation*}
 By the same token, one can derive that
\begin{equation*}
\begin{aligned}
e^{\psi(x)-\psi(y)}\big|\nabla_y\Gamma(x,t;y,s)\big|
&\leq \bigg(\int_{\mathbb{R}^d}dz\big|e^{\psi(x)-\psi(y)}\Gamma(x,t;z,r)\big|^2\bigg)^{\frac{1}{2}}
\bigg(\int_{\mathbb{R}^d}dz\big|e^{\psi(z)-\psi(y)}\nabla_y\Gamma^*(y,s;z,r)\big|^2\bigg)^{\frac{1}{2}}\\
&\lesssim^{\eqref{f:3.3-new*},\eqref{f:3.51-new*}} \mathcal{H}_{*}(z)\tilde{\mathcal{C}}_{*}(z')R^{-d-1}e^{\kappa\gamma^2R^2},
\end{aligned}
\end{equation*}
and
\begin{equation*}
\begin{aligned}
e^{\psi(x)-\psi(y)}\big|\nabla_x\nabla_y\Gamma(x,t;y,s)\big|
&\leq \bigg(\int_{\mathbb{R}^d}dz\big|e^{\psi(x)-\psi(y)}\nabla_x\Gamma(x,t;z,r)\big|^2\bigg)^{\frac{1}{2}}
\bigg(\int_{\mathbb{R}^d}dz\big|e^{\psi(z)-\psi(y)}\nabla_y\Gamma^*(y,s;z,r)\big|^2\bigg)^{\frac{1}{2}}\\
&\lesssim^{\eqref{f:3.50-new*},\eqref{f:3.51-new*}} \mathcal{C}_{*}(z)\tilde{\mathcal{C}}_{*}(z')R^{-d-2}e^{\kappa\gamma^2R^2}.
\end{aligned}
\end{equation*}

Consequently, using the same computation (about how to fix
the weight $\psi$ to get the Gaussian upper bound) as given for
$\eqref{pri:1.1-new*}$ leads to the desired estimates $\eqref{pri:1.1-new*2}$. Moreover, the stated estimate $\eqref{pri:1.2-new*}$ follows from
the estimates $\eqref{def:2.3-new}$ and $\eqref{f:2.2-new*}$, and therefore
one can derive the estimates $\eqref{pri:1.3-new*}$ straightforwardly by
H\"older's inequality.
This ends the whole proof.
\qed

\begin{corollary}
Let $\varepsilon\in(0,1]$, $R>0$, and $d\geq 2$.
Suppose that the ensemble
$\langle\cdot\rangle$ is stationary and satisfies the spectral gap condition $\eqref{c:3}$, and the admissible coefficients satisfy
the local smoothness condition $\eqref{c:4}$. Let $1\leq p<\bar{p}<\infty$, $0<\alpha\leq 1$, and $q>(d+2)/(1-\theta)$ with $0<\theta\ll 1$.  Assume that
$u_\varepsilon$ and $f$ are associated by the equation:
$(\partial_t+\mathcal{L}_\varepsilon)u_\varepsilon = \nabla\cdot f$ in $Q_{4R}(z)$ with $z\in\mathbb{R}^{d+1}$. Then, there holds the uniform $L^\infty$-estimate
\begin{equation}\label{pri:2.11-1*}
\begin{aligned}
&\big\langle|u_\varepsilon(z)|^p\big\rangle^{\frac{1}{p}}
\lesssim_{\lambda,\Lambda,\lambda_1,\lambda_2,d,q,p,\bar p}
\Bigg\{\bigg\langle\Big(\dashint_{Q_R(z)}|u_\varepsilon|^2\Big)^{\frac{\bar p}{2}}\bigg\rangle^{\frac{1}{\bar p}}
+ R\bigg\langle\Big(\dashint_{Q_R(z)}|f|^q
\Big)^{\frac{\bar p}{q}}\bigg\rangle^{\frac{1}{\bar p}}\Bigg\},
\end{aligned}
\end{equation}
and the annealed Lipschitz estimate
\begin{equation}\label{pri:2.10-1}
\big\langle |\nabla u_\varepsilon(z)|^p
\big\rangle^{\frac{1}{p}}
\lesssim_{\lambda,\Lambda,\lambda_1,\lambda_2,d,p,\bar{p},\alpha}
\frac{1}{R}\bigg\langle\Big(\dashint_{Q_R(z)}|u_\varepsilon|^2
\Big)^{\frac{\bar{p}}{2}}\bigg\rangle^{\frac{1}{\bar{p}}}
+ R^{\alpha}\big\langle [f]_{C^{\alpha,\alpha/2}(Q_R(z))}^{\bar p}
\big\rangle^{\frac{1}{\bar p}}
+\bigg\langle\Big(\dashint_{Q_R(z)}|f|^2
\Big)^{\frac{\bar{p}}{2}}\bigg\rangle^{\frac{1}{\bar{p}}}.
\end{equation}
\end{corollary}

\begin{proof}
The idea is standard, and we merely provide the proof of $\eqref{pri:2.11-1*}$ for the reader's convenience. It is fine to assume $z=0$. Let $\eta\in C^1_0(Q_R)$ be a cut-off function
satisfying $\eta = 1$ on $Q_{R/4}$, $\eta =0$ outside $Q_{R/2}$,
and $R^2|\partial_t\eta| + R|\nabla\eta|\lesssim_d 1$.
Let $v_\varepsilon=\eta u_\varepsilon$, and we first find that
\begin{equation*}
\big(\partial_t+\mathcal{L}_\varepsilon\big)v_\varepsilon
=\partial_t\eta u_\varepsilon -\nabla\cdot\big(a^\varepsilon\nabla\eta u_\varepsilon -\eta f\big)-\nabla\eta\cdot\big(f+a^\varepsilon\nabla u_\varepsilon\big).
\end{equation*}

Let $z':=(y,s)\in\mathbb{R}^{d+1}$. From the representation of solutions in terms of fundamental solutions, one can obtain
\begin{equation}\label{f:3.13-new*}
\begin{aligned}
|u_\varepsilon(0)|=|v_\varepsilon(0)|
&\leq \int_{\mathbb{R}^{d+1}}dz'
|\Gamma_\varepsilon(0;z')|
\Big(\big|(\partial_t\eta u_\varepsilon)(z')\big|+\big|\big(\nabla\eta\cdot(f+a^\varepsilon\nabla u_\varepsilon)\big)(z')\big|\Big)\\
&\qquad\qquad\qquad\qquad+\int_{\mathbb{R}^{d+1}}\big|\nabla_{y}\Gamma_\varepsilon(0,z')\cdot\big(a^\varepsilon\nabla\eta u_\varepsilon-\eta f\big)(z')\big|=: I_1+I_2.
\end{aligned}
\end{equation}
In view of Theorem $\ref{thm:0-new}$, it follows from H\"older's inequality and Caccioppoli's inequality that
\begin{equation}\label{f:3.11-new*}
\begin{aligned}
I_1&\lesssim^{\eqref{pri:1.1-new*}}
\mathcal{H}_\varepsilon(0)\Big(\dashint_{Q_{R}}\tilde{\mathcal{H}}_{*}^{2}(\cdot/\varepsilon)\Big)^{\frac{1}{2}}
\bigg\{\Big(\dashint_{Q_R}|u_\varepsilon|^2\Big)^{\frac{1}{2}}+R\Big(\dashint_{Q_R}|f|^2\Big)^{\frac{1}{2}} +R
\Big(\dashint_{Q_{R/2}}|\nabla u_\varepsilon|^2\Big)^{\frac{1}{2}}\bigg\}\\
&\lesssim \mathcal{H}_\varepsilon(0)\Big(\dashint_{Q_{R}}\tilde{\mathcal{H}}_{*}^{2}(\cdot/\varepsilon)\Big)^{\frac{1}{2}}
\bigg\{\Big(\dashint_{Q_R}|u_\varepsilon|^2\Big)^{\frac{1}{2}}+R\Big(\dashint_{Q_R}|f|^2\Big)^{\frac{1}{2}}\bigg\}.
\end{aligned}
\end{equation}
Let $s\gg 1$, $\theta':=1/s$, and $1/s+1/r+1/q=1$ with $1< r<\frac{d+2}{d+1}$. Taking $\theta=\theta'(d+2)$, it is not hard to verify the condition: $q>(d+2)/(1-\theta)$. This ensures that the following calculation holds true.
\begin{equation}\label{f:3.12-new*}
\begin{aligned}
I_2&\lesssim^{\eqref{pri:1.1-new*2}}
\mathcal{C}_*(0)\Big(\dashint_{Q_{R}}\tilde{\mathcal{C}}_{*}^{2}(\cdot/\varepsilon)\Big)^{\frac{1}{2}}\Big(\dashint_{Q_R}|u_\varepsilon|^2\Big)^{\frac{1}{2}}
+ \mathcal{C}_*(0)\int_{Q_{R/2}}dz'\frac{\tilde{\mathcal{C}}_{*}(z'/\varepsilon)f(z')}{[\text{d}(z',0)]^{d+1}}\\
&\lesssim \mathcal{C}_\varepsilon(0)\Big(\dashint_{Q_{R}}\tilde{\mathcal{C}}_{*}^{2}(\cdot/\varepsilon)\Big)^{\frac{1}{2}}
\bigg\{\Big(\dashint_{Q_R}|u_\varepsilon|^2\Big)^{\frac{1}{2}}+R\Big(\dashint_{Q_R}|f|^2\Big)^{\frac{1}{2}}\bigg\}
+ \mathcal{C}_\varepsilon(0)\Big(\dashint_{Q_R}\tilde{\mathcal{C}}_{*}^{s}(\cdot/\varepsilon)\Big)^{\frac{1}{s}}R\Big(\dashint_{Q_R}|f|^{q}\Big)^{\frac{1}{q}}\\
&\lesssim
\mathcal{C}_\varepsilon(0)\Big(\dashint_{Q_R}\tilde{\mathcal{C}}_{*}^{s}(\cdot/\varepsilon)\Big)^{\frac{1}{s}}
\bigg\{\Big(\dashint_{Q_R}|u_\varepsilon|^2\Big)^{\frac{1}{2}}+R\Big(\dashint_{Q_R}|f|^q\Big)^{\frac{1}{q}}\bigg\}.
\end{aligned}
\end{equation}

Plugging the estimates $\eqref{f:3.11-new*}$ and $\eqref{f:3.12-new*}$ back into $\eqref{f:3.13-new*}$, we have
\begin{equation*}
|u_\varepsilon(0)|\lesssim
\mathcal{C}_\varepsilon(0)\Big(\dashint_{Q_R}\tilde{\mathcal{C}}_{*}^{s}(\cdot/\varepsilon)\Big)^{\frac{1}{s}}
\bigg\{\Big(\dashint_{Q_R}|u_\varepsilon|^2\Big)^{\frac{1}{2}}+R\Big(\dashint_{Q_R}|f|^q\Big)^{\frac{1}{q}}\bigg\}.
\end{equation*}
This, together with $\eqref{pri:1.2-new*}$
and H\"older's inequality, leads to the desired estimate $\eqref{pri:2.11-1*}$.
This completes the whole proof.
\end{proof}

\section{Weighted estimates and weak norm estimates}\label{sec:4}

\noindent
We first introduce the definition of the weight function in this section. Let
\begin{equation}\label{f:4.1-weight}
w_{z'}(z)=\left\{\begin{aligned}
&\ln^{\frac{3}{4}}(4+\text{d}(z,z')) &\quad&\text{if}\quad d=2;\\
&1  &\quad&\text{if}\quad d>2
\end{aligned}\right.
\end{equation}
be a weight with $z'=(y,s)\in\mathbb{R}^{d+1}$ as the anchor point.
If taking $z'=0$, we observe that $w_{0} = \bar{\mu}_d$.
\begin{lemma}[Weighted estimates]\label{lemma:4.1-new} Let $d\geq 2$.
For any $\gamma>0$, let $\psi$ be given as in $\eqref{f:2.1-x}$ satisfying
\begin{equation}\label{f:2.1-x-new*}
|\nabla\psi(x)|\leq \gamma;
\qquad |\nabla^2\psi(x)|\lesssim_d \gamma^2\qquad\forall x\in\mathbb{R}^d.
\end{equation}
Assume that $u_0$ and $f$ are associated by the equation:
\begin{equation}\label{pde:4.1-new}
\big(\partial_t-\nabla\cdot\bar{a}\nabla\big)u_0
=e^{-\psi}f\qquad\text{in}\quad\mathbb{R}^{d+1},
\end{equation}
and $u_0(\cdot,t) = 0$ if $t\leq s$ with some given $s\in\mathbb{R}$,
where $\text{supp}(f)\Subset \mathbb{R}^d\times(s,\infty)$ is assumed.
Let $w_s:=w_{(y,s)}$ be the weight function with $y\in\mathbb{R}^d$ arbitrarily fixed, and $w_{(y,s)}$ is given as in  $\eqref{f:4.1-weight}$.
Then, for any $t>s$, there hold the following estimates
\begin{subequations}
\begin{align}
&\int_{\mathbb{R}^d}|e^{\psi}u_0(\cdot,t)|^2\lesssim (t-s)e^{\frac{C\Lambda^2\gamma^2}{\lambda}(t-s)}\int_{s}^{t}\int_{\mathbb{R}^d}
|f|^2w_{s}^2;\label{pri:3.7-x}	\\
&\int_{s}^t\int_{\mathbb{R}^d}
\big(|e^{\psi}\nabla^2u_0|^2+|e^{\psi}\partial_tu_0|^2\big)
w_s^2\lesssim e^{\frac{C\Lambda^2\gamma^2}{\lambda}(t-s)}\int_s^{t}\int_{\mathbb{R}^d}|f|^2w_{s}^2.
\label{pri:3.5-x}
\end{align}
\end{subequations}
\end{lemma}

\begin{proof}
The proof of $\eqref{pri:3.5-x}$ is much more difficult than that of $\eqref{pri:3.7-x}$, and therefore we focus on the estimate $\eqref{pri:3.5-x}$.
The whole proof is divided into four steps.

\textbf{Step 1.} Outline the proof of $\eqref{pri:3.5-x}$.
By the definition of $w_s$, it is known that $w_s^2\in A_2$ (the Muckenhoupt class of parabolic cylinders).
Let $v_0:= e^{\psi}u_0$, and we now consider the following equations:
\begin{equation}\label{pde:3.2-x}
\left\{\begin{aligned}
\partial_t v_0 -\nabla\cdot\bar{a}\nabla v_0 &= \tilde{f}
\quad\text{in}\quad \mathbb{R}^d\times(s,\infty);\\
v_0(\cdot,s)&= 0
\quad\text{on} \quad \mathbb{R}^d,
\end{aligned}\right.
\end{equation}
where $\tilde{f}:=f-\nabla\cdot(\bar{a}\nabla\psi v)
-\bar{a}\nabla v\cdot\nabla\psi +a\nabla\psi\cdot\nabla\psi v$. Now, apply weighted
Calder\'on-Zygmund estimates to the above equations, and we obtain
\begin{equation}\label{f:3.5-xx}
\begin{aligned}
\int_{s}^{t}\int_{\mathbb{R}^d}\big(|\partial_t v_0|^2
&+|\nabla^2v_0|^2\big)w_s^2
\lesssim \int_{s}^{t}\int_{\mathbb{R}^d}|\tilde{f}|^2w_s^2 \\
&\lesssim^{\eqref{c:1},\eqref{f:2.1-x-new*}} \int_{s}^{t}\int_{\mathbb{R}^d}|f|^2w_s^2
+2\Lambda^2\gamma^{4}\int_{s}^{t}\int_{\mathbb{R}^d}|v_0|^2w_s^2
+\Lambda^2\gamma^2\int_{s}^{t}\int_{\mathbb{R}^d}|\nabla v_0|^2w_s^2.
\end{aligned}
\end{equation}
Thus, it suffices to estimate the last two terms in the right-hand side of
$\eqref{f:3.5-xx}$. We assert that the following two estimates are valid:
\begin{subequations}
\begin{align}
&\Lambda^2\gamma^4
\int_{s}^{t}\int_{\mathbb{R}^d}v_0^2w_s^2
\lesssim
e^{\frac{C\Lambda^2\gamma^2(t-s)}{\lambda}}\int_s^t\int_{\mathbb{R}^d}|f|^2w_s^2; \label{f:3.11-x}\\
&\Lambda^2\gamma^2\int_{s}^{t}\int_{\mathbb{R}^d}|\nabla v_0|^2w_s^2
\lesssim e^{\frac{C\Lambda^2\gamma^2}{\lambda}(t-s)}
\int_s^t\int_{\mathbb{R}^d}|f|^2w_s^2.\label{f:3.12-x}
\end{align}
\end{subequations}
Consequently, plugging $\eqref{f:3.11-x}$ and $\eqref{f:3.12-x}$
back into $\eqref{f:3.5-xx}$, one can obtain the desired estimate
$\eqref{pri:3.5-x}$.

\textbf{Step 2.} Arguments for $\eqref{f:3.11-x}$.
To see this, taking $v_0w_s^2$ as the test function acting on both sides of $\eqref{pde:3.2-x}$, there holds
\begin{equation}\label{f:3.8-x}
\begin{aligned}
\frac{1}{2}\frac{d}{dt}\int_{\mathbb{R}^d}v_0^2w_s^2
&+\lambda\int_{\mathbb{R}^d}|\nabla v_0|^2w_s^2
\leq \int_{\mathbb{R}^d}|fv_0|w_s^2 +\Lambda\gamma^2\int_{\mathbb{R}^d}v_0^2w_s^2 + \int_{\mathbb{R}^d}v_0^2\partial_\tau w_s w_s
\\
&+2\Lambda\gamma\int_{\mathbb{R}^d}|\nabla v_0||v_0|w_s^2
+2\Lambda\gamma\int_{\mathbb{R}^d}v_0^2|\nabla w_s|w_s
+2\Lambda\int_{\mathbb{R}^d}|\nabla v_0||\nabla w_s||v_0w_s|,
\end{aligned}
\end{equation}
where we employ the uniform ellipticity condition $\eqref{c:1}$ and the fact $\eqref{f:2.1-x}$.
Note that
\begin{equation}\label{f:3.1-new}
\begin{aligned}
&|\partial_\tau w_s| \leq \frac{3 w_s}{8(\tau-s)};\\
&|\nabla w_s|\leq \frac{3}{4(\tau-s)^{\frac{1}{2}}};
\quad\text{and}\quad 1\leq w_s.
\end{aligned}
\end{equation}
It follows from Young's inequality and $\eqref{f:3.1-new}$ that
\begin{equation}\label{f:3.7-x}
\begin{aligned}
\frac{1}{2}\frac{d}{dt}\int_{\mathbb{R}^d}v_0^2w_s^2+\lambda\int_{\mathbb{R}^d}|\nabla v_0|^2w_s^2
&\leq
C_{\theta}(t-s)\int_{\mathbb{R}^d}|f|^2w_s^2 +\big(\frac{3}{8}+2\theta\big)\int_{\mathbb{R}^d}\frac{v_0^2w_s^2}{t-s}
+2\vartheta\lambda\int_{\mathbb{R}^d}|\nabla v|^2w_s^2\\
&+\frac{C_{\vartheta}\Lambda^2}{\lambda(t-s)}\int_{\mathbb{R}^d}|v_0|^2 +\frac{C_{\vartheta,\theta}\Lambda^2\gamma^2}{\lambda}\int_{\mathbb{R}^d}v_0^2w_s^2,
\end{aligned}
\end{equation}
where $\theta,\vartheta\in(0,1)$ will be fixed later according to the concrete purpose. Recalling $v_0 = e^{\psi}u_0$, it follows from the estimate $\eqref{pri:3.7-x}$ and choosing $\vartheta =1/2$ that
\begin{equation*}
\begin{aligned}
\frac{1}{2}\frac{d}{dt}\int_{\mathbb{R}^d}v_0^2w_s^2
&\lesssim C_\theta(t-s)\int_{\mathbb{R^d}}|f|^2w_s^2
+\frac{C_\vartheta\Lambda^2}{\lambda}e^{\frac{C\Lambda^2\gamma^2}{\lambda}(t-s)}\int_s^{t}\int_{\mathbb{R}^d}|f|^2w_s^2
\\
&+\big(\frac{3}{8}+2\theta\big)\int_{\mathbb{R}^d}\frac{v_0^2
w_s^2}{t-s}
+\frac{C_{\vartheta,\theta}\Lambda^2\gamma^2}{\lambda}\int_{\mathbb{R}^d}v_0^2w_s^2.
\end{aligned}
\end{equation*}

Let $\alpha\in (0,1)$, $t_s:=t-s$, and $E(t):=\int_{\mathbb{R}^d}v_0^2w_s^2$. Then, one can further lead to
\begin{equation*}
\begin{aligned}
\frac{1}{2}\big(t_s^{-\alpha}E(t)\big)'
&=-\frac{\alpha}{2}t_s^{-1-\alpha}E(t)+t_s^{-\alpha}\frac{1}{2}E(t)'\\
&\leq \Big(\frac{3}{8}+2\theta-\frac{\alpha}{2}\Big)t_s^{-1-\alpha}E(t)
+\frac{C_{\vartheta,\theta}\Lambda^2\gamma^2}{\lambda}\big(t_s^{-\alpha}E(t)\big)\\
&\qquad\qquad\quad+C_{\theta}t_s^{1-\alpha}\int_{\mathbb{R}^d}|f|^2w_s^2
+\frac{C_{\vartheta}\Lambda^2}{\lambda}t_s^{-\alpha}e^{\frac{C\Lambda^2\gamma^2t_s}{\lambda}}\int_s^t\int_{\mathbb{R}^d}|f|^2w_s^2,
\end{aligned}
\end{equation*}
where $E'$ represents the derivative of $E$
with respect to the time variable.
One can choose
$\alpha = \frac{3}{4}+4\theta$ such that $\alpha\in(0,1)$ and
\begin{equation*}
\begin{aligned}
\frac{1}{2}\big(t_s^{-\alpha}E(t)\big)'
&\leq
\frac{C_{\vartheta,\theta}\Lambda^2\gamma^2}{\lambda}\big(t_s^{-\alpha}E(t)\big)\\
&\qquad\qquad\quad+C_{\theta}t_s^{1-\alpha}\int_{\mathbb{R}^d}|f|^2w_s^2
+\frac{C_{\vartheta}\Lambda^2}{\lambda}t_s^{-\alpha}e^{\frac{C\Lambda^2\gamma^2t_s}{\lambda}}\int_s^t\int_{\mathbb{R}^d}|f|^2w_s^2.
\end{aligned}
\end{equation*}
By using Gronwall's inequality
\begin{equation*}
\begin{aligned}
t_s^{-\alpha}E(t)
&\lesssim e^{\frac{C\Lambda^2\gamma^2(t-s)}{\lambda}}
\bigg\{\int_s^td\tau (\tau-s)^{1-\alpha}\int_{\mathbb{R}^d}|f|^2w_s^2
+\int_s^td\tau e^{\frac{C\Lambda^2\gamma^2(\tau-s)}{\lambda}}(\tau-s)^{-\alpha}\int_s^{\tau}\int_{\mathbb{R}^d}|f|^2w_s^2\bigg\}\\
&\lesssim (t-s)^{1-\alpha}e^{\frac{C\Lambda^2\gamma^2(t-s)}{\lambda}}
\int_s^{t}\int_{\mathbb{R}^d}|f|^2w_s^2,
\end{aligned}
\end{equation*}
which further leads to
\begin{equation}\label{f:3.9-x}
E(t)\lesssim (t-s)e^{\frac{C\Lambda^2\gamma^2(t-s)}{\lambda}}\int_s^t\int_{\mathbb{R}^d}|f|^2w_s^2.
\end{equation}
Therefore, one can derive that
\begin{equation}\label{f:3.10-x}
\int_s^t\int_{\mathbb{R}^d}v_0^2w_s^2\lesssim (t-s)^2
e^{\frac{C\Lambda^2\gamma^2(t-s)}{\lambda}}\int_s^t\int_{\mathbb{R}^d}|f|^2w_s^2,
\end{equation}
and furthermore obtain that
\begin{equation*}
\Lambda^2\gamma^4
\int_{s}^{t}\int_{\mathbb{R}^d}v_0^2w_s^2\lesssim
\frac{\Lambda^4\gamma^4}{\lambda^2}(t-s)^2
e^{\frac{C\Lambda^2\gamma^2(t-s)}{\lambda}}\int_s^t\int_{\mathbb{R}^d}|f|^2w_s^2
\lesssim
e^{\frac{C\Lambda^2\gamma^2(t-s)}{\lambda}}\int_s^t\int_{\mathbb{R}^d}|f|^2w_s^2,
\end{equation*}
which implies the stated estimate $\eqref{f:3.11-x}$.

\textbf{Step 3.} Arguments for $\eqref{f:3.12-x}$.
To see the estimate on the last term in the right-hand side of $\eqref{f:3.5-xx}$,
we need to start from the estimate $\eqref{f:3.7-x}$ again. Since we have established some auxiliary estimates, we can rewrite it as follows (taking $\theta=5/16$ and $\vartheta=1/4$):
\begin{equation*}
\begin{aligned}
\frac{1}{2}\frac{d}{dt}\int_{\mathbb{R}^d}v_0^2w_s^2
+\frac{\lambda}{2}\int_{\mathbb{R}^d}|\nabla v_0|^2w_s^2
&\lesssim^{\eqref{f:3.9-x},\eqref{pri:3.7-x}}
(t-s)\int_{\mathbb{R}^d}|f|^2w_s^2 +
e^{\frac{C\Lambda^2\gamma^2}{\lambda}(t-s)}\int_s^t\int_{\mathbb{R}^d}|f|^2w_s^2
\\
&+\frac{C_{\vartheta,\theta}\Lambda^2\gamma^2}{\lambda}\int_{\mathbb{R}^d}v_0^2w_s^2.
\end{aligned}
\end{equation*}
Integrating both sides above, we obtain that
\begin{equation*}
\begin{aligned}
\lambda\int_{s}^{t}\int_{\mathbb{R}^d}|\nabla v_0|^2w_s^2
&\lesssim (t-s)\int_s^t\int_{\mathbb{R}^d}|f|^2w_s^2
+(t-s)e^{\frac{C\Lambda^2\gamma^2}{\lambda}(t-s)}\int_s^t\int_{\mathbb{R}^d}|f|^2w_s^2
+\frac{\Lambda^2\gamma^2}{\lambda}\int_s^t\int_{\mathbb{R}^d}v_0^2w_s^2\\
&\lesssim^{\eqref{f:3.10-x}}
\Big\{(t-s)
+(t-s)e^{\frac{C\Lambda^2\gamma^2}{\lambda}(t-s)}
+\frac{\Lambda^2\gamma^2}{\lambda}(t-s)^2e^{\frac{C\Lambda^2\gamma^2}{\lambda}(t-s)}\Big\}\int_s^t\int_{\mathbb{R}^d}|f|^2w_s^2\\
&\lesssim (t-s)e^{\frac{C\Lambda^2\gamma^2}{\lambda}(t-s)}
\int_s^t\int_{\mathbb{R}^d}|f|^2w_s^2.
\end{aligned}
\end{equation*}
Therefore, there holds that
\begin{equation*}
\Lambda^2\gamma^2\int_{s}^{t}\int_{\mathbb{R}^d}|\nabla v_0|^2w_s^2
\lesssim \frac{\Lambda^2\gamma^2}{\lambda}(t-s)
e^{\frac{C\Lambda^2\gamma^2}{\lambda}(t-s)}
\int_s^t\int_{\mathbb{R}^d}|f|^2w_s^2
\lesssim e^{\frac{C\Lambda^2\gamma^2}{\lambda}(t-s)}
\int_s^t\int_{\mathbb{R}^d}|f|^2w_s^2,
\end{equation*}
which provides us with the stated estimate $\eqref{f:3.12-x}$.

\textbf{Step 4.}
Arguments for $\eqref{pri:3.7-x}$. Based on the decay estimates of the fundamental solution $\Gamma_0$, we start from the following estimate
\begin{equation*}
\begin{aligned}
|e^{\psi(x)}u_0(x,t)|
&\leq \int_{s}^td\tau\int_{\mathbb{R}^d}dye^{\psi(x)-\psi(y)}
\Gamma_0(x,t;y,\tau)|f(y,\tau)| \\
&\lesssim e^{C\gamma^2(t-s)}\int_{s}^td\tau\int_{\mathbb{R}^d}dy
\frac{1}{(t-\tau)^{\frac{d}{2}}}\exp\Big\{-\frac{\kappa|x-y|^2}{t-\tau}\Big\}\big|(fw_s)(y,\tau)\big|\\
&\lesssim e^{C\gamma^2(t-s)} (t-s)^{\frac{1}{2}}\bigg(
\int_{s}^t\int_{\mathbb{R}^d}\frac{1}{(t-\tau)^{\frac{d}{2}}}\exp\Big\{
-\frac{\kappa|x-y|^2}{t-\tau}\Big\}\big|(fw_s)(y,\tau)\big|^2\bigg)^{\frac{1}{2}},
\end{aligned}
\end{equation*}
and this implies the desired estimate $\eqref{pri:3.7-x}$. We have completed the whole proof.
\end{proof}

\begin{definition}[Smoothing operators]\label{def:1}
Let $\zeta\in C_0^\infty(Q_{\frac{1}{2}}(0))$ satisfy $\zeta\geq 0$ and $\int_{\mathbb{R}^{d+1}}\zeta  = 1$.
A parabolic smoothing operator is defined as
\begin{equation}\label{def:S-new}
 S_\varepsilon(f)(x,t) = \int_{\mathbb{R}^{d+1}}
 dyds\zeta_\varepsilon(x-y,t-s)f(y,s),
\end{equation}
where $\zeta_\varepsilon(x,t)=\varepsilon^{-d-2}
\zeta(x/\varepsilon,t/\varepsilon^2)$.
\end{definition}

\begin{lemma}[Random cancellation]\label{lemma:3.2-x}
Let $S_\varepsilon$ be defined as in $\eqref{def:S-new}$ with $\varepsilon\in(0,1]$.  Let $q\in[2,\infty)$, and let $\psi$ be a bounded Lipschitz function given in $\eqref{f:2.1-x}$.
Given $\beta\in [1,\infty)$ sufficiently large,
assume that the random field $\varpi$ satisfies
\begin{equation}\label{pri:3.1-x}
 \Big\langle\big(\dashint_{Q_1(z)}|\varpi|^q\big)^{\frac{\beta}{q}}\Big\rangle^{\frac{1}{\beta}}
 \leq c_\beta\theta(z)
 \qquad \forall z\in\mathbb{R}^{d+1}.
\end{equation}
Then, for any $p\in[1,\infty)$ and $(s,t)\subset\mathbb{R}$, given \textcolor{red}{a} deterministic function $f$, we have
\begin{equation}\label{pri:3.2-x}
\bigg\langle\Big(\int_s^t\int_{\mathbb{R}^d}\big|\varpi^\varepsilon
S_\varepsilon(f)\big|^q e^{q\psi}\Big)^{p}\bigg\rangle^{\frac{1}{p}}
\lesssim e^{q\gamma\varepsilon}\int_{s-\varepsilon^2}^{t+\varepsilon^2}
\int_{\mathbb{R}^d}|fe^{\psi}|^q\theta^q(\cdot/\varepsilon).
\end{equation}
\end{lemma}

\begin{proof}
The idea of the proof mainly relies on the Lipschitz property of the ``weight function'' (see e.g., \cite[Lemma 3.2]{Geng-Shen-19} or
\cite[Lemma 3.5]{Geng-Xu-26}), which can also be found in the third author's work \cite{Xu-16}. We provide a proof for the reader's convenience, and use H\"older's inequality and Fubini's theorem to obtain
\begin{equation}\label{f:3.1-x}
\begin{aligned}
\int_{s}^{t}\int_{\mathbb{R}^d}\big|\varpi^\varepsilon S_\varepsilon(f)\big|^q
e^{q\psi}
&\leq
\int_{s}^{t}\int_{\mathbb{R}^d}|\varpi^\varepsilon|^qS_\varepsilon(|fe^{\psi}|^q)
\big[S_\varepsilon(e^{-{q'}\psi})\big]^{q-1}e^{q\psi}\\
&\lesssim^{\eqref{f:2.1-x}}e^{q\varepsilon\gamma}
\int_{s}^{t}\int_{\mathbb{R}^d}|\varpi^\varepsilon|^qS_\varepsilon(|fe^{\psi}|^q)
\lesssim e^{q\varepsilon\gamma}
\int_{s-\varepsilon^2}^{t+\varepsilon^2}\int_{\mathbb{R}^d}|fe^{\psi}|^q\dashint_{Q_1(\cdot/\varepsilon)}|\varpi|^q.
\end{aligned}
\end{equation}
It is fine to assume $\beta\geq qp$, and taking $\langle(\cdot)^p\rangle^{\frac{1}{p}}$ on both sides of $\eqref{f:3.1-x}$ we obtain that
\begin{equation*}
\begin{aligned}
\bigg\langle\Big(\int_{s}^{t}\int_{\mathbb{R}^d}\big|\varpi^\varepsilon &S_\varepsilon(f)\big|^q
e^{q\psi}\Big)^{p}\bigg\rangle^{\frac{1}{p}}
\lesssim e^{q\varepsilon\gamma}
\bigg\langle\Big(\int_{s-\varepsilon^2}^{t+\varepsilon^2}\int_{\mathbb{R}^d}|fe^{\psi}|^q\dashint_{Q_1(\cdot/\varepsilon)}|\varpi|^q\Big)^{\frac{\beta}{q}}\bigg\rangle^{\frac{q}{\beta}}\\
&\lesssim e^{q\varepsilon\gamma}
\int_{s-\varepsilon^2}^{t+\varepsilon^2}\int_{\mathbb{R}^d}|fe^{\psi}|^q\Big\langle
\big(\dashint_{Q_1(\cdot/\varepsilon)}|\varpi|^q\big)^{\frac{\beta}{q}}
\Big\rangle^{\frac{q}{\beta}}
\lesssim^{\eqref{pri:3.1-x}}e^{q\varepsilon\gamma}
\int_{s-\varepsilon^2}^{t+\varepsilon^2}\int_{\mathbb{R}^d}|fe^{\psi}|^q\theta^q(\cdot/\varepsilon).
\end{aligned}
\end{equation*}
This completes the proof.
\end{proof}

\begin{lemma}[\cite{Geng-Shen-19}]\label{lemma:3.2-xx}
Let $S_\varepsilon$ be defined as in $\eqref{def:S-new}$ with $\varepsilon\in(0,1]$. Let $p\in[1,\infty)$ and $\psi$ be a bounded Lipschitz function given in $\eqref{f:2.1-x}$. Then, for any $(s,t)\subset\mathbb{R}$, there holds the following estimate
\begin{equation}\label{pri:3.4-x}
\int_{s}^{t}\int_{\mathbb{R}^d}e^{p\psi}\big|S_\varepsilon(\nabla f)-\nabla f\big|^p
\lesssim \varepsilon^pe^{p\varepsilon\gamma}
\int_{s-\varepsilon^2}^{t+\varepsilon^2}\int_{\mathbb{R}^d}
e^{p\psi}(|\nabla^2f|^p+|\partial_t f|^p).
\end{equation}
\end{lemma}

\begin{lemma}\label{lemma:3.1-x}
Let $2\leq q<\infty$, $\varepsilon\in(0,1]$, and $(s,t)\subset\mathbb{R}$ with $|t-s|\geq \varepsilon$. Assume that $u_0$ and $g$ are associated with the equations
$\big(\partial_t-\nabla\cdot\bar{a}\nabla\big)u_0
=g$ in $\mathbb{R}^{d+1}$ and $u_0(\cdot,t)=0$ if $t<s$. Let $f_\varepsilon$ be given as in
\begin{equation}\label{f:3.2-x}
\begin{aligned}
f_\varepsilon
:= (a^\varepsilon &-\bar{a})\big(\nabla u_0-\varphi\big)
+\varepsilon\big[a^\varepsilon\bar{\phi}_j^\varepsilon-\bar{\sigma}_j^\varepsilon+a^\varepsilon(\partial\sigma_{(d+1)j})^\varepsilon\big]
\nabla\varphi_j  \\
&+ \varepsilon^2\big[a^\varepsilon\otimes\bar{\sigma}_{(d+1)j}^\varepsilon:\partial^2\varphi_j
-\bar{\sigma}_{(d+1)j}^{\varepsilon}\partial_t\varphi_j\big]
+\varepsilon e_j\big(\partial_k\sigma_{k(d+1)j}\big)^\varepsilon(z_0)\partial_tS_\varepsilon(u_0),
\quad j=1,\cdots,d,
\end{aligned}
\end{equation}
where $\varphi_j:=S_\varepsilon(\partial_ju_0)$, $\bar{\phi}_j^\varepsilon:=\phi_j^\varepsilon-\phi_j^{\varepsilon}(z_0)$, and
$\bar{\sigma}_{ik'j}^{\varepsilon}:=\sigma_{ik'j}^{\varepsilon}-\sigma_{ik'j}^{\varepsilon}(z_0)$ with $k'=1,\cdots,d+1$, and $z_0=(x_0,t_0)\in\mathbb{R}^{d+1}$ with $t_0\in[s,t]$.
Then, for any $p\in[1,\infty)$, there holds that
\begin{equation}\label{pri:3.3-x}
\bigg\langle\Big(\int_{s}^{t}\int_{\mathbb{R}^d}|f_\varepsilon|^q e^{q\psi}
\Big)^{p}\bigg\rangle^{\frac{1}{p}}
\lesssim \varepsilon^q e^{q\varepsilon\gamma}
\int_{s}^{t+\varepsilon^2}
\int_{\mathbb{R}^d}\big(|e^{\psi}\nabla^2u_0|^q+|e^{\psi}\partial_tu_0|^q\big)
w_{z_0}^q(\cdot/\varepsilon),
\end{equation}
where $w_{z_0}$ is given in $\eqref{f:4.1-weight}$ with $z_0$ as
the anchor point. Moreover, for any given $r\geq \varepsilon$, there holds
\begin{equation}\label{pri:3.8-x}
\bigg\langle\Big(\dashint_{Q_r(z_0)}|f_\varepsilon|^qe^{q\psi}\Big)^{p}
\bigg\rangle^{\frac{1}{p}}
\lesssim \varepsilon^qe^{qr\gamma}
\dashint_{Q_{2r}(z_0)}
\big(|e^{\psi}\nabla^2u_0|^q+|e^{\psi}\partial_tu_0|^q\big)w_{z_0}^q(\cdot/\varepsilon).
\end{equation}
\end{lemma}

\begin{proof}
In view of the expression of $\eqref{f:3.2-x}$, one can firstly derive that
\begin{equation}\label{f:3.6-x}
\begin{aligned}
\int_{s}^{t}\int_{\mathbb{R}^d}|f_\varepsilon|^qe^{q\psi}
&\lesssim \int_{s}^{t}\int_{\mathbb{R}^d}|\nabla u_0-S_\varepsilon(\nabla u_0)|^q
e^{q\psi}+\int_{s}^{t}\int_{\mathbb{R}^d}
\big|\varepsilon\varpi_j^\varepsilon\big(\nabla S_\varepsilon(\partial_ju_0),
\partial_tS_\varepsilon(u_0)\big)\big|^qe^{q\psi}\\
&+\int_{s}^{t}\int_{\mathbb{R}^d}|\varepsilon^2
\tilde{\varpi}_j^\varepsilon\big(\nabla^2,\partial_t\big)S_\varepsilon(\partial_ju_0)|^qe^{q\psi}
=:I_1+I_2+I_3,
\end{aligned}
\end{equation}
where $\varpi_j:=(\bar{\phi}_j,\bar{\sigma}_j,\partial\sigma_{(d+1)j})$
and $\tilde{\varpi}_j:=\big(a\otimes\bar{\sigma}_{(d+1)j},\bar{\sigma}_{(d+1)j}\big)$. Then, one can handle the right-hand side above term by term. Taking $p=q$ in Lemma $\ref{lemma:3.2-xx}$, we have
\begin{equation}\label{f:3.5-x}
I_1\lesssim^{\eqref{pri:3.4-x}} \varepsilon^q e^{q\varepsilon\gamma}\int_{s}^{t+\varepsilon^2}
\int_{\mathbb{R}^d}
e^{q\psi}(|\nabla^2u_0|^q+|\partial_t u_0|^q),
\end{equation}
where we also employ the fact that $u_0(\cdot,t) = 0$ for $t\leq s$.
The remaining part requires the use of Lemma $\ref{lemma:3.2-x}$, and we obtain that
\begin{subequations}
\begin{align}
&\big\langle I_2^p\big\rangle^{\frac{1}{p}}
= \bigg\langle\Big(\int_{s}^{t}\int_{\mathbb{R}^d}
\big|\varepsilon\varpi_j^\varepsilon S_\varepsilon(\nabla\partial_ju_0)\big|^qe^{q\psi}\Big)^p\bigg\rangle^{\frac{1}{p}}
\lesssim^{\eqref{pri:3.2-x}}
\varepsilon^q
e^{q\varepsilon\gamma}\int_{s}^{t+\varepsilon^2}
\int_{\mathbb{R}^d}e^{q\psi}|(\nabla^2u_0,\partial_t u_0)|^q;
\label{f:3.3-x}\\
&\big\langle I_3^p\big\rangle^{\frac{1}{p}}
=\bigg\langle\Big(\int_{s}^{t}\int_{\mathbb{R}^d}|\varepsilon^2
\tilde{\varpi}_j^\varepsilon\partial_jS_\varepsilon\big((\nabla^2,\partial_t)u_0\big)|^qe^{q\psi}\Big)^p\bigg\rangle^{\frac{1}{p}}
\lesssim \varepsilon^q
e^{q\varepsilon\gamma}
\int_{s}^{t+\varepsilon^2}
\int_{\mathbb{R}^d}e^{q\psi}|(\nabla^2,\partial_t)u_0|^q w_{z_0}^q(\cdot/\varepsilon),
\label{f:3.4-x}
\end{align}
\end{subequations}
where we employ the following identity: $(\nabla^2,\partial_t)S_\varepsilon(\partial_ju_0) =\partial_j S_\varepsilon\big((\nabla^2,\partial_t)u_0\big)$
for the equality of $\eqref{f:3.4-x}$. Then, plugging
the estimates $\eqref{f:3.5-x}$, $\eqref{f:3.3-x}$, and $\eqref{f:3.4-x}$
back into $\eqref{f:3.6-x}$, we can obtain the desired
estimate $\eqref{pri:3.3-x}$. This completes the proof.
\end{proof}

\medskip

\noindent
\textbf{Proof of Lemma $\ref{lemma:4-xx}$.}
The basic idea is analogous to that given in \cite[Lemma 4]{Clozeau-Josien-Otto-Xu}, originally developed in \cite[Lemma 4]{Bella-Giunti-Otto-17}.
Although the basic method is the same, the approach used here to obtain the interpolation-type estimate $\eqref{f:5.2-new}$ is new. Additionally, the operator $\mathcal{A}^{k}$ is not straightforward. It can act as a positive operator to induce a set of orthogonal bases, while also maintaining the parabolic scaling of the parabolic operator, which ensures the feasibility of the previous method.
The whole proof has been divided into four steps.

\medskip

\textbf{Step 1.} Outline the main steps.
It is fine to assume $r=1$ by a rescaling argument. Let $\eta\in C^\infty_0(Q_1)$ be a cut-off function satisfying $\eta = 1$ in $Q_{3/4}$
and $\eta = 0$ outside $Q_{4/5}$ with $|\nabla \eta|^2+|\partial_t\eta|\lesssim_d 1$.
One can continue to construct
the zero-Dirichlet boundary value problem for the operator $\mathcal{A}:=-\partial_t^2+\Delta_x^{2}$ (with $k\geq 1$ being selected to be a sufficiently large even number later on) as follows:
\begin{equation}\label{pde:21-xx}
\left\{\begin{aligned}
\mathcal{A}^kv &= \eta u_\varepsilon &\quad&\text{in}\quad Q_1;\\
\big\{\partial_t^{k'-1-i}\Delta_x^{i}\big\}[v] &= 0 &\quad&\text{on}\quad\partial Q_1;\\
\frac{\partial}{\partial\nu}
\big\{\partial_t^{k'-1-i}\Delta_x^{i}\big\}[v]
&=0 &\quad&\text{on}\quad\partial B_1;
\qquad i\in\{0,\cdots,k'-1\}~\text{and}~k'\in\{1,\cdots,k\},
\end{aligned}\right.
\end{equation}
where $\partial/\partial\nu=\nu\cdot\nabla$ with $\partial^0_tv : = v$
and $\Delta_x^{0}v=: v$. Due to the Lax-Milgram theorem,
there exists a unique solution $v\in H_0^{2k,k}(Q_1)$.

Then, one can establish the following estimate
\begin{equation}\label{f:5.2-new}
\|\eta u_\varepsilon\|_{L^2(Q_{1})}
\lesssim_{\lambda,d,k}\|u_\varepsilon\|_{L^2(Q_1)}^{\frac{4k}{4k+1}}\|v\|_{L^2(Q_1)}^{\frac{1}{4k+1}}.
\end{equation}
By the unit decomposition and covering arguments, one can further infer that
\begin{equation}\label{f:5.3-new}
	\|u_\varepsilon\|_{L^2(Q_{5/8})}
	\lesssim_{\lambda,d,k}\|v\|_{L^2(Q_1)}.
\end{equation}
Moreover, using Caccioppoli's inequality, we arrive at
\begin{equation}\label{f:54-xx}
\int_{Q_{1/2}}|\nabla u_\varepsilon|^2
\lesssim_{\lambda,d} \int_{Q_{5/8}}|u_\varepsilon|^2
\lesssim^{\eqref{f:5.3-new}}
\int_{Q_1} |v|^2.
\end{equation}

On the other hand, thanks to the complete orthonormal system of  eigenfunctions $\{w_j\}\subset H_{0}^{2k,k}(Q_1)$ and eigenvalues
$\{\lambda_j\}_{j=1}^\infty$ with $\lambda_j> 0$, which are associated by
$\mathcal{A}^{k} w_j = \lambda_j w_j$ in $Q_1$ with zero-Dirichlet boundary conditions (it's convenient to assume $\|w_j\|_{L^2(Q_1)} =1$), we have
\begin{equation}\label{f:56-xx}
\begin{aligned}
\int_{Q_1} |v|^2
=\sum_{j}\frac{1}{\lambda_j^2}\Big(\int_{Q_1}\eta u_\varepsilon w_j \Big)^2
\leq \bigg(\sum_{j}\frac{1}{\lambda_j}\Big|\int_{Q_1}\eta u_\varepsilon w_j\Big|\bigg)^2
\lesssim \bigg(\sum_{j}\frac{1}{\sqrt{\lambda_j}}\frac{\big|\int_{Q_1}u_\varepsilon w_j\eta
\big|}{\|w_j\|_{H^{2k,k}(Q_1)}}\bigg)^{2}.
\end{aligned}
\end{equation}
This together with $\eqref{f:54-xx}$ gives
\begin{equation*}
\Big(\int_{Q_{1/2}}|\nabla u_\varepsilon|^2\Big)^{\frac{1}{2}} +
\Big(\int_{Q_{1/2}}|u_\varepsilon|^2\Big)^{\frac{1}{2}}
\lesssim
\sum_{j}\frac{1}{\sqrt{\lambda_j}}\frac{\big|\int_{Q_1}u_\varepsilon  w_j\eta
\big|}{\|w_j\|_{H^{2k,k}(Q_1)}}.
\end{equation*}
Putting $\langle(\cdot)^{p}\rangle^{\frac{1}{p}}$ with $p\in[1,\infty)$ on both sides above, and
using H\"older's inequality, we obtain that
\begin{equation}\label{f:55-xx}
\begin{aligned}
&\Big\langle\big(\int_{Q_{1/2}}|\nabla u_\varepsilon|^2\big)^{\frac{p}{2}}
\Big\rangle^{\frac{1}{p}} + \Big\langle\big(\int_{Q_{1/2}}|u_\varepsilon|^2\big)^{\frac{p}{2}}
\Big\rangle^{\frac{1}{p}}\\
&\lesssim \sum_{j}\frac{1}{\sqrt{\lambda_j}}\frac{\big\langle\big|\int_{Q_1}u_\varepsilon  w_j\eta
\big|^{p}\big\rangle^{1/p}}{\|w_j\eta\|_{H^{2k,k}(Q_1)}}
\lesssim \Big(\sum_{j}\lambda_j^{-\frac{1}{2}}\Big)
\sup_{g\in H_0^{2k,k}(Q_1)}
\frac{\big\langle|\int_{Q_1}u_\varepsilon g|^p
\big\rangle^{1/p}}{\|g\|_{H^{2k,k}(Q_1)}}.
\end{aligned}
\end{equation}
In addition, due to the finite trace of the inverse of $\mathcal{A}^{k}$
with zero-Dirichlet boundary conditions, one can derive that
\begin{equation}\label{f:5.2-xx}
   \sum_{j} \lambda_j^{-\frac{1}{2}} <\infty,
\end{equation}
provided that $k>(d+2)/2$ is selected to be an even number.
Plugging $\eqref{f:5.2-xx}$ back into $\eqref{f:55-xx}$ and taking $k>(d+2)/2+1$, and then applying the space-time Sobolev embedding theorem leads to
\begin{equation*}
\begin{aligned}
\Big\langle\big(\int_{Q_{1/2}}|\nabla u_\varepsilon|^2\big)^{\frac{p}{2}}
\Big\rangle^{\frac{1}{p}} + \Big\langle\big(\int_{Q_{1/2}}|u_\varepsilon|^2\big)^{\frac{p}{2}}
\Big\rangle^{\frac{1}{p}}
\lesssim \sup_{w\in C^{\infty}_0(Q_1)}\frac{\big\langle|\int_{Q_1}u_\varepsilon w|^p
\big\rangle^{1/p}}{\sup|\nabla^2 w|},
\end{aligned}
\end{equation*}
which finally implies the desired estimates $\eqref{pri:27-xxnew*}$ and $\eqref{pri:27-xx}$ after the rescaling argument.

\textbf{Step 2.} Arguments for $\eqref{f:5.2-new}$.
Let $g_\delta:=S_\delta(g)$ and $g:=\eta u_\varepsilon$ (where we omit the tracking of the parameter $\varepsilon$ for a simple notation since it plays no role in later computations).
We note that $g_\delta\in H_0^{2k,k}(Q_1)$ and supp($g_\delta$)$\subset Q_1$ since one may set $0<\delta\ll 1$.
To achieve our goal, we start from improving the corresponding estimate stated in \cite[Lemma 3.4, (3.14b)]{Geng-Xu-26} as follows:
\begin{equation}\label{f:5.4-new}
\int_{\mathbb{R}^{d+1}}\big|S_\delta(g)-g\big|^2
\lesssim
\delta^2\bigg\{\int_{\mathbb{R}}ds\|g(\cdot,s)\|_{H^1(\mathbb{R}^d)}^2
+\int_{\mathbb{R}}ds\|\partial_t g(\cdot,s)\|^2_{H^{-1}}\bigg\}.
\end{equation}
To see this, let $z=(x,t)$ and $z'=(x',t')$ with $d(z',0)\leq 1$. One can derive that
\begin{equation*}
\begin{aligned}
\int_{\mathbb{R}^{d+1}}dz
\big|g(z)&-g(z-\delta z')\big|^2
\lesssim \delta^2\int_{\mathbb{R}^{d+1}}|\nabla g|^2
+\int_{\mathbb{R}^{d+1}}dz\big|g(x-\delta x',t)
-g(x-\delta x',t-\delta^2t')\big|^2\\
&\lesssim \delta^2\int_{\mathbb{R}^{d+1}}|\nabla g|^2
+ \int_{\mathbb{R}}dt\|g(\cdot,t)\|_{H^1(\mathbb{R}^d)}
\|g(\cdot,t)
-g(\cdot,t-\delta^2t')\|_{H^{-1}(\mathbb{R}^d)}\\
&\lesssim \delta^2\int_{\mathbb{R}^{d+1}}|\nabla g|^2
+\delta^2\int_{\mathbb{R}}dt
\|g(\cdot,t)\|_{H^1(\mathbb{R}^d)}
\|\partial_sg(\cdot,t)\|_{H^{-1}(\mathbb{R}^d)}\\
&\lesssim \delta^2\int_{\mathbb{R}^{d+1}}\big(|\nabla g|^2
+|g|^2\big)
+\delta^2\int_{\mathbb{R}}dt\|\partial_sg(\cdot,t)\|_{H^{-1}(\mathbb{R}^d)}^2,
\end{aligned}
\end{equation*}
where we employ H\"older's inequality for the last inequality,
and this, together with the Fubini theorem and
$\int_{\mathbb{R}^{d+1}}\zeta =1$, leads to
\begin{equation*}
\begin{aligned}
\int_{\mathbb{R}^{d+1}}\big|S_\delta(g)-g\big|^2
&\lesssim \int_{\{z'\in\mathbb{R}^{d+1}:\text{d}(z',0)\leq 1\}}dz'\zeta(z')\int_{\mathbb{R}^{d+1}}dz
\big|g(z)-g(z-\delta z')\big|^2\\
&\lesssim
 \delta^2\int_{\mathbb{R}^{d+1}}\big(|\nabla g|^2
+|g|^2\big)
+\delta^2\int_{\mathbb{R}}ds\|\partial_tg(\cdot,s)\|_{H^{-1}(\mathbb{R}^d)}^2,
\end{aligned}
\end{equation*}
which provides us with the stated estimate $\eqref{f:5.4-new}$.

Moreover, we can employ the equation: $\partial_t u_\varepsilon
=\nabla\cdot a^\varepsilon\nabla u_\varepsilon$ in $Q_2$ to further estimate the right-hand side of $\eqref{f:5.4-new}$. By this,
one can have $\partial_t g = \partial_t\eta u_\varepsilon + \eta\partial_tu_\varepsilon
=\partial_t\eta u_\varepsilon + \nabla\cdot\eta a^\varepsilon\nabla u_\varepsilon-\nabla\eta\cdot a^\varepsilon\nabla u_\varepsilon$. Therefore, it follows from Caccioppoli's inequality that
\begin{equation}\label{f:5.5-new}
\begin{aligned}
&\|\partial_t g\|_{L^2(\mathbb{R};H^{-1}(\mathbb{R}^d))}
\lesssim \|u_\varepsilon\|_{L^2(Q_1)}
+\|\eta\nabla u_\varepsilon\|_{L^2(Q_1)}
+\|\nabla u_\varepsilon\|_{L^2(Q_{4/5})}
\lesssim_{\lambda,d} \|u_\varepsilon\|_{L^2(Q_1)}; \\
&\|g\|_{L^2(\mathbb{R};H^1(\mathbb{R}^d))}
\lesssim
\|u_\varepsilon\|_{L^2(Q_1)}
+\|\eta\nabla u_\varepsilon\|_{L^2(Q_1)}
\lesssim_{\lambda,d} \|u_\varepsilon\|_{L^2(Q_1)}.
\end{aligned}
\end{equation}
Combining the estimates $\eqref{f:5.4-new}$ and $\eqref{f:5.5-new}$, we have
\begin{equation}\label{f:5.6-new}
\|g_\delta -g\|_{L^2(\mathbb{R}^{d+1})}
\lesssim_{\lambda,d} \delta\|u_\varepsilon\|_{L^2(Q_1)}.
\end{equation}

In view of the equation
$\eqref{pde:21-xx}$, we obtain that
\begin{equation}\label{f:5.7-new}
\int_{Q_1}g g_\delta
=\int_{Q_1} \mathcal{A}^kv g_\delta
=\int_{Q_1} v\mathcal{A}^kg_\delta.
\end{equation}
To estimate the left-hand side of $\eqref{f:5.2-new}$, we appeal to the following formula
\begin{equation*}
\int_{Q_1}|g|^2 = \int_{Q_1}g(g-g_\delta)
+\int_{Q_1}g g_\delta
=^{\eqref{f:5.7-new}}\int_{Q_1}g(g-g_\delta)
+\int_{Q_1}v\mathcal{A}^{k}g_\delta,
\end{equation*}
and one can further derive that
\begin{equation}\label{f:5.8-new}
\begin{aligned}
\|g\|_{L^2(Q_1)}^2
&\leq \|g\|_{L^2(Q_1)}
\|g-g_\delta\|_{L^2(Q_1)} + \|v\|_{L^2(Q_1)}\|\mathcal{A}^kg_\delta\|_{L^2(Q_1)} \\
&\lesssim^{\eqref{f:5.6-new}}
\delta \|g\|_{L^2(Q_1)}\|u_\varepsilon\|_{L^2(Q_1)}
+ \delta^{-4k}\|v\|_{L^2(Q_1)}\|g\|_{L^2(Q_1)},
\end{aligned}
\end{equation}
where we also employ the fact that
$\|\mathcal{A}^{k}S_\delta(g)\|_{L^2(Q_1)}
\lesssim \delta^{-4k}\|g\|_{L^2(Q_1)}$ (which can be found
in \cite[Lemma 3.4]{Geng-Xu-26}) for the second inequality.
Noting that $\|g\|_{L^2(Q_1)}\leq \|u_\varepsilon\|_{L^2(Q_1)}$,
one can take $\delta := \theta\|g\|_{L^2(Q_1)}/\|u_\varepsilon\|_{L^2(Q_1)}$ with $0<\theta\ll 1$. Plugging this $\delta$
back into $\eqref{f:5.8-new}$ we obtain that
\begin{equation*}
\|g\|_{L^2(Q_1)}\lesssim \theta\|g\|_{L^2(Q_1)}
+\theta^{-4k}\|u_\varepsilon\|_{L^2(Q_1)}^{4k}\|v\|_{L^2(Q_1)}
\|g\|_{L^2(Q_1)}^{-4k},
\end{equation*}
which consequently proves the stated estimate $\eqref{f:5.2-new}$.

\textbf{Step 3.} Arguments for $\eqref{f:5.3-new}$. Applying  Young's inequality to the estimate $\eqref{f:5.2-new}$, one has
\begin{equation}\label{f:5.9-new*}
 \|\eta u_\varepsilon\|_{L^2(Q_1)}
 \leq \vartheta\|u_\varepsilon\|_{L^2(Q_1)}+ C_{\vartheta,\lambda,d}\|v\|_{L^2(Q_1)},
\end{equation}
where we can choose $\vartheta\in(0,1)$ small later on. The idea is to incorporate the first term in the right-hand side of $\eqref{f:5.9-new*}$ into its left-hand side through iteration. To do so, let $\varphi(s):=\|u_\varepsilon\|_{L^2(Q_s)}$, and $V:=\|v\|_{L^2(Q_1)}$. Set $1/2\leq s<t\leq 3/4$ and $h:=t-s$.
We first transform the estimate $\eqref{f:5.9-new*}$ into a local version by applying a truncation function, that is, the following estimate holds:
\begin{equation}\label{f:5.11-new*}
	\varphi(s)\leq \vartheta \varphi(t) + C\vartheta^{-4k}h^{-8k}V.
\end{equation}

To see this, let $\tilde{\eta}\in C_0^{\infty}(Q_1)$ be another cut-off function
satisfying $\tilde{\eta} = 1$ on $Q_s$, $\text{supp}(\tilde{\eta})\subset Q_{(s+t)/2}$, $0\leq \tilde{\eta}\leq 1$, and $|\partial_x^{\alpha}\partial_t^{j}\tilde{\eta}|\lesssim h^{-|\alpha|-2j}$. Now, as in Step 2, reset  $g=\tilde{\eta}u_\varepsilon$ and $g_\delta = S_\delta(g)$. The same computation (retaining the scale factor brought about by the cut-off function $\tilde{\eta}$) leads to
\begin{equation}\label{f:5.20-new*}
\|g_\delta -g\|_{L^2(\mathbb{R}^{d+1})}
\lesssim_{\lambda,d} \delta h^{-2}\varphi(t).
\end{equation}
Noting that $\eta = 1$ on $Q_{3/4}$, we therefore have $\eta\tilde{\eta} = \tilde{\eta}$, and
\begin{equation*}
\int_{Q_1}gg_\delta = \int_{Q_1}\eta u_\varepsilon\tilde{\eta}g_\delta
=\int_{Q_1}\mathcal{A}^{k}v\tilde{\eta}g_\delta
=\int_{Q_1}v\mathcal{A}^{k}(\tilde{\eta}g_\delta)
\end{equation*}
Repeating the computation as given for $\eqref{f:5.8-new}$, we obtain that
\begin{equation}
\varphi(s)\leq \|g\|_{L^2(Q_1)}\lesssim \delta h^{-2}\varphi(t) + \delta^{-4k}V,
\end{equation}
where we use that $\|\mathcal{A}^{k}(\tilde{\eta}g_\delta)\|_{L^2(Q_1)}\lesssim \delta^{-4k}\|g\|_{L^2(Q_1)}$ and $\eqref{f:5.20-new*}$. By choosing $\delta = c\vartheta h^2$ with
$c>0$ small enough, we can derive the stated estimate $\eqref{f:5.11-new*}$.

Then, let $s_n=\frac{3}{4}-\frac{1}{8}2^{-n}$, $n=0,1,\cdots$. We therefore have $s_0 = 5/8$ and $s_{n+1}-s_{n}=2^{-n-4}$. It follows from $\eqref{f:5.11-new*}$
that
\begin{equation*}
\varphi(s_0)\leq \vartheta^N\varphi(s_N) + C_{\vartheta}V\sum_{n=0}^{N-1}(\vartheta 2^{8k})^n.	
\end{equation*}
Let $\vartheta\in(0,1)$ be such that $\vartheta 2^{8k}<1$. Then, the desired
estimate $\eqref{f:5.3-new}$ follows from
the above estimate by letting $N\to\infty$, where we mention that $\varphi(s_{\infty})=\varphi(3/4)\leq \varphi(1)<\infty$.


\textbf{Step 4.} Show the estimate $\eqref{f:5.2-xx}$. Let $\mathcal{B}_k:=\mathcal{A}^k$ and $\lambda_j:=\lambda_j(\mathcal{B}_k)$ be the related  eigenvalues.
We start from the following estimate
\begin{equation}\label{f:5.21-new*}
\sum_j \frac{1}{\lambda_j}<\infty,
\end{equation}
provided $4k>d+2$. Let the notation $(\cdot,\cdot)$
represent the scalar product in $L^2(Q_1)$.
Let $G(\cdot,z')$ be the Green function associated with $\mathcal{A}^{k}$
with zero-Dirichlet boundary conditions, and for any $z,z'\in Q_1$ we have
\begin{equation*}
	G(z,z') = \sum_{k=1}^{\infty}\big(G(z,\cdot),w_k\big)w_{k}(z')
	=\sum_{k=1}^{\infty}\frac{w_k(z)w_k(z')}{\lambda_k}
\end{equation*}
This, together with the Fubini-Tonelli theorem, leads to
\begin{equation}\label{f:5.3-xx}
\sum_{j=1}^\infty\frac{1}{\lambda_j}
=\sum_{j=1}^\infty\frac{1}{\lambda_j}
\int_{Q_1}dz|w_j(z)|^2=\int_{Q_1}dz\sum_{j=1}^{\infty}\frac{|w_j(z)|^2}{\lambda_j}
=\int_{Q_1}dz G(z,z)<\infty,
\end{equation}
provided $4k>(d+2)$. To see the last inequality, one notices that
$G(\cdot,z')\in H_{0}^{2k,k}(Q_1)\subset C^{\alpha,\alpha/2}(\overline{Q}_1)$ with $\alpha:=1\wedge[2k-(d+2)/2]>0$,
and therefore the precondition $4k>(d+2)$ is required for $\eqref{f:5.21-new*}$. Moreover, we have 
$\int_{Q_1}dzG(z,z)\leq G(z_0,z_0)|Q_1|<\infty$.

Next, by setting $k=2l$, one can show that
\begin{equation}\label{f:5.4-xx}
 \lambda_j(\mathcal{B}_{2l})\geq \lambda_j^2(\mathcal{B}_{l}).
\end{equation}
We first mention that $\mathcal{B}_{2l}\not = \mathcal{B}_{l}^2$ here; otherwise there is nothing to do since it is known that $\sigma(\mathcal{B}_l^2) = \sigma(\mathcal{B}_l)^2$, where the notation $\sigma(\mathcal{B}_l)$ represents the spectrum set of operator $\mathcal{B}_l$. However, one can derive that for any $u\in H_{0}^{4l,2l}$,
\begin{equation}
 B_{2l}(u,u) = \|\mathcal{A}^{l}u\|_{L^2(Q_1)}^2 = \|\mathcal{B}_{l}u\|_{L^2(Q_1)}^2,
\end{equation}
where the bilinear form associated with $\mathcal{B}_{k}$ is denoted by
\begin{equation*}
B_k(u,v):=\sum_{r=0}^k \binom kr\int_{Q_1}(\partial_t^r\Delta_x^{k-r}u)(\partial_t^r\Delta_x^{k-r}v)
\qquad \forall u,v\in H_{0}^{2k,k}(Q_1).
\end{equation*}
Thus, the min-max principle leads to
\begin{equation*}
\begin{aligned}
\lambda_j(\mathcal{B}_{2l}) :=\inf_{S\subset H_0^{4l,2l}(Q_1)\atop \text{dim}(S)=j}
\sup_{u\in S\atop u\not=0}\frac{B_{2l}(u,u)}{\|u\|_{L^2(Q_1)}^2}
\geq \inf_{S\subset \text{Dom}(\mathcal{B}_l)\atop \text{dim}(S)=j}
\sup_{u\in S\atop u\not=0}\frac{\|B_{l}u\|_{L^2(Q_1)}^2}{\|u\|_{L^2(Q_1)}^2}
=\lambda_j^2(\mathcal{B}_l),
\end{aligned}
\end{equation*}
which gives the stated estimate $\eqref{f:5.4-xx}$.
This, together with $\eqref{f:5.21-new*}$, implies
\begin{equation*}
\sum_{j} \lambda_j^{-\frac{1}{2}}(\mathcal{B}_{2l})
\leq \sum_{j}\lambda_j^{-1}(\mathcal{B}_l)<\infty.
\end{equation*}
This ends the whole proof.
\qed

\section{Proof of Theorems $\ref{thm:1-new}$ and $\ref{thm:3-new}$}\label{sec:5}

\subsection{Proof of Theorem $\ref{thm:1-new}$}

\noindent
The deterministic duality argument of Geng and Shen \cite{Geng-Shen-19} does not directly extend to the annealed setting even for the estimate $\eqref{pri:1.1-x}$. Indeed, moment bounds for pairings against each deterministic test function do not, in general, control the moments of the corresponding pathwise dual norm, since the probability norm and the supremum over test functions cannot be interchanged. We therefore invoke the weak-norm estimate in Lemma $\ref{lemma:4-xx}$, which provides the required passage from these pairings to local annealed estimates while keeping the supremum outside the probability norm.

Set $z_0:=(x_0,t_0),z_0':=(y_0,s_0)\in\mathbb{R}^{d+1}$ with $s_0<t_0$. We may assume that $\varepsilon<r=(t_0-s_0)^{1/2}/100$, and set $s_*:=s_0-4r^2$ with $t_*:=t_0+4r^2$. For the case $0<r\leq \varepsilon$,
one can derive the desired estimates $\eqref{pri:1.1-x}$ and $\eqref{pri:1.1-xx}$  from the results stated
in Theorem $\ref{thm:0-new}$ straightforwardly.
For any $g\in C_0^{\infty}(
Q_r(z_0'))$, let $u_\varepsilon,u_0$ be associated with $g$ by the following equations:
\begin{equation}\label{pde:3.1-x}
\big(\partial_t-\nabla\cdot a^\varepsilon\nabla\big)u_\varepsilon
=\big(\partial_t-\nabla\cdot\bar{a}\nabla\big)u_0
=g\qquad\text{in}\quad\mathbb{R}^{d+1},
\end{equation}
where $u_\varepsilon(x,t)=u_0(x,t) = 0$ if $t\leq s_*$, and $\psi$ is given as in Lemma $\ref{lemma:4.1-new}$. Therefore, it follows that
\begin{equation}\label{pde:3.1x-x}
\begin{aligned}
u_\varepsilon(x,t)
&:=\int_{-\infty}^tds\int_{\mathbb{R}^d}dy
\Gamma_\varepsilon(x,t;y,s)g(y,s)
=\int_{-\infty}^tds\int_{\mathbb{R}^d}dy
\Gamma_\varepsilon^*(y,s;x,t)g(y,s); \\
u_0(x,t)
&:=\int_{-\infty}^{t}ds\int_{\mathbb{R}^d}dy \Gamma_0(x,t;y,s)g(y,s)
=\int_{-\infty}^{t}ds\int_{\mathbb{R}^d}dy \Gamma_0^{*}(y,s;x,t)g(y,s).
\end{aligned}
\end{equation}
From the above expression, it can be observed that the error estimate for the asymptotic expansion of the fundamental solutions can ultimately be transformed into the error estimate for the asymptotic expansion
between $u_\varepsilon$ and its homogenized solution $u_0$. Next, we will demonstrate how to specifically implement this basic idea, and the whole proof is divided into five steps. The first three steps are for providing the proof of the estimate $\eqref{pri:1.1-x}$, while the last two steps are based on similar ideas to prove $\eqref{pri:1.1-xx}$.

\textbf{Step 1.} Outline the key ingredients for the estimate $\eqref{pri:1.1-x}$. Let $U_\varepsilon(y,s):=\Gamma_\varepsilon(z_0;y,s)$,
$U_0(y,s):=\Gamma_0(z_0;y,s)$, and $z':=(y,s)\in\mathbb{R}^{d+1}$. Consider
\begin{equation}\label{expansion2}
W_\varepsilon:=U_\varepsilon-U_0-\varepsilon\bar{\phi}_{j}^{*\varepsilon}\tilde{\varphi}_j+\varepsilon^2\bar{\sigma}_{l(d+1)j}^{*\varepsilon}\partial_l\tilde{\varphi}_j,
\quad \text{with}\quad \tilde{\varphi}_j:=S_\varepsilon(\partial_jU_0),
\end{equation}
where $\bar{\phi}_j^{*\varepsilon}(z'):=\phi_j^{*}(z'/\varepsilon)-\phi_j^{*}(z_0'/\varepsilon)$ with
$\bar{\sigma}_{l(d+1)j}^{*\varepsilon}(z'):=\sigma_{l(d+1)j}^{*}(z'/\varepsilon)-\sigma_{l(d+1)j}^{*}(z_0'/\varepsilon)$. Let
$\phi_j^{*}$ and $\sigma_{k'}^{*}$ (with $k'=1,\cdots,d+1$) be the corresponding correctors associated with the adjoint operator $-\partial_t-\nabla\cdot a^*\nabla$. They own
the same results as those derived for $\phi_j$ and $\sigma_{k'}$ in Lemma $\ref{lemma:2.1}$.
\begin{equation}\label{pde:4.1-new*}
\left\{\begin{aligned}
\big(-\partial_s+\mathcal{L}_\varepsilon^{*}\big)W_\varepsilon
&=\nabla\cdot \tilde{f}_\varepsilon
&\quad&\text{in}\quad\mathbb{R}^{d}\times(-\infty,t_*);\\
W_\varepsilon(\cdot,s) &= 0	
&\quad&\text{on}\quad\mathbb{R}^{d},\quad\text{if}~s>t_*,
\end{aligned}\right.
\end{equation}
where $\tilde{f}_\varepsilon$ is given by
\begin{equation}\label{f:3.2-x-new*}
\begin{aligned}
\tilde{f}_\varepsilon
:= (a^{*\varepsilon} &-\bar{a}^{*})\big(\nabla U_0-\tilde{\varphi}\big)
+\varepsilon\big[a^{*\varepsilon}\bar{\phi}_j^{*\varepsilon}
-\bar{\sigma}_j^{*\varepsilon}-a^{*\varepsilon}(\partial\sigma_{(d+1)j}^{*})^{\varepsilon}\big]
\nabla\tilde{\varphi}_j  \\
&- \varepsilon^2\big[a^{*\varepsilon}\otimes\bar{\sigma}_{(d+1)j}^{*\varepsilon}:\partial^2\tilde{\varphi}_j
+\bar{\sigma}_{(d+1)j}^{*\varepsilon}\partial_s\tilde{\varphi}_j\big]
+\varepsilon\big(\partial_i\sigma_{i(d+1)\cdot}^{*}\big)^{\varepsilon}(z_0')S_{\varepsilon}(\partial_tU_0),
\quad  j=1,\cdots,d.
\end{aligned}
\end{equation}
We are interested in the estimate of $W_\varepsilon$ in the region $Q_{r}(z_0')$. Let $V_{\varepsilon}$ be the solution to the following equations
\begin{equation}\label{pde:4.2-new*}
\left\{\begin{aligned}
\big(-\partial_s+\mathcal{L}_\varepsilon^{*}\big)V_\varepsilon
&=\nabla\cdot \tilde{f}_\varepsilon
&\quad&\text{in}\quad Q_{r}(z_0');\\
V_\varepsilon  &= 0	
&\quad&\text{on}\quad\partial_p^{*}Q_{r}(z_0'),
\end{aligned}\right.
\end{equation}
where $\partial_p^{*}Q_{r}(z):=\big(\partial B_{r}(x)\times[t-r^2,t+r^2]\big)\cup \big(\overline{B_r(x)}\times\{{t+r^2}\}\big)$. We mention that the evolutionary time here ranges from $t + r^2$ to $t - r^2$. Let $H_\varepsilon:= W_\varepsilon - V_\varepsilon$ on $Q_{r}(z_0')$, and there holds
\begin{equation}\label{pde:4.3-new*}
	\big(-\partial_s+\mathcal{L}_\varepsilon^{*}\big)H_\varepsilon = 0
	\quad\text{in}\quad Q_{r}(z_0').
\end{equation}

In view of the equation $\eqref{pde:4.2-new*}$,
it follows from Proposition $\ref{prop:2.1}$ and Lemma $\ref{lemma:3.1-x}$ that
\begin{equation}\label{f:4.51-new*}
\begin{aligned}
&\|V_\varepsilon(z_0')\|_{p}
\lesssim^{\eqref{pri:2.11-1*}}
\bigg\langle
\Big(\dashint_{Q_{r}(z_0')}|V_\varepsilon|^2\Big)^{\frac{p_1}{2}}\bigg\rangle^{\frac{1}{p_1}}
+ r
\bigg\langle
\Big(\dashint_{Q_{r}(z_0')}|\tilde{f}_{\varepsilon}|^q\Big)^{\frac{p_1}{q}}\bigg\rangle^{\frac{1}{p_1}}\\
&\lesssim
r\bigg\langle
\Big(\dashint_{Q_{r}(z_0')}|\nabla V_\varepsilon|^2\Big)^{\frac{p_1}{2}}\bigg\rangle^{\frac{1}{p_1}}
+ r
\bigg\langle
\Big(\dashint_{Q_{r}(z_0')}|\tilde{f}_{\varepsilon}|^q\Big)^{\frac{p_1}{q}}\bigg\rangle^{\frac{1}{p_1}}
\lesssim r
\bigg\langle
\Big(\dashint_{Q_{r}(z_0')}|\tilde{f}_{\varepsilon}|^q\Big)^{\frac{p_1}{q}}\bigg\rangle^{\frac{1}{p_1}} \\
&\lesssim^{\eqref{pri:3.8-x}}
r\varepsilon
\bigg(\dashint_{Q_{2r}(z_0')}
\big(|\nabla^2U_0|^q+|\partial_sU_0|^q\big)w_{z_0'}^q(\cdot/\varepsilon)\bigg)^{\frac{1}{q}}
\lesssim^{\eqref{f:4.50-new*}} \frac{\varepsilon\mu_d(r/\varepsilon)}{(t_0-s_0)^{\frac{d+1}{2}}}\exp\Big\{-\frac{\kappa |x_0-y_0|^2}{t_0-s_0}\Big\},
\end{aligned}
\end{equation}
where we use Poincar\'e's inequality for the second inequality and the energy estimate for the third one. Also, we employ
the following decay estimate
\begin{equation}\label{f:4.50-new*}
\sup_{z'\in Q_{2r}(z_0')}\big(|\nabla^2 U_0| +|\partial_s U_0|+r|\nabla^3U_0|\big)
\lesssim \frac{1}{(t_0-s_0)^{\frac{d+2}{2}}}\exp\Big\{-\frac{\kappa |x_0-y_0|^2}{t_0-s_0}\Big\}.
\end{equation}

In view of the equation $\eqref{pde:4.3-new*}$, it follows from Lemma $\ref{lemma:4-xx}$ that
\begin{equation}\label{f:4.17-new*}
\begin{aligned}
\|H_\varepsilon(z_0')\|_p
&\lesssim^{\eqref{pri:2.11-1}}
\bigg\langle\Big(\dashint_{Q_{r/2}(z_0')}|H_\varepsilon|^2\Big)^{\frac{p_1}{2}}\bigg\rangle^{\frac{1}{p_1}}
\lesssim^{\eqref{pri:27-xxnew*}}
\sup_{g\in C^{\infty}_0(Q_r(z_0'))}\frac{\big\langle|\dashint_{Q_r(z_0')}H_\varepsilon
  g|^{p_1}
\big\rangle^{\frac{1}{p_1}}}{r^2\sup|\nabla^2 g|}\\
&\lesssim
\sup_{g\in C^{\infty}_0(Q_r(z_0'))}\frac{\big\langle|\dashint_{Q_r(z_0')}V_\varepsilon
  g|^{p_1}
\big\rangle^{\frac{1}{p_1}}}{r^2\sup|\nabla^2 g|}
+
\sup_{g\in C^{\infty}_0(Q_r(z_0'))}\frac{\big\langle|\dashint_{Q_r(z_0')}W_\varepsilon
  g|^{p_1}
\big\rangle^{\frac{1}{p_1}}}{r^2\sup|\nabla^2 g|}
=:I_1+I_2.
\end{aligned}
\end{equation}
We claim that there hold the following estimates:
\begin{subequations}
\begin{align}
& I_1= \sup_{g\in C^{\infty}_0(Q_r(z_0'))}\frac{\big\langle|\dashint_{Q_r(z_0')}V_\varepsilon
  g|^{p_1}
\big\rangle^{\frac{1}{p_1}}}{r^2\sup|\nabla^2 g|}\lesssim \frac{\varepsilon\mu_d(r/\varepsilon)}{(t_0-s_0)^{\frac{d+1}{2}}}\exp\Big\{-\frac{\kappa |x_0-y_0|^2}{t_0-s_0}\Big\};
\label{f:4.16-new*}\\
& I_2 = \sup_{g\in C^{\infty}_0(Q_r(z_0'))}\frac{\big\langle|\dashint_{Q_r(z_0')}W_\varepsilon
  g|^{p_1}
\big\rangle^{\frac{1}{p_1}}}{r^2\sup|\nabla^2 g|}
\lesssim \frac{\varepsilon\mu_d(r/\varepsilon)}{(t_0-s_0)^{\frac{d+1}{2}}}\exp\Big\{-\frac{\kappa|x_0-y_0|^2}{t_0-s_0}\Big\}. \label{f:4.15-new*}
\end{align}
\end{subequations}

Consequently, plugging the estimates $\eqref{f:4.16-new*}$ and $\eqref{f:4.15-new*}$ back into
$\eqref{f:4.17-new*}$, we have
\begin{equation}
\|H_\varepsilon(z_0')\|_p
\lesssim \frac{\varepsilon\mu_d(r/\varepsilon)}{(t_0-s_0)^{\frac{d+1}{2}}}\exp\Big\{-\frac{\kappa|x_0-y_0|^2}{t_0-s_0}\Big\}.
\end{equation}
This, together with $\eqref{f:4.51-new*}$, further gives
\begin{equation}
\|W_\varepsilon(z_0')\|_{p}
\leq \|V_\varepsilon(z_0')\|_{p} +
\|H_\varepsilon(z_0')\|_{p}\lesssim
\frac{\varepsilon\mu_d(r/\varepsilon)}{(t_0-s_0)^{\frac{d+1}{2}}}\exp\Big\{-\frac{\kappa|x_0-y_0|^2}{t_0-s_0}\Big\},
\end{equation}
which has already shown the desired estimate
$\eqref{pri:1.1-x}$, since it is not hard to observe that $W_\varepsilon(z_0') = U_\varepsilon(z_0')-U_0(z_0')$ according to
the expansion $\eqref{expansion2}$.

\textbf{Step 2.} Arguments for the estimates $\eqref{f:4.16-new*}$ and $\eqref{f:4.15-new*}$.
We start from showing the proof of
the estimate $\eqref{f:4.16-new*}$. To see this, a similar computation given for $\eqref{f:4.51-new*}$ leads to
\begin{equation}\label{f:4.52-new*}
\begin{aligned}
&\bigg\langle
\Big(\dashint_{Q_{r}(z_0')}|V_\varepsilon|^2\Big)^{\frac{p_1}{2}}\bigg\rangle^{\frac{1}{p_1}}
\lesssim r\bigg\langle
\Big(\dashint_{Q_{r}(z_0')}|\nabla V_\varepsilon|^2\Big)^{\frac{p_1}{2}}\bigg\rangle^{\frac{1}{p_1}}
\lesssim r
\bigg\langle
\Big(\dashint_{Q_{r}(z_0')}|\tilde{f}_{\varepsilon}|^2\Big)^{\frac{p_1}{2}}\bigg\rangle^{\frac{1}{p_1}}\\
&\lesssim^{\eqref{pri:3.8-x}}
r\varepsilon
\bigg(\dashint_{Q_{2r}(z_0')}
\big(|\nabla^2U_0|^2+|\partial_sU_0|^2\big)w_{z_0'}^2(\cdot/\varepsilon)\bigg)^{\frac{1}{2}}
\lesssim^{\eqref{f:4.50-new*}}
\frac{\varepsilon\mu_d(r/\varepsilon)}{(t_0-s_0)^{\frac{d+1}{2}}}\exp\Big\{-\frac{\kappa |x_0-y_0|^2}{t_0-s_0}\Big\}.
\end{aligned}
\end{equation}
This, together with H\"older's inequality, implies that
\begin{equation*}
I_1\leq \sup_{g\in C^{\infty}_0(Q_r(z_0'))}
\frac{\big\langle\big(\dashint_{Q_r(z_0')}|V_\varepsilon|^2\big)^{\frac{p_1}{2}}
\big\rangle^{\frac{1}{p_1}}\big(\dashint_{Q_{r}(z_0')}|g|^2\big)^{\frac{1}{2}}}{r^2\sup|\nabla^2 g|}
\lesssim^{\eqref{f:4.52-new*}}
\frac{\varepsilon\mu_d(r/\varepsilon)}{(t_0-s_0)^{\frac{d+1}{2}}}\exp\Big\{-\frac{\kappa |x_0-y_0|^2}{t_0-s_0}\Big\},
\end{equation*}
which gives the stated estimate $\eqref{f:4.16-new*}$.

We continue to show the estimate  $\eqref{f:4.15-new*}$, and we start from the following estimate:
\begin{equation}\label{f:4.11-new*}
\begin{aligned}
\Big\langle\big|\dashint_{Q_r(z_0')}W_\varepsilon
  g\big|^{p_1}
\Big\rangle^{\frac{1}{p_1}}
&\lesssim \frac{1}{r^{d+2}}
\bigg\langle\Big|\int_{\mathbb{R}^{d+1}}W_\varepsilon g\Big|^{p_1}\bigg\rangle^{\frac{1}{p_1}} \\
&\lesssim^{\eqref{f:4.12-new*}} \frac{1}{r^{d+2}}\bigg\{
\|w_\varepsilon(z_0)\|_{p_1}
+ \Big\|\int_{\mathbb{R}^{d+1}}\big(\varepsilon\bar{\phi}_j^{*\varepsilon}\tilde{\varphi}_j
-\varepsilon^2\bar{\sigma}_{l(d+1)j}^{*\varepsilon}\partial_l\tilde{\varphi}_j\big)g\Big\|_{p_1}\bigg\},
\end{aligned}
\end{equation}
where we employ the following identity\footnote{The second equality of
$\eqref{f:4.12-new*}$ involves the error estimate between $u_\varepsilon$ and $u_0$, which will be postponed to the next step (Step 3) for a detailed computation. Here we only use the notation $w_\varepsilon$ for representing this error.}:
\begin{equation}\label{f:4.12-new*}
\begin{aligned}
\int_{\mathbb{R}^{d+1}}W_\varepsilon g
&=^{\eqref{pde:3.1x-x}} u_\varepsilon(z_0)-u_0(z_0)-\int_{\mathbb{R}^{d+1}}\big(\varepsilon\bar{\phi}_j^{*\varepsilon}
\tilde{\varphi}_j -\varepsilon^2\bar{\sigma}_{l(d+1)j}^{*\varepsilon}\partial_l\tilde{\varphi}_j\big)g\\
&=^{\eqref{f:4.20-new*}}w_\varepsilon(z_0)
-\int_{\mathbb{R}^{d+1}}\big(\varepsilon\bar{\phi}_j^{*\varepsilon}\tilde{\varphi}_j -\varepsilon^2\bar{\sigma}_{l(d+1)j}^{*\varepsilon}\partial_l\tilde{\varphi}_j\big)g.
\end{aligned}
\end{equation}

On the one hand, it follows from H\"older's inequality and Lemma $\ref{lemma:3.2-x}$ that
\begin{equation}\label{f:4.13-new*}
\begin{aligned}
\Big\|\int_{\mathbb{R}^{d+1}}\big(\varepsilon\bar{\phi}_j^{*\varepsilon}\tilde{\varphi}_j
&-\varepsilon^2\bar{\sigma}_{l(d+1)j}^{*\varepsilon}\partial_l\tilde{\varphi}_j\big)g\Big\|_{p_1}
\leq
\bigg\langle\Big(\int_{Q_r(z_0')}\big|
\varepsilon\bar{\phi}_j^{*\varepsilon}\tilde{\varphi}_j -\varepsilon^2\bar{\sigma}_{l(d+1)j}^{*\varepsilon}\partial_l
\tilde{\varphi}_j\big|^2\Big)^{\frac{p_1}{2}}\bigg\rangle^{\frac{1}{p_1}}\|g\|_{L^2(Q_r(z_0'))}\\
&\lesssim^{\eqref{pri:3.2-x}}
\bigg\{\varepsilon\Big(
\int_{Q_{2r}(z_0')}|\nabla U_0|^2\Big)^{\frac{1}{2}}
+\varepsilon^2\Big(
\int_{Q_{2r}(z_0')}|\nabla^2 U_0|^2w_{z_0'}^2(\cdot/\varepsilon)\Big)^{\frac{1}{2}}\bigg\}\|g\|_{L^2(Q_r(z_0'))}\\
&\lesssim \frac{\varepsilon\mu_d(r/\varepsilon)r^{\frac{d+2}{2}}}{(t_0-s_0)^{\frac{d+1}{2}}}\exp\Big\{-\frac{\kappa|x_0-y_0|^2}{t_0-s_0}\Big\}\|g\|_{L^2(Q_r(z_0'))},
\end{aligned}
\end{equation}
where we employ the decay estimates of the fundamental solution $\Gamma_0$ in the last inequality. On the other hand, one can further claim that
for any $p\in[1,\infty)$ there holds
\begin{equation}\label{f:4.10-new*}
\big\|w_\varepsilon(z_0)\big\|_{p}
\lesssim
\varepsilon r \mu_d(r/\varepsilon)e^{-\kappa|x_0-y_0|^2/(t_0-s_0)}\Big(\dashint_{Q_{r}(z_0')}|g|^2\Big)^{\frac{1}{2}}.
\end{equation}

Combining the estimates $\eqref{f:4.11-new*}$,
$\eqref{f:4.13-new*}$,  and $\eqref{f:4.10-new*}$, we obtain that
\begin{equation*}
\begin{aligned}
\Big\langle\big|\dashint_{Q_r(z_0')}W_\varepsilon
  g\big|^{p_1}
\Big\rangle^{\frac{1}{p_1}}
&\lesssim \frac{1}{r^{d+2}}\bigg\{
\|w_\varepsilon(z_0)\|_{p_1}
+ \Big\|\int_{\mathbb{R}^{d+1}}\big(\varepsilon\bar{\phi}_j^{*\varepsilon}\tilde{\varphi}_j
-\varepsilon^2\bar{\sigma}_{l(d+1)j}^{*\varepsilon}\partial_l\tilde{\varphi}_j\big)g\Big\|_{p_1}\bigg\}\\
&\lesssim \frac{\varepsilon\mu_d(r/\varepsilon)}{(t_0-s_0)^{\frac{d+1}{2}}}\exp\Big\{-\frac{\kappa|x_0-y_0|^2}{t_0-s_0}\Big\}\sup_{Q_{r}(z_0')}|g|,
\end{aligned}
\end{equation*}
which consequently leads to
the stated estimate  $\eqref{f:4.15-new*}$.

\textbf{Step 3.} Arguments for the estimate $\eqref{f:4.10-new*}$.
Let the error of the two-scale expansions be given by
\begin{equation}\label{f:4.20-new*}
w_\varepsilon:=u_\varepsilon-u_0-\varepsilon\bar{\phi}_{j}^\varepsilon\varphi_j-\varepsilon^2\bar{\sigma}_{l(d+1)j}^{\varepsilon}\partial_l\varphi_j,
\end{equation}
where $\varphi_j:=S_\varepsilon(\partial_ju_0)$, $\bar{\phi}_j^\varepsilon:=\phi_j^\varepsilon-\phi_j^{\varepsilon}(z_0)$, and
$\bar{\sigma}_{ik'j}^{\varepsilon}:=\sigma_{ik'j}^{\varepsilon}-\sigma_{ik'j}^{\varepsilon}(z_0)$ with $k'=1,\cdots,d+1$.
It satisfies the following equations:
\begin{equation*}
\left\{\begin{aligned}
\big(\partial_t-\nabla\cdot a^\varepsilon\nabla\big)w_\varepsilon
&=\nabla\cdot f_\varepsilon
&\quad&\text{in}\quad\mathbb{R}^{d}\times(s_*,\infty);\\
w_\varepsilon(\cdot,s_*) &= 0	
&\quad&\text{on}\quad\mathbb{R}^{d},
\end{aligned}\right.
\end{equation*}
where $f_\varepsilon$ is expressed by $\eqref{f:3.2-x}$.
In view of energy estimates, for any $p\in[1,\infty)$ and $t\in[s_*,\infty)$, there holds that
\begin{equation*}
\bigg\langle
\Big(\int_{\mathbb{R}^d}|w_\varepsilon(\cdot,t)|^2e^{2\psi}\Big)^{p}\bigg\rangle^{\frac{1}{p}}
\lesssim e^{\kappa\gamma^2(t-s_{*})}
\bigg\langle\Big(\int_{s_*}^{t}\int_{\mathbb{R}^d}|f_\varepsilon|^2 e^{2\psi}
\Big)^{p}\bigg\rangle^{\frac{1}{p}}.
\end{equation*}
By Lemma $\ref{lemma:3.1-x}$, we have
\begin{equation}\label{f:6.9-new}
\begin{aligned}
\bigg\langle
\Big(\int_{\mathbb{R}^d}|w_\varepsilon(\cdot,t)|^2e^{2\psi}\Big)^{p}\bigg\rangle^{\frac{1}{p}}
&\lesssim^{\eqref{pri:3.3-x}} e^{\kappa\gamma^2(t_{*}-s_*)}
\varepsilon^2e^{2\varepsilon\gamma}
\int_{s_*}^{t_{*}}
\int_{\mathbb{R}^d}\big(|e^{\psi}\nabla^2u_0|^2+|e^{\psi}\partial_tu_0|^2\big)
w_{z_0}^2(\cdot/\varepsilon)\\
&\lesssim^{\eqref{f:6.8-new}}
\varepsilon^2e^{2\varepsilon\gamma
+\kappa\gamma^2(t_{*}-s_*)}
\int_{s_*}^{t_{*}}
\int_{\mathbb{R}^d}\big(|e^{\psi}\nabla^2u_0|^2+|e^{\psi}\partial_tu_0|^2\big)
w_{(x_0,s_*)}^2(\cdot/\varepsilon) \\
&\lesssim^{\eqref{pri:3.5-x}} 
\varepsilon^2e^{2\varepsilon\gamma
+\kappa\gamma^2(t_{*}-s_*)}\int_{s_*}^{t_{*}}\int_{\mathbb{R}^d}e^{2\psi}|g|^2 w_{(x_0,s_{*})}^2(\cdot/\varepsilon)\\
&\lesssim^{\eqref{f:6.8-new}} 
\varepsilon^2e^{2\varepsilon\gamma
+\kappa\gamma^2(t_{*}-s_*)}
\int_{Q_r(z_0')}
e^{2\psi}|g|^2 w_{z_0}^2(\cdot/\varepsilon),
\end{aligned}
\end{equation}
where we recall that $w_{z_0}$ is given in $\eqref{f:4.1-weight}$. We mention that $u_0(\cdot,t) = 0$ if $t<(s_0-r^2)$, and therefore
\begin{equation}\label{f:6.8-new}
\begin{aligned}
&w_{z_0}(x,\tau)\leq C w_{(x_0,s_*)}(x,\tau)
&\qquad &\forall \tau\in(s_0-r^2,t_*),~\text{and}~  x\in\mathbb{R}^d; \\
&w_{x_0,s_*}(x,\tau)\leq w_{z_0}(x,\tau)
&\qquad &\forall (x,\tau)\in Q_r(z_0'),
\end{aligned}
\end{equation}
in which the constant $C\geq 1$ is independent of $\varepsilon$ and $r$.
Let $R:=|x_0-y_0|$.
The above estimate $\eqref{f:6.9-new}$ further leads to
\begin{equation}\label{f:3.13-x}
\begin{aligned}
\bigg\langle\Big(\dashint_{Q_r(z_0)}|e^{\psi}w_\varepsilon|^2\Big)^{\frac{p}{2}}\bigg\rangle^{\frac{1}{p}}
&\lesssim \varepsilon e^{C\gamma^2r^2}r^{-\frac{d+2}{2}+1}
\bigg(\int_{Q_{r}(z_0')}e^{2\psi}|g|^2 w_{z_0}^2(\cdot/\varepsilon)\bigg)^{\frac{1}{2}}
\\
&\lesssim\varepsilon r \mu_d((R+r)/\varepsilon) e^{C\gamma^2r^2}
\Big(\dashint_{Q_{r}(z_0')}e^{2\psi}|g|^2\Big)^{\frac{1}{2}}.
\end{aligned}
\end{equation}

Let $1\leq p<p_1<\infty$, and $q>(d+2)$.
It follows from Proposition $\ref{prop:2.1}$ that
\begin{equation}\label{f:4.5-new*}
\begin{aligned}
\|w_\varepsilon(z_0)\|_{p}
\lesssim^{\eqref{pri:2.11-1}} \bigg\langle\Big(\dashint_{Q_r(z_0)}|w_\varepsilon|^2\Big)^{\frac{p_1}{2}}\bigg\rangle^{\frac{1}{p_1}} + r\bigg\langle\Big(\dashint_{Q_r(z_0)}|f_\varepsilon|^q
\Big)^{\frac{p_1}{q}}\bigg\rangle^{\frac{1}{p_1}},
\end{aligned}
\end{equation}
where we recall that $\|\cdot\|_p:=\langle(\cdot)^p\rangle^{\frac{1}{p}}$ represents the $p$-th moment. In view of Lemma $\ref{lemma:3.1-x}$
and the interior $L^p$ estimates (see \cite{Schlag-96}), we have
\begin{equation}\label{f:4.6-new*}
\begin{aligned}
\bigg\langle\Big(\dashint_{Q_r(z_0)}|f_\varepsilon|^q\Big)^{\frac{p_1}{q}}\bigg\rangle^{\frac{1}{p_1}}
&\lesssim^{\eqref{pri:3.8-x}}
\varepsilon \mu_d(r/\varepsilon)
\bigg(
\dashint_{Q_{2r}(z_0)}
\big(|\nabla^2u_0|^q+|\partial_tu_0|^q\big)\bigg)^{\frac{1}{q}}\\
&\lesssim
\varepsilon \mu_d(r/\varepsilon)
\bigg(
\dashint_{Q_{3r}(z_0)}
\big(|\nabla^2u_0|^2+|\partial_tu_0|^2\big)\bigg)^{\frac{1}{2}},
\end{aligned}
\end{equation}
where we mention that $u_0$ satisfies $(\partial_t+\mathcal{L}_0)u_0=0$ in $Q_{4r}(z_0)$.

For any $\gamma>0$ (to be fixed later on), 
we set $\psi(x):=1-\sqrt{1+\gamma^2|x-x_0|^2}$. Then,
it follows that $\psi\in C^{\infty}(\mathbb{R}^d)$,
$|\nabla\psi|\leq \gamma$, $|\nabla^2\psi|\lesssim_d\gamma^2$. Moreover, we have $-\gamma|x-x_0|\leq \psi(x)\leq 1-\gamma|x-x_0|$, so
there holds the following fact that
\begin{equation}\label{f:4.4-new*}
e^{-\psi(x)}\leq e^{\gamma r}
\quad\text{for~any~}x\in B_r(x_0);
\qquad
e^{\psi(y)}\leq e^{1-\gamma R+\gamma r}\quad \text{for~any~}y\in B_r(y_0).
\end{equation}
Therefore, it follows from $\eqref{f:4.5-new*}$, $\eqref{f:4.6-new*}$, and $\eqref{f:4.4-new*}$ that
\begin{equation}\label{f:3.15-x}
\begin{aligned}
&\big\|w_\varepsilon(z_0)\big\|_{p}
\lesssim
e^{\gamma r}\Bigg\{\bigg\langle\Big(\dashint_{Q_r(z_0)}|e^{\psi}w_\varepsilon|^2\Big)^{\frac{p_1}{2}}\bigg\rangle^{\frac{1}{p_1}}
+\varepsilon r \mu_d(r/\varepsilon)
\bigg(
\dashint_{Q_{3r}(z_0)}
e^{2\psi}\big(|\nabla^2u_0|^2+|\partial_tu_0|^2\big)\bigg)^{\frac{1}{2}}\Bigg\}\\
&\lesssim^{\eqref{f:3.13-x},\eqref{pri:3.5-x}}
\varepsilon  r
e^{\gamma r+C\gamma^2r^2}
\Big(\dashint_{Q_{r}(z_0')}e^{2\psi}|g|^2\Big)^{\frac{1}{2}}\Bigg\{
\mu_d((R+r)/\varepsilon)
+ \mu_d(r/\varepsilon)\Bigg\}\\
&\lesssim^{\eqref{f:4.4-new*},\eqref{f:3.14-x}}
\varepsilon  r\mu_d(r/\varepsilon)\big(1+R/r\big)^{\delta}
e^{C_0 + C_1\gamma r+C_2\gamma^2r^2 -\gamma R}
\Big(\dashint_{Q_{r}(z_0')}|g|^2\Big)^{\frac{1}{2}}
\end{aligned}
\end{equation}
where we also employ the following estimate for the last step above:
\begin{equation}\label{f:3.14-x}
\mu_d((R+r)/\varepsilon) \lesssim_{\delta}\mu_d(r/\varepsilon)\big(1+R/r\big)^{\delta}
\qquad \forall\delta>0.
\end{equation}

Taking $\gamma= cR/r^2$ and $0<c\leq (2C_2)^{-1}$, we have
\begin{equation*}
\big\|w_\varepsilon(z_0)\big\|_{p}
\lesssim
\varepsilon  r\mu_d(r/\varepsilon)\big(1+R/r\big)^{\delta}
e^{-\frac{cR^2}{4r^2}}
\Big(\dashint_{Q_{r}(z_0')}|g|^2\Big)^{\frac{1}{2}}
\lesssim \varepsilon r\mu_d(r/\varepsilon) e^{-\frac{c R^2}{8r^2}}\Big(\dashint_{Q_{r}(z_0')}|g|^2\Big)^{\frac{1}{2}},
\end{equation*}
where we set $\kappa =c/8$, and this is the desired estimate $\eqref{f:4.10-new*}$.

%
%

\textbf{Step 4.} Outline the key ingredients for the estimate $\eqref{pri:1.1-xx}$. First, for this purpose, one needs to choose $\tilde{\varphi}_j = \partial_jU_0$ in $\eqref{expansion2}$.
In fact, we only need to make appropriate modifications to the first three steps to show a part of $\eqref{pri:1.1-xx}$, i.e., there holds
\begin{equation}\label{f:4.21-new*}
\begin{aligned}
\big\langle\big|\nabla_y\Gamma_\varepsilon(z_0;z_0')
&-(e_j+\nabla\phi_{j}^{*})^{\varepsilon}(y_0,s_0)\partial_{y_j}\Gamma_0(z_0;z_0')\big|^p\big\rangle^{\frac{1}{p}}\\
&\qquad\quad\lesssim
\frac{\varepsilon\mu_d(\sqrt{t_0-s_0}/\varepsilon)\ln(2+\sqrt{t_0-s_0}/\varepsilon)}{(t_0-s_0)^{\frac{d+2}{2}}}
\exp\Big\{-\frac{\kappa|x_0-y_0|^2}{t_0-s_0}\Big\},
\end{aligned}
\end{equation}
where we recall that $\varepsilon<r:=(t_0-s_0)^{1/2}/100$.
The estimate on the first term in the left-hand side of $\eqref{pri:1.1-xx}$ can be
derived by the same token. To see the estimate $\eqref{f:4.21-new*}$,
a triangle inequality leads to
\begin{equation*}
\begin{aligned}
\big\langle\big|\nabla_y\Gamma_\varepsilon(z_0;z_0')
&-(e_j+\nabla\phi_{j}^{*})^{\varepsilon}(y_0,s_0)\partial_{y_j}\Gamma_0(z_0;z_0')\big|^p\big\rangle^{\frac{1}{p}}\\
&\lesssim  \big\langle\big|\nabla_y W_\varepsilon(z_0')\big|^p\big\rangle^{\frac{1}{p}}
+ \frac{\varepsilon\mu_d(\sqrt{t_0-s_0}/\varepsilon)}{(t_0-s_0)^{\frac{d+2}{2}}}
\exp\Big\{-\frac{\kappa|x_0-y_0|^2}{t_0-s_0}\Big\},
\end{aligned}
\end{equation*}
where $W_\varepsilon$ is defined in $\eqref{expansion2}$. Therefore, it is reduced to showing
\begin{equation}\label{f:4.22-new*}
 \big\langle\big|\nabla_y W_\varepsilon(z_0')\big|^p\big\rangle^{\frac{1}{p}}
 \lesssim \frac{\varepsilon\mu_d(\sqrt{t_0-s_0}/\varepsilon)\ln(2+\sqrt{t_0-s_0}/\varepsilon)}{(t_0-s_0)^{\frac{d+2}{2}}}
\exp\Big\{-\frac{\kappa|x_0-y_0|^2}{t_0-s_0}\Big\}.
\end{equation}

Let $V_\varepsilon$ and $H_\varepsilon$ be given as in
the equations $\eqref{pde:4.2-new*}$ and $\eqref{pde:4.3-new*}$, respectively. On the one hand, by modifying the estimate $\eqref{f:4.51-new*}$, we claim that there holds
\begin{equation}\label{f:4.23-new*}
\|\nabla V_\varepsilon(z_0')\|_p
\lesssim \frac{\varepsilon\mu_d(\sqrt{t_0-s_0}/\varepsilon)\ln(2+\sqrt{t_0-s_0}/\varepsilon)}{(t_0-s_0)^{\frac{d+2}{2}}}
\exp\Big\{-\frac{\kappa|x_0-y_0|^2}{t_0-s_0}\Big\}.
\end{equation}
On the other hand, it follows from Proposition $\ref{prop:2.2-new}$ and Lemma $\ref{lemma:4-xx}$ that
\begin{equation*}
\begin{aligned}
\|\nabla H_\varepsilon(z_0')\|_p
&\lesssim^{\eqref{pri:2.10}}\bigg\langle\Big(\dashint_{Q_{r/2}(z_0')}|\nabla H_\varepsilon|^2\Big)^{\frac{p_1}{2}}\bigg\rangle^{\frac{1}{p_1}}
\lesssim^{\eqref{pri:27-xx}}
\sup_{g\in C^{\infty}_0(Q_r(z_0'))}\frac{\big\langle|\dashint_{Q_r(z_0')}H_\varepsilon
  g|^{p_1}
\big\rangle^{\frac{1}{p_1}}}{r^3\sup|\nabla^2 g|}\\
&\lesssim
\sup_{g\in C^{\infty}_0(Q_r(z_0'))}\frac{\big\langle|\dashint_{Q_r(z_0')}V_\varepsilon
  g|^{p_1}
\big\rangle^{\frac{1}{p_1}}}{r^3\sup|\nabla^2 g|}
+
\sup_{g\in C^{\infty}_0(Q_r(z_0'))}\frac{\big\langle|\dashint_{Q_r(z_0')}W_\varepsilon
  g|^{p_1}
\big\rangle^{\frac{1}{p_1}}}{r^3\sup|\nabla^2 g|} \\
&\lesssim^{\eqref{f:4.16-new*},\eqref{f:4.15-new*}} \frac{\varepsilon\mu_d(\sqrt{t_0-s_0}/\varepsilon)}{(t_0-s_0)^{\frac{d+2}{2}}}\exp\Big\{-\frac{\kappa|x_0-y_0|^2}{t_0-s_0}\Big\},
\end{aligned}
\end{equation*}
where we mention that, in the process of demonstrating the estimate $\eqref{f:4.15-new*}$, replacing $S_{\varepsilon}(\partial_j U_0)$ with $\partial_j U_0$ therein does not change the conclusion. If we temporarily accept the validity of the estimate $\eqref{f:4.23-new*}$, this, together with the estimate $\eqref{f:4.23-new*}$, implies the required estimate $\eqref{f:4.22-new*}$, where we mention that $W_\varepsilon = V_\varepsilon + H_\varepsilon$ on
$Q_r(z_0')$. In this regard, we have completed the proof of $\eqref{pri:1.1-xx}$.

\textbf{Step 5.} Arguments for $\eqref{f:4.23-new*}$.
Let $n_0\in\mathbb{Z}$ be such that $2^{n_0-1}<r\leq 2^{n_0}$.
Let $V_{\varepsilon}^{(n)}$ be the solutions to the following equations:
\begin{equation}\label{pde:4.2-new*-n}
\left\{\begin{aligned}
\big(-\partial_s+\mathcal{L}_\varepsilon^{*}\big)V_\varepsilon^{(n)}
&=\nabla\cdot \hat{f}_\varepsilon\mathbf{1}_{Q_{2^n}(z_0')\setminus Q_{2^{n-1}}(z_0')}
&\quad&\text{in}\quad Q_{r}(z_0');\\
V_\varepsilon^{(n)}  &= 0	
&\quad&\text{on}\quad\partial_p^{*}Q_{r}(z_0'),
\end{aligned}\right.
\quad \forall n\in\mathbb{Z}, ~n\leq n_0.
\end{equation}
Here, since we have changed the definition of $W_\varepsilon$ in $\eqref{expansion2}$, we have replaced $\tilde{f}$ with $\hat{f}$, and the new source term $\hat{f}_\varepsilon$ is given by
\begin{equation}\label{eq:5.100}
\begin{aligned}
\hat{f}_\varepsilon
&:=  \varepsilon\big[a^{*\varepsilon}\bar{\phi}_j^{*\varepsilon}-\bar{\sigma}_j^{*\varepsilon}\big]\nabla\tilde{\varphi}_j -\varepsilon a^{*\varepsilon}(\partial\sigma_{(d+1)j}^{*})^{\varepsilon}
\nabla\tilde{\varphi}_j+\varepsilon\big[a^{*\varepsilon}(\partial\sigma_{(d+1)j}^{*})^{\varepsilon}
\nabla\tilde{\varphi}_j\big](z_0')  \\
&- \varepsilon^2\big[a^{*\varepsilon}\otimes\bar{\sigma}_{(d+1)j}^{*\varepsilon}:\partial^2\tilde{\varphi}_j
+\bar{\sigma}_{(d+1)j}^{*\varepsilon}\partial_s\tilde{\varphi}_j\big]
+\varepsilon e_j\big(\partial_i\sigma_{i(d+1)\cdot}^{*}\big)^{\varepsilon}(z_0')\big(\partial_tU_0-
\partial_t U_0(z_0')\big),
\end{aligned}
\end{equation}
in which we recall that $\varphi_j = \partial_j U_0$ with $j=1,\cdots,d$.
For each $n\leq n_0$ and $n\in\mathbb{Z}$, there holds
\begin{equation}\label{f:4.24-new*}
\begin{aligned}
&\|\nabla V_\varepsilon^{(n)}(z_0')\|_{p}
\lesssim^{\eqref{pri:2.10}}
\bigg\langle
\Big(\dashint_{Q_{2^{n-1}}(z_0')}|\nabla V_\varepsilon^{(n)}|^2\Big)^{\frac{p_1}{2}}\bigg\rangle^{\frac{1}{p_1}}
\lesssim
\bigg\langle
\Big(\dashint_{Q_{2^n}(z_0')}|\hat{f}_{\varepsilon}|^2\Big)^{\frac{p_1}{2}}\bigg\rangle^{\frac{1}{p_1}}\\
&\lesssim^{\eqref{pri:5.2-n},\eqref{pri:5.4-n}}_{\eqref{pri:5.4b-n}}
\varepsilon\min\big\{\mu_d(2^n/\varepsilon),1\big\}\mu_d(r/\varepsilon)\\
&\qquad\qquad \times\bigg\{
\sup_{z'\in Q_{2r}(z_0')}\big|(\nabla^2_yU_0,\partial_sU_0)(z')\big|+\varepsilon^{\alpha}\big[(\nabla_y^2,\partial_s )U_0\big]_{C^{\alpha,\alpha/2}(Q_{2r}(z_0'))}\bigg\}\\
&\qquad\qquad
\qquad\qquad\quad
\qquad\qquad
\lesssim^{\eqref{f:4.50-new*}}
\varepsilon\min\big\{\mu_d(2^n/\varepsilon),1\big\}
\frac{\mu_d(r/\varepsilon)}{(t_0-s_0)^{\frac{d+2}{2}}}\exp\Big\{-\frac{\kappa |x_0-y_0|^2}{t_0-s_0}\Big\},
\end{aligned}
\end{equation}
where the second inequality follows from the energy estimate, and we also employ the precondition $r\geq 100\varepsilon$ for the last inequality.
This further leads to
\begin{equation*}
\|\nabla V_\varepsilon(z_0')\|_{p}
\lesssim \sum_{n\leq n_0}
\|\nabla V_\varepsilon^{(n)}(z_0')\|_{p}
\lesssim^{\eqref{f:4.24-new*}} \frac{\varepsilon\mu_d(r/\varepsilon)\ln(2+r/\varepsilon)}{r^{d+2}}\exp\Big\{-\frac{\kappa |x_0-y_0|^2}{r^2}\Big\},
\end{equation*}
which
consequently provides the desired estimate $\eqref{f:4.23-new*}$. We have completed the whole proof.

\subsection{Proof of Theorem $\ref{thm:3-new}$.}

\noindent
The proof can be independently demonstrated using Lemma $\ref{lemma:4-xx}$.  However, we have already established the estimates for the first-order derivative expansion of the fundamental solution in Theorem $\ref{thm:1-new}$. The advantage of this proof is to fully utilize the existing argument and results in order to simplify the argument. Let $1\leq p<p_1<p_2<\infty$.
Throughout the proof, we fix two ``base points'' $z_0:=(x_0,t_0),z_0':=(y_0,s_0)\in\mathbb{R}^{d+1}$ with $s_0<t_0$.
 It is fine to assume that $\varepsilon<r=(t_0-s_0)^{1/2}/100$, and the whole proof will be split into two steps.

\textbf{Step 1.} Outline the key ingredients for the estimate $\eqref{pri:1.1-new}$. Let $U_{\varepsilon,k}(y,s):=\partial_{x_k}\Gamma_\varepsilon(z_0;y,s)$,
$U_{0,k}(y,s):=(e_i+(\partial_{x_{k}}\phi_i)^{\varepsilon}(z_0))\partial_{x_i}\Gamma_0(z_0;y,s)$, and $z':=(y,s)\in\mathbb{R}^{d+1}$. Consider
\begin{equation}\label{expansion3}
W_{\varepsilon,k}:=U_{\varepsilon,k}-U_{0,k}-\varepsilon\bar{\phi}_{j}^{*\varepsilon}\tilde{\varphi}_{j,k}+\varepsilon^2\bar{\sigma}_{l(d+1)j}^{*\varepsilon}\partial_l\tilde{\varphi}_{j,k},
\quad \text{with}\quad \tilde{\varphi}_{j,k}:=\partial_jU_{0,k}
\text{~and~}k\in\{1,\cdots,d\}.
\end{equation}
where $\bar{\phi}_j^{*\varepsilon}(z'):=\phi_j^{*}(z'/\varepsilon)-\phi_j^{*}(z_0'/\varepsilon)$ with
$\bar{\sigma}_{l(d+1)j}^{*\varepsilon}(z'):=\sigma_{l(d+1)j}^{*}(z'/\varepsilon)-\sigma_{l(d+1)j}^{*}(z_0'/\varepsilon)$. Let
$\phi_j^{*}$ and $\sigma_{k'}^{*}$ (with $k'=1,\cdots,d+1$) be the corresponding correctors associated with the adjoint operator $-\partial_t-\nabla\cdot a^{*}\nabla$. Then,
$W_{\varepsilon,k}$ satisfies the following equation:
\begin{equation}
\begin{aligned}
\big(-\partial_s+\mathcal{L}_\varepsilon^{*}\big)W_{\varepsilon,k}
&=\nabla\cdot \hat{f}_{\varepsilon,k}
&\quad&\text{in}\quad Q_{2r}(z_0');
\end{aligned}
\end{equation}
where $\hat{f}_{\varepsilon,k}$ is given by $\eqref{eq:5.100}$
through replacing $\varphi_j$ with $\varphi_{j,k}$ correspondingly. For any $m,k\in\{1,\cdots,d\}$,
extract the component $\mathcal{E}_{mk}$ of $\mathcal{E}$ defined in $\eqref{eq:6.1-x}$. Then, for any $z'\in Q_{r}(z_0')$,
we have the following identity:
\begin{equation}
	\mathcal{E}_{m,k}(z_0,z') =\partial_{m} W_{\varepsilon,k}(z')+\varepsilon
	\big[e_{l}^{m}\bar{\phi}_{j}^{*\varepsilon}(z')
	-(\partial_{m}\sigma_{l(d+1)j}^{*})^{\varepsilon}(z')\big]
	\partial_{lj}^2U_{0,k}(z')-
	\varepsilon^2(\bar{\sigma}_{l(d+1)j}^{*})^{\varepsilon}(z')\partial_{mlj}^{3}U_{0,k}(z'),
\end{equation}
where $e_{l}^m = 1$ if $l=m$, and $e_{l}^m = 0$ if $l\not= m$. We introduce the following notations:
$\partial_m:=\partial_{y_m}$, $\partial_{lj}^2:=\partial_{y_l}\partial_{y_j}$, and
$\partial_{mlj}^3:=\partial_{y_m}\partial_{y_l}\partial_{y_j}$.
For any $p\in [2,\infty)$,  this implies
\begin{equation}\label{f:6.2-new}
\|\mathcal{E}_{m,k}(z_0,z_0')\|_p
\lesssim^{\eqref{pri:5.2-n},\eqref{f:4.50-new*}} \|\nabla W_{\varepsilon,k}(z_0')\|_p
+ \frac{\varepsilon}{(t_0-s_0)^{\frac{d+3}{2}}}\exp\Big\{-\frac{\kappa |x_0-y_0|^2}{t_0-s_0}\Big\}.
\end{equation}
It is reduced to showing
\begin{equation}\label{f:6.1-new}
\begin{aligned}
\|\nabla W_{\varepsilon,k}(z_0')\|_p
\lesssim \frac{\varepsilon\mu_d(\sqrt{t_0-s_0}/\varepsilon)\ln(2+\sqrt{t_0-s_0}/\varepsilon)}{(t_0-s_0)^{\frac{d+3}{2}}}
\exp\Big\{-\frac{\kappa|x_0-y_0|^2}{t_0-s_0}\Big\}.
\end{aligned}
\end{equation}
Acknowledging the above estimate for a while, plugging
the estimate $\eqref{f:6.1-new}$ back into $\eqref{f:6.2-new}$
leads to the desired estimate $\eqref{pri:1.1-new}$.

\textbf{Step 2.} Arguments for $\eqref{f:6.1-new}$. Let
$V_{\varepsilon,k}^{(n)}$ be the solutions to the equations
$\eqref{pde:4.2-new*-n}$ with the substitution of
$\hat{f}_{\varepsilon,k}$ for $\hat{f}_\varepsilon$ therein. Then, define $H_{\varepsilon,k}:=W_{\varepsilon,k}-\sum_{n\leq n_0}V_{\varepsilon,k}^{(n)}$, and it satisfies $(-\partial_s+\mathcal{L}_\varepsilon^{*})H_{\varepsilon,k} = 0$ in $Q_r(z_0')$. It follows from Proposition $\ref{prop:2.2-new}$ that
\begin{equation}\label{f:6.3-new}
\begin{aligned}
\|\nabla H_{\varepsilon,k}(z_0')\|_{p}
&\lesssim^{\eqref{pri:2.10}}
\bigg\langle\Big(\dashint_{Q_{r/2}(z_0')}
|\nabla H_{\varepsilon,k}|^2\Big)^{\frac{p_1}{2}}\bigg\rangle^{\frac{1}{p_1}}
\lesssim \frac{1}{r}\bigg\langle\Big(\dashint_{Q_r(z_0')}
|H_{\varepsilon,k}|^2\Big)^{\frac{p_1}{2}}\bigg\rangle^{\frac{1}{p_1}} \\
&\lesssim
\frac{1}{r}\bigg\langle\Big(\dashint_{Q_r(z_0')}
|W_{\varepsilon,k}|^2\Big)^{\frac{p_1}{2}}\bigg\rangle^{\frac{1}{p_1}}
+
\sum_{n\leq n_0}\bigg\langle\Big(\dashint_{Q_r(z_0')}
|\nabla V_{\varepsilon,k}^{(n)}|^2\Big)^{\frac{p_1}{2}}\bigg\rangle^{\frac{1}{p_1}},
\end{aligned}
\end{equation}
where we employ Caccioppoli's inequality in the second inequality and Poincar\'e's inequality in the last one.
In view of the equations $\eqref{pde:4.2-new*-n}$, we have
\begin{equation}\label{f:6.4-new}
\|\nabla V_{\varepsilon,k}^{(n)}(z_0')\|_p
\lesssim^{\eqref{pri:2.10}} \bigg\langle\Big(\dashint_{Q_{2^{n-1}}(z_0')}|\nabla V_{\varepsilon,k}^{(n)}|^2\Big)^{\frac{p_1}{2}}\bigg\rangle^{\frac{1}{p_1}}
\lesssim \bigg\langle\Big(\dashint_{Q_{2^n}(z_0')}|\hat{f}_{\varepsilon,k}|^2\Big)^{\frac{p_1}{2}}\bigg\rangle^{\frac{1}{p_1}},
\end{equation}
where we employ the energy estimate in the second inequality.

Combining the estimates $\eqref{f:6.3-new}$ and $\eqref{f:6.4-new}$ leads to
\begin{equation}\label{f:6.5-new}
\begin{aligned}
\|\nabla W_{\varepsilon,k}(z_0')\|_{p}
&\leq \|\nabla H_{\varepsilon,k}(z_0')\|_{p}
+\sum_{n\leq n_0}\|\nabla V_{\varepsilon,k}^{(n)}(z_0')\|_{p}\\
&\lesssim \frac{1}{r}\bigg\langle\Big(\dashint_{Q_r(z_0')}
|W_{\varepsilon,k}|^2\Big)^{\frac{p_1}{2}}\bigg\rangle^{\frac{1}{p_1}}
+\sum_{n\leq n_0}
\Big(\big(\frac{2^{n}}{r}\big)^{d+2}+1\Big)
\bigg\langle\Big(\dashint_{Q_{2^n}(z_0')}|\hat{f}_{\varepsilon,k}|^2\Big)^{\frac{p_1}{2}}\bigg\rangle^{\frac{1}{p_1}}\\
&\lesssim \frac{1}{r}\sup_{z'\in Q_r(z_0')}\|W_{\varepsilon,k}(z')\|_{p_1}
+ \sum_{n\leq n_0}
\bigg\langle\Big(\dashint_{Q_{2^n}(z_0')}|\hat{f}_{\varepsilon,k}|^2\Big)^{\frac{p_1}{2}}\bigg\rangle^{\frac{1}{p_1}},
\end{aligned}
\end{equation}
where one can take $p_1\geq 2$ for the second inequality.

On the one hand, it follows from Theorem $\ref{thm:1-new}$ that,
for any $z'\in Q_r(z_0')$,
\begin{equation}\label{f:6.6-new}
\begin{aligned}
\|W_{\varepsilon,k}(z')\|_{p_1}
&\lesssim^{\eqref{pri:5.2-n},\eqref{c:4}}  \|(U_{\varepsilon,k}-U_{0,k})(z')\|_{p_1}
+ \varepsilon\|\nabla U_{0,k}(z')\|_{p_2}+\varepsilon^2
\mu_d(r/\varepsilon)\|\nabla^2 U_{0,k}(z')\|_{p_2} \\
&\lesssim^{\eqref{pri:1.1-xx},\eqref{pri:5.4-n}}
\bigg\{\frac{\varepsilon\mu_d(r/\varepsilon)\ln(2+r/\varepsilon)}{r^{d+2}}+\frac{\varepsilon}{r^{d+2}}+\frac{\varepsilon^2\mu_d(r/\varepsilon)}{r^{d+3}}\bigg\}
\exp\Big\{-\frac{\kappa |x_0-y_0|^2}{r^2}\Big\}\\
&\lesssim \frac{\varepsilon\mu_d(r/\varepsilon)\ln(2+r/\varepsilon)}{r^{d+2}}\exp\Big\{-\frac{\kappa|x_0-y_0|^2}{r^2}\Big\},
\end{aligned}
\end{equation}
where we also employ the decay estimates of $\Gamma_0$ in the second inequality.
On the other hand, a similar computation as shown for $\eqref{f:4.24-new*}$ leads to
\begin{equation}\label{f:6.7-new}
\begin{aligned}
\sum_{n\leq n_0}
\bigg\langle\Big(\dashint_{Q_{2^n}(z_0')}|f_{\varepsilon,k}|^2\Big)^{\frac{p_1}{2}}\bigg\rangle^{\frac{1}{p_1}}
\lesssim \frac{\varepsilon\mu_d(r/\varepsilon)\ln(2+r/\varepsilon)}{r^{d+3}}\exp\Big\{-\frac{\kappa |x_0-y_0|^2}{r^2}\Big\},
\end{aligned}
\end{equation}
Since $\varphi_{j,k}$ in $\hat{f}_{\varepsilon,k}$ has one more derivative than $\varphi_j$
in $\tilde{f}_\varepsilon$ shown in $\eqref{f:4.24-new*}$,
the above result decays one order higher than  that shown in
$\eqref{f:4.24-new*}$.
Consequently, plugging the estimates $\eqref{f:6.6-new}$ and
$\eqref{f:6.7-new}$ back into $\eqref{f:6.5-new}$ gives
the desired estimate $\eqref{f:6.1-new}$. This completes the whole proof.
\qed

\paragraph{Acknowledgements.}
The second author was supported by the Young Scientists Fund of the National Natural Science Foundation of China (Grant No. 12501280); the third author was supported by the Natural Science Foundation of Gansu Province, China (Grant No. 26JRRA174). This project was partially supported by the Tianyuan Fund for Mathematics of the National Natural Science Foundation of China (Grant No. 12326102).

\paragraph{Statements and Declarations.} 
The authors have no competing interests to declare that are relevant to the content of this article.

\paragraph{On the AI usage.}
This study began in April 2026. The research problem and the overall proof framework were developed by the authors before AI assistance was introduced. During revision, the authors used ChatGPT 5.6-6.1 (OpenAI) to assist with language editing, manuscript organization, and the examination of selected mathematical arguments, including identifying potential errors and suggesting corrections. Specifically, this is reflected in the following mathematics:
\begin{itemize}
\item Revising the second step in the proof of Proposition $\ref{prop:2.1}$.  ChatGPT suggested using $R'$ for classification discussion, which simplified and supplemented the situations discussed in the earlier version.
\item The idea of the proof of the interpolation inequality $\eqref{f:5.2-new}$ is due to ChatGPT, which is a new argument different from that shown in \cite{Bella-Giunti-Otto-17}. 
\item The concrete form $\psi(x)=1-\sqrt{1+\gamma^2|x-x_0|^2}$ appearing in 
Step 3 of the proof of Theorem $\ref{thm:1-new}$ is attributed to ChatGPT, which 
is required to satisfy the conditions $\eqref{f:2.1-x-new*}$ shown in Lemma $\ref{lemma:4.1-new}$.   
\end{itemize}

The authors evaluated these suggestions and independently verified the mathematical arguments and references retained in the final manuscript. All decisions concerning the results, proofs, and final presentation were made by the authors. The authors take responsibility for all content.

\addcontentsline{toc}{section}{References}


\end{document}